\documentclass[10pt]{article}
\usepackage[utf8]{inputenc}
\usepackage{geometry}
\usepackage{bbm}
\usepackage{amsmath,amssymb,amsthm}
\theoremstyle{definition}
\newtheorem{theorem}{Theorem}
\newtheorem{corollary}[theorem]{Corollary}
\newtheorem{lemma}[theorem]{Lemma}
\newtheorem{definition}{Definition}
\newtheorem{remark}{Remark}

\newtheorem{assumption}{Assumption}
\usepackage{graphicx}
\usepackage{subfig}
\usepackage{float}
\usepackage[]{algorithm2e}
\usepackage{listings}
\usepackage[toc,page]{appendix}
\usepackage[table,xcdraw]{xcolor}
\newcommand{\R}{\mathbb{R}}
\newcommand{\E}{\mathbb{E}}
\newcommand{\lb}{\langle}
\newcommand{\rb}{\rangle}
\newcommand{\argmin}{\text{argmin}}
\usepackage{hyperref}

\title{Conditional Density Estimation, Latent Variable Discovery and Optimal Transport}
\date{July 1, 2020}
\author{Hongkang Yang, Esteban G. Tabak}

\begin{document}
\maketitle

\begin{abstract}
A framework is proposed that addresses both conditional density estimation and latent variable discovery.
The objective function maximizes explanation of variability in the data,  achieved through the optimal transport barycenter generalized to a collection of conditional distributions indexed by a covariate---either given or latent---in any suitable space. Theoretical results establish the existence of barycenters, a minimax formulation of optimal transport maps, and a general characterization of variability via the optimal transport cost. This framework leads to a family of non-parametric neural network-based algorithms, the BaryNet, with a supervised version that estimates conditional distributions and an unsupervised version that assigns latent variables. The efficacy of BaryNets is demonstrated by tests on both artificial and real-world data sets. A parallel drawn between autoencoders and the barycenter framework leads to the Barycentric autoencoder algorithm (BAE).
\end{abstract}


\medskip


\section{Introduction}

In machine learning, one often considers joint distributions of the form $\rho(x,z)$, where $x$ is an observable and $z$ some latent variable, or alternatively $z$ is the source and $x$ the target variable. For instance, in images of human faces, the data space $x\in X$ may have a dimension up to $10^5 \sim 10^6$ if counted in pixels, while a covariate $z \in Z$ consisting of the face orientation is only two-dimensional.

A task of broad applicability is to extract, given data drawn from $\rho(x, z)$, the conditional distributions $\rho(x|z)$. An example in medical studies has $\rho(x|z)$ representing the distribution of blood sugar level conditioned on a patient's age and diet. In generative modeling, $\rho(x|z)$ can represent the distribution of images conditioned on a text description such as $z$=``cat", and one seeks to generate samples from $\rho(x|z)$. A knowledge of $\rho(x|z)$ allows one to estimate the conditional expectation $\E_{\rho(x|z)}[f(x)]$ for any function $f$ of interest.

Alternatively, if one is only given the raw data $\rho(x)$, then one can try to infer a reasonable latent variable $z$ that underlies $x$. For instance, for the facial images, discovering $z$ as the face orientation explains away a great portion of the data's variability. Whether or not this latent variable has a clear interpretation, it can potentially facilitate data compression and generative modeling.

These two problems are known, respectively, as conditional density estimation and latent variable discovery.
The former can be seen as a probabilistic generalization of classification and regression, while the later contains as special cases clustering and dimensionality reduction.
They form a pair of supervised/unsupervised problems, such that one learns the dependency of $x$ on $z$, while the other discovers $z$. This paper formulates and solves both problems in a single framework based on optimal transport.

\subsection{Related work}
Existing methods for conditional density estimation generally follow one of two approaches: to directly model $\rho(x,z)$ using kernel smoothing techniques \cite{hyndman1996estimating,holmes2012fast}, or to model the mapping
$$z \mapsto \rho(x|z)$$
For instance, the Mixture Density Network \cite{bishop1994mixture} models $\rho(x|z)$ as a Gaussian mixture with $z$-dependent parameters, the Conditional GAN \cite{mirza2014conditional} models $\rho(x|z)$ by generative adversarial networks (GAN), Deep Conditional Generative Models \cite{sohn2015learning} use variational autoencoders (VAE), and \cite{Agnelli-et-al} and \cite{trippe2018conditional} use normalizing flows. Essentially, these methods apply density estimation techniques for a single distribution to the modeling of all conditional distributions simultaneously. We will introduce an alternative approach such that all $\rho(x|z)$ are represented by a single distribution $\mu$, from which we can easily recover each $\rho(x|z)$, so that we only need to estimate $\mu$.

Existing methods for latent variable discovery are vast and rich. For discrete latent variables $z$, the problem reduces to clustering, where popular methods include $k$-means and the EM algorithm \cite{alpaydin2009introduction,bishop2006pattern}. For continuous $z$, we have dimensionality reduction algorithms such as principal component analysis (PCA), principal curves and surfaces \cite{hastie2005elements}, and undercomplete autoencoders (also known as autoassociative neural networks) \cite{alpaydin2009introduction}. Depending on different regularizations on $z$, there are also the VAE \cite{kingma2013auto}, AAE \cite{makhzani2015adversarial}, WAE \cite{tolstikhin2017wasserstein}, and denoising and sparse autoencoders \cite{alpaydin2009introduction}.
We will identify below a parallelism between autoencoders and the algorithms that we propose.

Our theoretical model is based on optimal transport, in particular on the barycenter of probability measures. The idea of applying barycenters to conditional density estimation originates from \cite{tabak2018conditionaldensity}, while the application to latent variable discovery is based on the previous work in \cite{tabak2018explanation,yang2019clustering}. This paper lays the theoretical foundation for the technique of barycenters, and introduces several neural network-based algorithms.

\subsection{Sketch of our approach}

Intuitively, one of the principles of learning is to reduce uncertainty. Given arbitrary data, an effective way to learn it is to find a representation of it so that some measure of uncertainty is reduced. One prototypical example is the Kolmogorov complexity (or descriptive complexity) \cite{sipser2013introduction}: when the data consists of a string such as $ababababab$, it is natural to represent it by $ab\times 5$, so the variability of a long string is reduced to that of a shorter representative.
Another instance is PCA, which seeks a low-dimensional representation of high-dimensional data. It maximizes the amount of variance explained by the principal components, thus minimizing the uncertainty remaining.

Clustering provides a similar setting: suppose that we are given the data displayed on the left image of Figure \ref{fig: three clusters}, divided into three labeled clusters. We would naturally \emph{learn} the data by memorizing the clusters' common shape and their relative positions. Equivalently, as in the right image, the data can be represented by a common distribution plus the translations that bring the three clusters to it. As we apply these translations to transform the original data, the variance is greatly reduced.
\begin{figure}[H]
\centering
\subfloat{\includegraphics[scale=0.75]{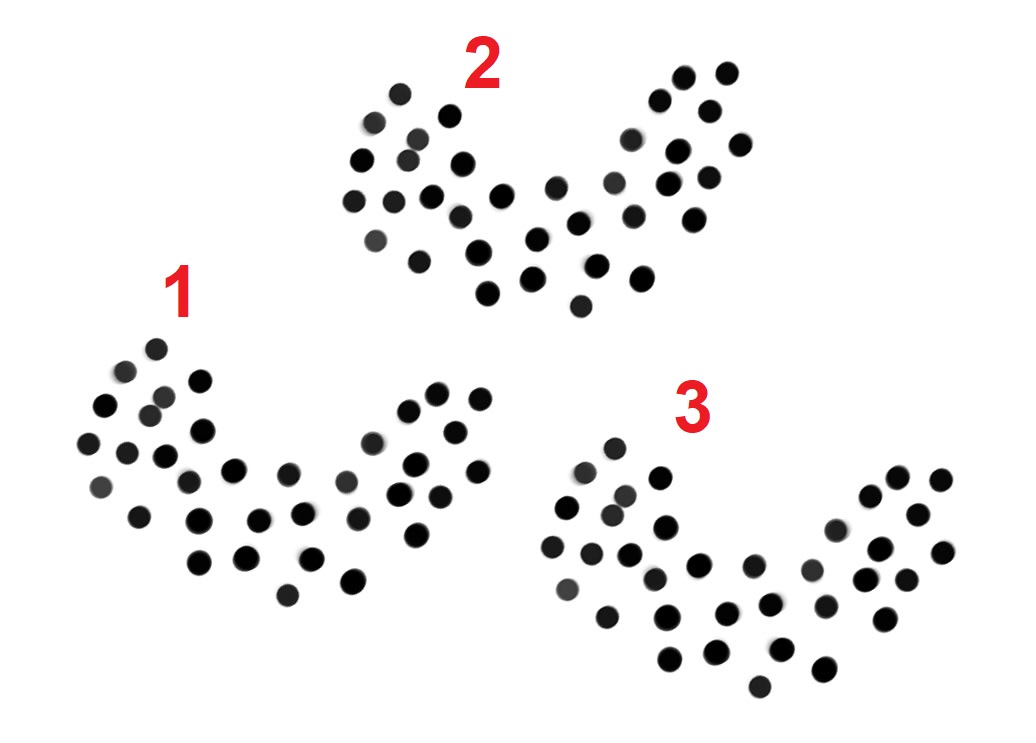}}
\subfloat{\includegraphics[scale=0.75]{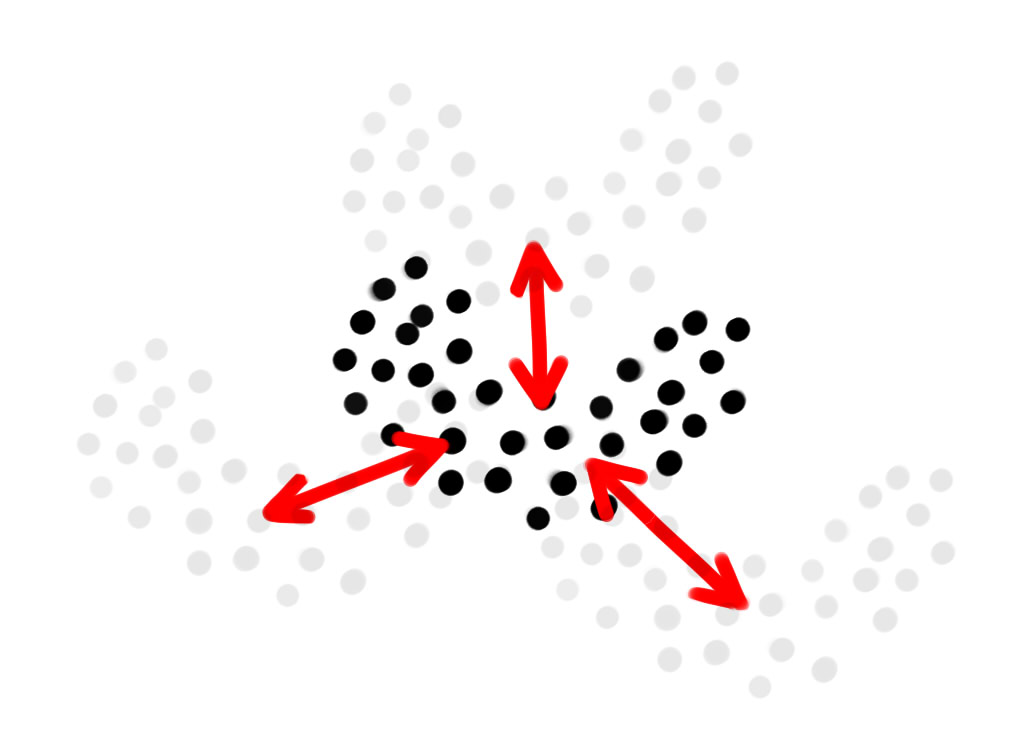}}
\caption{Clusters and their representative}
\label{fig: three clusters}
\end{figure}

This intuition can be summarized as follows: given a sample space $X$ such as $\R^2$ and a latent variable space $Z$ such as $\{1,2,3\}$, an effective way to learn a distribution $\rho(x,z) \in P(X\times Z)$ is to find a representative distribution $\mu$ with smaller variability, as well as the transformations between each conditional distribution $\rho(x|z)$ and $\mu$.

In a data-based scenario, we are given a labeled sample $\{x_i,z_i\}_{i=1}^N$. Once we obtain the transformations $T_z$ that send $\rho(x|z)$ to $\mu$ and their inverse transformations $S_z$, the representative $\mu$ can be estimated by aggregating all sample points: $\{T_{z_i}(x_i)\}$. Then, given any $z$, the conditional distribution $\rho(x|z)$ can be sampled via $\{S_z \circ T_{z_i}(x_i)\}$.

We can already see one advantage of this procedure: once the transformations are known, $N$ sample points for $\rho(x, z)$ automatically provide $N$ samples for each conditional $\rho(x|z)$. This is particularly helpful when there are many latent variables $z$ or when these are continuous, so that for most values of $z$, the conditional $\rho(x|z)$ has very few or zero samples in $\{x_i,z_i\}_{i=1}^N$. Furthermore, this procedure can be used in conjunction with other density estimation algorithms or generative models: the difficult task of modeling many $\rho(x|z)$ (or a high-dimensional $\rho(x,z)$) is simplified into modeling a single $\mu$, and one can, for instance,  first train the GAN or VAE on $\mu$ and then concatenate it with $S_z$ to model each $\rho(x|z)$.\\

The theory of optimal transportation is ideal for the formalization of the procedure above. Intuitively, the representative distribution $\mu$ should closely resemble each conditional $\rho(x|z)$, that is, $\mu$ is the minimizer of some average ``distance" between $\mu$ and $\rho(x,z) = \rho(x|z)\nu(z)$:
$$\mu = \argmin_{\tilde{\mu}} \int_Z \text{distance}\big(\tilde{\mu},\rho(x|z)\big) d\nu(z)$$
Thus, we refer to $\mu$ as the ``barycenter" of $\rho(x,z)$. The optimal transport cost (or ``Earth mover's distance") is a good candidate for distance, such that informally the distance between two distributions $\mu_1$ and $\mu_2$ is the minimum ``work" required to transport $\mu_1$ (thought of as a pile of sand) to $\mu_2$. When a cost function $c(x,y)$ is given (such as the Euclidean distance $\|x-y\|$), the optimal transport cost is
$$I_c(\mu_1,\mu_2) \approx \inf_{T} \int c(x,T(z)) d\mu_1,$$
where the infimum is taken over all maps $T$ that transport $\mu_1$ to $\mu_2$. If we consider $c(x,T(x))$ as a measure of pointwise distortion, then $I_c(\rho(x|z),\mu)$ is the distortion or information loss incurred on the original data.

Optimal transportation has several advantages. The optimal transport cost $I_c(\mu_1,\mu_2)$ depends on a user-specified cost $c$ and thus can directly incorporate task-specific information. In particular, if the cost is based on the metric of the space, then $I_c(\mu_1,\mu_2)$ reliably captures our intuitive sense of distance between distributions \cite{arjovsky2017wasserstein}, whereas other measures, such as the Kullback-Leibler divergence and total variation, fail when the distributions have disjoint supports. Also, given that optimal transport minimizes data distortion, it is natural to expect that $\mu_1$ can be easily recovered from its transported image $\mu_2$, that is, the transport maps $T_z$ are invertible. This is a useful property, since our procedure needs to transform back and forth between the conditionals $\rho(x|z)$ and the representative $\mu$.

The greatest advantage, however, is duality. The theory of optimal transport abounds with duality techniques, through which we convert optimization problems over probability measures to problems over functions, and vice versa. A general rule is that, being less restricted, functions are easier to model and optimize than probability distributions, and we perform this conversion whenever possible. The primal problem of conditional density estimation is often intractable, because it is difficult to model directly each of the possibly infinitely many conditionals $\rho(x|z)$. Yet, optimal transport duality converts the primal into an optimization over one transport map $T(x,z)$ and one discriminator $\psi(y,z)$, which can be more easily solved by methods such as neural networks.\\

The next step is to apply our principle of minimum uncertainty to latent factor discovery, the unsupervised counterpart to conditional density estimation. Recall that in the supervised setting with $\rho(x, z)$, our procedure reduces the uncertainty or ``variability" of $\rho(x)$ to the smaller variability of the representative (or barycenter) $\mu$. If one has the freedom to determine the labels $z$, then the variability of $\mu$ can be further reduced. Specifically, what matters is the choice of $\rho(x|z)$, whereas the labels $z$ by themselves are equivalent under permutations and we can arrange them into some prior distribution $v(z)$.

Thus, given an unlabeled data $\rho(x)$, factor discovery should seek a labeling $\rho(x,z)$ that minimizes the variability of its barycenter $\mu$, or equivalently,
$$\max_{\rho(x,z)} \text{Variability}(\rho(x)) - \text{Variability}(\mu).$$

If, for the dataset in Figure \ref{fig: three clusters}, one did not know the labels $\{1,2,3\}$, one could assign them. Clearly some labelings are better than others: in the worst scenario, the labels $\{1,2,3\}$ would be distributed uniformly within each cluster, and our procedure would yield a barycenter $\mu$ with the same shape and size as the original data, with no variability reduction at all.

We should define ``variability" in a way that generalizes variance, so that factor discovery can yield the obvious labeling of Figure \ref{fig: three clusters}. Also, variability should depend on the cost $c(x,y)$ in order to incorporate task-specific information. Intuitively, how much we learn is proportional to how much effort we spend learning, or equivalently,
\begin{equation*}
\text{Reduced uncertainty} = \text{Work}.
\end{equation*}
So we characterize ``variability" as a measurement that satisfies
\begin{equation}
\label{variability definition}
\text{Variability}(\rho(x)) - \text{Variability}(\mu) = \text{Total transport cost } \int I_c\big(\rho(x|z),\mu\big) d\nu(z).
\end{equation}
In fact, Corollary \ref{cor: variability is variance} below shows that definition (\ref{variability definition}) yields exactly the variance when we use the squared Euclidean distance cost $c=\|x-y\|^2$. Hence, factor discovery becomes
$$\max_{\rho(x,z)} \text{Total transport cost},$$
which differs from conditional density estimation only by the additional maximization.\\

This paper is structured as follows. Section \ref{sec: theoretical foundation} develops the ideas presented in the introduction, formulating conditional density estimation and latent factor discovery in the framework of optimal transport barycenters. Section \ref{sec: algorithmic design} addresses the algorithmic aspects, proposing the supervised and unsupervised BaryNet algorithms, which use neural networks. It also discusses BaryNet's relation to existing methods, in particular the autoencoders, and introduces the Barycentric autoencoder (BAE) based on BaryNet. Section \ref{sec: test results} tests the performance of the BaryNet algorithms on real-world and artificial data sets, and verifies that they can reliably solve conditional density estimation and latent factor discovery. Finally, Section \ref{sec: conclusion} summarizes the results and discusses possible future work. The proofs of most theorems are provided in an appendix.

\section{Theoretical foundation}
\label{sec: theoretical foundation}
The ideas presented in the introduction are formalized and proved in this section. We first define optimal transport barycenter and prove its existence. Then, we obtain the conditional transport map $T(x,z)$ from a minimax problem. Finally, we prove the variance decomposition theorem and justify our definitions of variability and latent factor discovery.

\subsection{Preliminaries}
We denote by $X$ and $Z$ the sample and latent variable spaces, and by $Y$ the space that the barycenter $\mu$ belongs to. In practice one often has $X=Y$, but this is not required here.

Most of our results will be presented with $(X,Y, Z)$ Polish spaces, which are complete separable metric spaces. These have enough structure to handle problems of optimal transport, while they are general enough to include most spaces in real-world applications, such as Euclidean spaces $\R^d$, closed subsets of $\R^d$, complete Riemannian manifolds $M^d$, discrete sets such as $\{1,\dots K\}$, and function spaces such as $C([0,1]),P(\R^d)$.

Given a Polish space $X$, we denote the space of continuous functions by $C(X)$, the space of bounded continuous functions by $C_b(X)$ and the space of Borel probability measures by $P(X)$.

For clarity, we sometimes write a measure $\rho \in P(X)$ informally as $\rho(x)$ to indicate the space it belongs to, not implying by this that $\rho$ has a density function, unless explicitly declared. For joint probability measures, e.g. $\pi \in P(X\times Y\times Z)$, we denote its marginals by $\pi_X,\pi_{YZ}$, etc. The tensor product of probability measures $\mu$ and $\nu$ is denoted by $\mu \otimes \nu$.

Given $\rho(x,z) \in P(X\times Z)$, we define the conditional distributions $\rho(x|z)\nu(z) = \rho(x,z)$ using the disintegration theorems \cite{chang1997conditioning}. The conditional $\rho(x|z)$ always exists as a Borel measurable map from $Z$ to $P(X)$ in the topology of weak convergence, and it is unique $\nu(z)$-almost surely. Conversely, given $\nu(z)$ and a measurable $\rho(x|z)$, we define the joint distribution $\rho(x,z):=\rho(x|z)\nu(z)$ by
$$\forall \psi \in C_b(X\times Z), ~\int \psi d\rho(x,z) := \int \psi d\rho(x|z)d\nu(z)$$

\subsection{Optimal transport and barycenter}

A map $T: X\to Y$ pushes-forward $\rho \in P(X)$ to $\mu \in P(Y)$ (denoted $T\#\rho = \mu$) if
$$\mu(A) = \rho(T^{-1}(A))$$
for all measurable subsets $A \subseteq X$. Monge's original formulation of optimal transport  \cite{monge1781memoire}:
\begin{equation*}
\inf_{T\#\rho=\mu} \int_X c(x,T(x)) d\rho(x)
\end{equation*}
minimizes over all transport maps $T$ the expected value of a cost function $c$ on $X\times Y$. Kantorovich \cite{kantorovich1942translocation} generalized the transport maps to probabilistic couplings,
$$\Pi(\rho,\mu) := \{ \pi \in P(X\times Y), ~\pi_X = \rho, \pi_Y = \mu \},$$
relaxing the optimal transport problem to
\begin{equation}
\label{Kantorovich OT}
I_c(\rho,\mu) = \inf_{\pi \in \Pi(\rho,\mu)} \int_{X\times Y} c(x,y) d\pi(x,y).
\end{equation}
If the minimum of (\ref{Kantorovich OT}) is achieved by some coupling $\pi$, we call it an optimal transport plan, or a Kantorovich solution. If $\pi$ is concentrated on the graph of some function $T:X\to Y$, then $T$ is called an optimal transport map, or a Monge solution.\\

Inspired by definitions from \cite{tabak2018conditionaldensity} and \cite{kim2017wasserstein}, we define optimal transport barycenter as follows:

\begin{definition}[Barycenter problem]
Given a cost function $c(x,y)$, a labeled distribution $\rho(x,z) = \rho(x|z)\nu(z) \in P(X\times Z)$, and any $\mu \in P(Y)$, the \textit{total transport cost} between $\rho(x,z)$ and $\mu$ is defined by
\begin{equation}
\label{total transport cost}
I_c(\rho(x,z),\mu) = \int_Z I_c(\rho(x|z),\mu) d\nu(z)
\end{equation}
If the minimum total transport cost
$$\inf_{\mu \in P(Y)} I_c(\rho(x,z),\mu)$$
is achieved by some $\mu$, then we call it the \textit{barycenter} of $\rho(x,z)$.
\end{definition}

The notion of a barycenter of finitely many conditionals $\rho(x|z)$ (that is, with finite label space $Z=\{1,\dots K\}$) was introduced in \cite{carlier2010matching,chiappori2010hedonic,rabin2011wasserstein}, and its existence, uniqueness, and regularity were examined in \cite{agueh2011barycenters,pass2012multi,kim2017wasserstein}. Barycenters of infinitely many conditional distributions are studied in \cite{pass2013optimal,kim2017wasserstein}, which deal with the special case when $X$ is either a Euclidean space or a compact Riemannian manifold and $c$ is the squared distance cost.

We show that the barycenter problem as defined above is well-posed, and the barycenter exists under general conditions:

\begin{theorem}[Well-posedness and Existence of Barycenter]
\label{thm: well-posedness and existence}
Let $X,Y,Z$ be Polish spaces, let $c \in C(X\times Y)$ be a continuous cost that is bounded below ($\inf c > -\infty$), and let $\rho(x,z) = \rho(x|z)\nu(z) \in P(X\times Z)$ be a probability measure. Then,
\begin{enumerate}
\item Given any $\mu \in P(Y)$, the total transport cost (\ref{total transport cost}) is well-defined, and
\begin{equation}
\label{joint measure formulation of total transport cost}
\int_{Z} I_c(\rho(\cdot|z),\mu) d\nu(z) =
\min_{\substack{\pi \in P(X\times Y\times Z)\\ \pi_{XZ} = \rho\\ \pi_{YZ} = \mu \otimes v}}
\int_{X\times Y\times Z} c(x,y) d\pi(x,y,z),
\end{equation}
so there exists a Kantorovich solution in the form $\pi \in P(X\times Y\times Z)$.

\item If Assumption \ref{assumption: coercive and Heine Borel} from Appendix \ref{appendix: assumptions} holds, then there exists a barycenter $\mu \in P(Y)$. Specifically,
\begin{align}
\label{joint measure formulation of barycenter problem}
\min_{\mu \in P(Y)}
\int_{Z} I_c(\rho(\cdot|z),\mu) d\nu(z)
=&
\min_{\substack{\pi \in P(X\times Y\times Z)\\ \pi_{XZ} = \rho\\ \pi_{YZ} = \pi_Y \otimes \pi_Z}}
\int_{X\times Y\times Z} c(x,y) d\pi(x,y,z),
\end{align}
and the marginal $\pi_Y$ of every solution $\pi$ is a barycenter.
\end{enumerate}
\end{theorem}

\begin{proof}
See Appendix \ref{appendix: proof of wellposedness and existence}. Note that $\pi_{YZ} = \pi_Y \otimes \pi_Z$ implies that the barycenter is independent of the latent variable.
\end{proof}

\begin{remark}
There are pathological examples where the barycenter does not exist: if $X=Y=\R^d$ and
$c(x,y) = \exp[-\|x-y\|^2]$, then any barycenter will tend to be pushed arbitrarily far away. Assumption \ref{assumption: coercive and Heine Borel} is modeled after the squared distance cost $c=\|x-y\|^2$ and prevents such degeneracy.
\end{remark}

\subsection{Conditional transport maps}
\label{sec: conditional transport map}
Having shown that the barycenter $\mu$ exists, the next step is to find the transport maps $T_z$ and inverse transport maps $S_z$ between each conditional distribution $\rho(x|z)$ and $\mu(y)$.

From Theorem \ref{thm: well-posedness and existence}, the barycenter problem admits a Kantorovich solution $\pi \in P(X\times Y\times Z)$. If $\pi$ should also be a Monge solution, that is, $\pi$ were concentrated on the graph of some transport map $T: X\times Z \to Y$, it would follow that for $v$-almost all $z$, 
$$T(\cdot,z)\#\rho(x|z)=\mu(y)$$
or equivalently,
\begin{equation}
\label{def: monge transport map}
\tilde{T}\#\rho(x,z) = \mu(y)\otimes \nu(z), \text{ where } \tilde{T}(x,z):=(T(x,z),z)
\end{equation}
We show that this holds in general:

\begin{theorem}
\label{thm: barycenter transport map}
Given Polish spaces $X,Y,Z$, probability $\rho(x,z) \in P(X\times Z)$ and cost $c \in C(X \times Y)$ that is bounded below ($\inf c > -\infty$), under Assumptions \ref{assumption: coercive and Heine Borel} and \ref{Monge assumption} from Appendix \ref{appendix: assumptions}, the barycenter problem has a Monge solution: the minimum total transport cost (\ref{joint measure formulation of barycenter problem}) becomes
\begin{align}
\label{transport map objective}
\min_{\substack{\text{Borel measurable}\\T:X\times Z \to Y}}
\sup_{\substack{\psi_Y(y) \in C_b(Y)\\
\psi_Z(z) \in C_b(Z)\\
\int \psi_Z(z) d\nu(z) = 0}}
\int \big[c(x,T(x,z)) - \psi_Y(T(x,z))\psi_Z(z) \big] d\rho(x,z)
\end{align}
and every minimizer $T$ is a transport map from $\rho(x,z)$ to a barycenter $\mu(y)$.
\end{theorem}

\begin{proof}
See Appendix \ref{appendix: barycenter transport map}.
\end{proof}

\begin{remark}
The test function $\psi_Y(y)\psi_Z(z)$ in (\ref{transport map objective}) serves as the ``discriminator" that checks that all the conditional distributions $\rho(x|z)$ have been pushed-forward to the same barycenter $\mu$. The technique of discriminator has appeared in \cite{goodfellow2014generative,arjovsky2017wasserstein} to train the generative adversarial networks, and it has been applied to the barycenter problem by \cite{tabak2018conditionaldensity}, which derived test functions of the form
$$\psi(y,z) \text{ such that} \int \psi(y,z) d\nu(z) \equiv 0$$
Theorem \ref{thm: barycenter transport map} improves this technique, because $\psi_Y(y)\psi_Z(z)$ has much less complexity than $\psi(y, z)$. From a data-based perspective, with the distributions given through sample points $\{x_i,y_i,z_i\}_{i=1}^N$, the test function $\psi(y,z)$ becomes a full $N\times N$ matrix, whereas $\psi_Y(y), \psi_Z(z)$ are two $1\times N$ vectors, which can be seen as providing a rank-one factorization of $\psi(y, z)$. Later sections show that all computations are thereby reduced from quadratic to linear time $O(N)$.

An explanation for this improvement is that the barycenter problem has more freedom than the ordinary optimal transport problem. Optimal transport would require the pushforward $\tilde{T}\#\rho(x,z)$ to match a fixed target distribution, so that the dual problem needs to mobilize the entire $C_b(Y\times Z)$ to pin it down. For the barycenter problem, however, Theorem \ref{thm: well-posedness and existence} shows that $\pi_{YZ}=\tilde{T}\#\rho(x,z)$ only needs to satisfy the independence condition
$$\pi_{YZ} = \pi_Y \otimes \pi_Z.$$
Correspondingly, the dual problem only requires a small subspace of $C_b(Y\times Z)$.
\end{remark}

\medskip
Regarding the inverse transport maps $S_z$, one approach is to set $S_z = T_z^{-1}$. This is viable in many scenarios: for instance, by Brennier's Theorem \cite{villani2003topics}, the inverse function $T^{-1}_z$ exists almost surely,
$$T^{-1}_z \circ T_z(x) = x \text{ for $\rho(x|z)$-almost all $x$},$$
and it is the optimal transport map for the inverse transport,
$$T^{-1}_z \# \mu = \rho(x|z).$$
Then, computing $S_z$ becomes a simple regression problem. Given the labeled data $\{x_i,z_i\}_{i=1}^N$, we first compute the barycenter $\{y_i = T(x_i,z_i)\}$ and then find a map $S:Y\times Z\to X$ that approximates $\{x_i\}$ by $\{y_i,z_i\}$. This is the approach used by our algorithms.

It might be helpful to note that there is a more general approach, which directly solves the optimal transport from $\mu(y)\otimes \nu(z)$ back to $\rho(x,z)$. Assertion 1 of Theorem \ref{thm: well-posedness and existence} shows that there is always a Kantorovich solution, while the arguments of Theorem \ref{thm: barycenter transport map} can be applied to show that the inverse transport map $S(y,z)$ can be solved from
$$\min_{\substack{\text{Borel measurable}\\S:Y\times Z \to X}}
\sup_{\psi(x,z) \in C_b(X\times Z)}
\int \big[c(S(y,z),y) - \psi(S(y,z),z)\big] d\mu(y)d\nu(z) + \int \psi(x,z) d\rho(x,z).$$

\begin{remark}
As discussed in the introduction, given a labeled sample $\{x_i,z_i\}_{i=1}^N$, we can estimate each conditional distribution $\rho(x|z)$ by the computed sample $\{S_z \circ T_{z_i}(x_i)\}_{i=1}^N$. Yet, when $X=Y$ is Euclidean and $T_z(x)$ is differentiable (e.g. when modeled by a neural net), we can also derive the density function of $\rho(x|z)$: First, estimate the barycenter's density $\mu(y)$ from the computed sample $\{y_i\}_{i=1}^N$ (e.g. by kernel smoothing). Then, estimate the density through
$$\rho(x|z) = |J(T_z(x))| ~ \mu(T_z(x)),$$
where $|J|$ is the Jacobian determinant. Then, we can estimate the density function $\rho(x,z)$ through $\rho(x|z) \nu(z)$, where $\nu(z)$ is estimated from the $\{z_i\}$.
\end{remark}

\subsection{Latent factor discovery}
\label{sec: latent factor discovery}

Finally, we justify the definition (\ref{variability definition}) of variability, from which it follows that the minimization of the barycenter's variability, which is the objective of factor discovery, is equivalent to the maximization of total transport cost.
We illustrate this intuition in the case where $X=Y=\R^d$ and $c(x,y)=\|x-y\|^2$. Then, the optimal transport cost $I_c$ becomes $W_2^2$, where $W_2$ is the $2$-Wasserstein distance. (See \cite{villani2003topics,villani2008optimal} for a reference of Wasserstein distance, and \cite{agueh2011barycenters} for Wasserstein barycenters.)

\begin{theorem}
\label{thm: variance decomposition}
Given any measurable space $Z$ and probability measure $\rho(x,z) = \rho(x|z)\nu(z) \in P(\R^d\times Z)$, there exists a Wasserstein barycenter $\mu \in P(\R^d)$ that satisfies
\begin{equation}
\label{variance decomposition continuous}
Var(\rho(x)) - Var(\mu) =  \int_Z W_2^2(\rho(x|z),\mu) d\nu(z).
\end{equation}
\end{theorem}

\begin{proof}
See Appendix \ref{appendix: variance decomposition} for the proof. To illustrate the proof's intuition, consider the trivial case with Dirac masses:
$$\rho(x,z) = \sum_{k=1}^K P_k \delta_{x_k} \otimes \delta_{z_k}$$
where $P_k$ are positive weights. Then, $\rho(x)$ becomes $\sum P_k \delta_{x_k}$, and the unique $W_2$-barycenter $\mu$ is the Dirac mass on the mean of $\rho(x)$. Both sides of (\ref{variance decomposition continuous}) reduces to $Var(\rho)$. 
The rest would be an approximation argument that goes from Dirac masses to general probabilities, and the high-level idea is that the geometric properties of $\R^d$ can be lifted to $(P(\R^d),W_2)$.
One result that we apply repeatedly comes from Proposition 3.8 and Remark 3.9 in \cite{agueh2011barycenters}: Under technical conditions, given weights $P_k$ and conditionals $\rho_k$, their $W_2$-barycenter $\mu$, and optimal transport maps $T_k$ from $\mu$ back to $\rho_k$, we have the identity
\begin{equation*}
\sum_{k=1}^K P_k T_k = Id, ~\mu-a.e.
\end{equation*}
This identity indicates that the $W_2$-barycenter is exactly the convex sum of the conditionals, just as in the simple setting with Dirac masses illustrated above.
It is worth mentioning that this identity can be generalized to compact Riemannian manifolds \cite[Theorem 4.4]{kim2017wasserstein}.
\end{proof}

\begin{corollary}
\label{cor: variability is variance}
Let $V:P(\R^d)\to \R\cup\{\infty\}$ be any function such that $V(\delta_x) = 0$ for any Dirac mass $\delta_x$. Then, $V$ is the variance $Var$ if and only if for any measurable space $Z$ and any $\rho(x,z) \in P(\R^d\times Z)$, there exists a barycenter $\mu$ that satisfies
\begin{equation}
\label{variance decomposition V}
V(\rho(x)) - V(\mu) =  \int_Z W_2^2(\rho(x|z),\mu) d\nu(z)
\end{equation}
\end{corollary}

\begin{proof}
The ``only if" part follows from Theorem \ref{thm: variance decomposition}. For the ``if" part, set $Z=\R^d$. Given any $\rho(x) \in P(\R^d)$, set $\rho(x,z) = \delta_z(x) \rho(z)$. Then, the $X$-marginal of $\rho(x,z)$ is $\rho(x)$, and the unique barycenter $\mu$ is the Dirac measure at the mean of $\rho(x)$. Then, (\ref{variance decomposition V}) reduces to $V(\rho(x)) = Var(\rho(x))$.
\end{proof}

Since a Dirac measure $\delta_{x}$ represents a deterministic event without any uncertainty, $V(\delta_{x})$ should be zero for any reasonable variability function $V$. Then, it follows from Corollary \ref{cor: variability is variance} that the variability defined by (\ref{variability definition}) is exactly the variance $Var$, when the cost is the Euclidean squared distance.

\begin{remark}
As a further justification, notice that if the cost $c(x,y) = \|x-y\|^2$ is generalized to $(x-y)^T Q (x-y)$ for some positive-definite symmetric matrix $Q$, then the corresponding variability becomes a ``weighted" variance, with a different scaling factor in each eigenspace of $Q$:
$$V(\rho) = \int (x-\overline{x})^T Q (x-\overline{x}) d\rho(x) = Var(\sqrt{Q}\#\rho)$$
where $\overline{x}$ is the mean and $\sqrt{Q}\#\rho$ is the pushforward by the linear map $\sqrt{Q}$.
\end{remark}
\begin{proof}
Given any $\rho(x,z)$, formula (\ref{joint measure formulation of barycenter problem}) becomes
\begin{align*}
\min_{\substack{\pi\in P(\R^d\times \R^d\times Z)\\ \pi_{XZ}=\rho(x,z)\\ \pi_{YZ} = \pi_Y\otimes\pi_Z}} \int (x-y)^T Q (x-y) d\pi(x,y,z) &= 
\min_{\substack{\pi\in P(\R^d\times \R^d\times Z)\\ \pi_{XZ}=\rho(x,z)\\ \pi_{YZ} = \pi_Y\otimes\pi_Z}} \int \|x-y\|^2 d(\sqrt{Q},\sqrt{Q},Id)\#\pi(x,y,z)\\
&= Var(\sqrt{Q}\#\rho(x)) - Var(\sqrt{Q}\#\mu),
\end{align*}
where $\sqrt{Q}\#\mu$ is a barycenter of $(\sqrt{Q},Id)\#\rho(x,z)$ (under cost $\|x-y\|^2$) that satisfies (\ref{variance decomposition continuous}). Then, $\mu$ is a barycenter of $\rho(x,z)$ (under cost $(x-y)^TQ(x-y)$) and satisfies (\ref{variance decomposition V}) with $V(\rho)=Var(\sqrt{Q}\#\rho)$.
\end{proof}

\medskip
It follows from definition (\ref{variability definition}) that the variability of the barycenter is complementary to the total transport cost (\ref{transport map objective}) to the barycenter. As argued in the introduction,  given any unlabeled data $\rho(x)$, latent factor discovery looks for a labeling $\rho(x,z)$ that minimizes the variability of its barycenter. Then, factor discovery has the following equivalent formulation based on (\ref{transport map objective}):
\begin{equation}
\label{factor discovery objective}
\sup_{\substack{\rho(x,z)\\ \rho_X = \rho(x)}}
\min_{\substack{\text{measurable}\\T:X\times Z \to Y}}
\sup_{\substack{\psi_Y(y) \in C_b(Y)\\
\psi_Z(z) \in C_b(Z)\\
\int \psi_Z(z) d\nu(z) = 0}}
\int \big[c(x,T(x,z)) - \psi_Y(T(x,z))\psi_Z(z) \big] d\rho(x,z).
\end{equation}

To solve (\ref{factor discovery objective}), factor $\rho(x,z)$ into $p(z|x)\rho(x)$, where $p(z|x)$ is the conditional label distribution for the sample point $x$. In particular, when $Z$ is finite, $p(z|x)$ can be seen as a classifier on $X$. Thus, an effective way to optimize $\rho(x,z)$ and regularize the solution is to parameterize $p(z|x)$: e.g. we can set
\begin{equation}
\label{neural net assignment distribution}
p_{\theta}(z |x) = G_{\theta}(x, \cdot)\# \mathcal{N}(0,I)
\end{equation}
where $G_{\theta}$ is a neural net with two inputs and $\mathcal{N}(0,I)$ is a unit normal distribution.


Something to be aware of is the $\sup_{\rho}$ in problem (\ref{factor discovery objective}), which could lead to degenerate solutions if there is no bound on the ``expressivity" of $p(z|x)$. Its behavior is well-controlled when $Z$ is finite and $\rho(x,z)$ is simply a clustering plan of $\rho(x)$. However, theoretical issues could arise when $Z$ is a larger space such as $[0,1]$: we can always find a measurable map $f:[0,1]\to X$ that transports the uniform distribution $U[0,1]$ onto $\rho(x)$ (e.g. via Theorem 1.1 of \cite{kallenberg2017random}), and then we can set
\begin{equation}
\label{singular decomposition}
v(z):= U[0,1], ~\rho(x,z) := \delta_{f(z)}(x) ~v(z)
\end{equation}
This $\rho(x,z)$ has the right marginal $\rho_X=\rho(x)$, but since every conditional $\rho(x|z)$ is a Dirac mass, the barycenter is also a Dirac mass. So problem (\ref{factor discovery objective}) leads to a trivial solution compressing all data to a single point.

This situation is analogous to ``overfitting" in regression problems, when one intends to learn a rough sketch of the data but the algorithm learns all the fine details instead. Fortunately, such an issue can be avoided in practice by controlling how much the algorithm learns during training. Suppose we adopt an implementation similar to (\ref{neural net assignment distribution}). A desirable solution $p(z|x)$ should capture the overall shape of the data $\rho(x)$ but should not become as singular as the solution $p(z|x) = \delta_{f^{-1}(x)}(z)$ from (\ref{singular decomposition}). The key is to control the complexity of the function $G_{\theta}$, so that $G_{\theta}$ does not become as complex as, for instance, the $f(z)$ in (\ref{singular decomposition}). This can be achieved by either explicit or implicit regularizations. Specifically, the complexity of $G_{\theta}$ can be characterized by appropriate functional norms such as the Barron norm \cite{e2019barron}, the number of parameters, or the Fourier spectrum \cite{xu2019frequency}. Explicit regularizations include penalizing the norm \cite{e2018priori}, bounding the number of parameters, and early stopping \cite{li2019stopping}. Implicit regularizations include training dynamics that protect against overfitting \cite{advani2017high} and the frequency principle \cite{rahaman2018spectral} that prioritizes learning low frequencies. Even though these regularity results were obtained in settings different from problem (\ref{factor discovery objective}), they are in a sense universal in neural network training, and hence applicable to our setting. Indeed, none of our experimental results in Section \ref{sec: test results} exhibit degenerate solutions.

Besides the unlimited complexity of $p(z|x)$, there is another, rather trivial way for degenerate solutions to arise in problem (\ref{factor discovery objective}). If $X$ and $Z$ are Euclidean and $Z$ has higher dimension than $X$ (or more generally, $X \subseteq Z$), then we can set
\begin{equation*}
p(z|x) = \delta_z(x), ~\rho(x,z) = \delta_x(z) \rho(x)
\end{equation*}
This kind of overfitting is analogous to an autoencoder whose hidden layers have the same size as the input/output layers, so that the network can become the identity function and learn the trivial latent variable $z = x$. We will discuss more about the connection between problem (\ref{factor discovery objective}) and autoencoders in Sections \ref{sec: label net} and \ref{sec: autoencoder}.

\section{Algorithmic design}
\label{sec: algorithmic design}
In the previous section, we converted conditional density estimation, a problem involving probability distributions, to the dual of the barycenter problem (\ref{transport map objective}), which involves only functions. Then, latent factor discovery becomes (\ref{factor discovery objective}) with an additional maximization over all labelings $\rho(x, z)$ whose marginal $\rho_X$ is the given unlabeled distribution $\rho(x)$.

In practice, we are given a labeled finite sample set $\{x_i,z_i\}_{i=1}^N$ for conditional density estimation, and the objective (\ref{transport map objective}) becomes
\begin{equation}
\label{data-based barycenter objective}
\inf_{\tau} \sup_{\xi} L(\tau,\xi)
= \frac{1}{N} \sum_{i=1}^N \big[ c\big(x_i,T_{\tau}(x_i,z_i)\big) - \psi^Y_{\xi} \big(T_{\tau}(x_i,z_i)\big) \tilde{\psi}^Z_{\xi} (z_i) \big],
\end{equation}
where $T_{\tau}, \psi^Y_{\xi}, \psi^Z_{\xi}$ are maps parameterized by $\tau,\xi$ and
\begin{equation}
\label{integral constraint for psi Z}
\tilde{\psi}^Z_{\xi}(z) := \psi^Z_{\xi}(z) - \frac{1}{N}\sum_{i=1}^N \psi^Z_{\xi}(z_i),
\end{equation}
which is a sample-based version of the constraint $\int \psi^Z dv = 0$ from (\ref{transport map objective}).

For factor discovery, we follow the analysis in Section \ref{sec: latent factor discovery} to model the labelings $\rho(x,z)$ via $p_{\theta}(z|x)\rho(x)$, where $p_{\theta}(z|x)$ is a parameterized conditional label distribution. Then the objective (\ref{factor discovery objective}) becomes
\begin{align}
\label{data-based factor discovery objective}
\begin{split}
\sup_{\theta} \inf_{\tau} \sup_{\xi} L(\theta,\tau,\xi)
&= \frac{1}{N} \sum_{i=1}^N \E_{p_{\theta}(z|x_i)} \big[ c\big(x_i,T_{\tau}(x_i,z)\big) - \psi^Y_{\xi} \big(T_{\tau}(x_i,z)\big) \tilde{\psi}^Z_{\xi} \big(z\big) \big]\\
\tilde{\psi}^Z(z) &:= \psi^Z(z) - \frac{1}{N} \sum_{i=1}^N \E_{p_{\theta}(\tilde{z}|x_i)} \psi^Z(\tilde{z}).
\end{split}
\end{align}
It could be difficult to compute the expectation $\E_{p_{\theta}(z|x_i)}$ directly, unless it has an analytical solution or $Z$ is finite. For simplicity, we often restrict to the case $p_{\theta}(z|x) = \delta_{z_{\theta}(x)}$, that is, the labeling is given by a deterministic map, $z_i = z_{\theta}(x_i)$.

In the following sections, we demonstrate the efficacy of (\ref{data-based barycenter objective}) and (\ref{data-based factor discovery objective}) by implementing them through neural networks. We focus on the special case when $X,Y$ are Euclidean spaces, $Z$ is either Euclidean or finite, and the cost $c$ is differentiable.

\subsection{BaryNet}

Since (\ref{data-based barycenter objective}) and the deterministic version of (\ref{data-based factor discovery objective}) are optimization problems that involve only functions, it is natural to solve them by neural networks. By Theorems 1 and 2 of \cite{hornik1991approximation}, feedforward neural nets are universal approximators for continuous functions $C(\R^d)$ and measurable functions $L^1(d\rho)$, so they can model the continuous test function $\psi^Y(y)$ (and $\psi^Z(z)$ when $Z$ is Euclidean) and the measurable transport map $T(x,z)$. We can also model the conditional latent distribution $p_{\theta}(z|x)$ or the deterministic $z_{\theta}(x)$ by neural nets, if we require that they depend continuously on their parameters. Hence, (\ref{data-based barycenter objective}) and (\ref{data-based factor discovery objective}) become a collection of interacting networks, an architecture that we refer to as ``BaryNet", for barycenter network. As (\ref{data-based barycenter objective}) and (\ref{data-based factor discovery objective}) can be seen as a supervised/unsupervised pair, we call the corresponding networks the supervised/unsupervised BaryNet.

One advantage of neural nets is the ease to control their expressivity. A neural net can approximate any continuous function if either its width \cite{hornik1991approximation} or depth \cite{lu2017expressive} goes to infinity, so we can adjust the network's size, or more generally its functional norm \cite{e2019barron}, to solve problems with varying complexity. In factor discovery, for instance, if we know a priori that the ideal labeling $z_{\theta}$ should approximate the data's principal components, or if we desire simple labelings that are more interpretable, then we can reduce the size of $z_{\theta}$ or penalize its norm.

Another advantage is that the structure of the solution can be easily encoded in the network architecture. For instance, if the ideal solution should be a perturbation to the identity: $f(x) = x+o(|x|)$, then we can model only the perturbation part: $f_{\theta}(x) = x+ g_{\theta}(x)$. This approach, known as ``residual network" \cite{he2016deep}, makes the network easier to optimize and increases the likelihood to reach optimal solutions. It turns out that this residual design resembles the structure of the transport map $T(x,z)$.

\subsubsection{Transport and inverse transport nets}
\label{sec: transport net}
As in residual networks, our transport map can be modeled as
\begin{equation}
\label{residual transport map}
T_{\tau}(x,z) = x + R_{\tau}(x,z),
\end{equation}
if the transportation takes place within a single space, $X=Y=\R^d$. A motivation is that solutions to the barycenter problem (\ref{transport map objective}) generally have the following properties:

\begin{enumerate}
\item Each transport map $T_{z}\#\rho(x|z) = \mu$ starts from an identity component $x$. This holds in general for optimal transport maps in Euclidean spaces, as these are special cases of the transport maps on complete Riemannian manifolds:
\begin{equation*}
T(x) = \exp_x(R(x))
\end{equation*}
where $R(x)$ is a tangent vector that ``points" to the transportation's destination (e.g. see McCann's theorem, Theorem 2.47 of \cite{villani2003topics}), which in the Euclidean setting reduces to $T(x) = x+R(x)$.

As a concrete example, Theorem 2.44 of \cite{villani2003topics} shows that if the cost $c=c(x-y)$ is strictly convex and superlinear, and if the source and target measures are absolutely continuous, then the optimal transport map has the residual form
\begin{equation*}
T(x) = x - \nabla c^* (\nabla \phi(x)),
\end{equation*}
where $c^*$ is the Legendre transform of $c$ and $\phi$ is $c$-concave.
Another example is provided in Section 3.3 of \cite{tabak2018explanation}: if $\rho(x|z)$ and $\mu$ have similar shapes, then the transport map have the form
\begin{equation}
\label{first-moment approximation transport}
T(x) \approx x + \beta(z), ~\beta(z) = \overline{y} - \overline{x}(z),
\end{equation}
where $\overline{x}(z),\overline{y}$ are the means of $\rho(x|z)$ and $\mu$. This $T(x)$ approximates the optimal transport map up to the first moment.

\item Each $T_z$ is invertible: it was argued in Section \ref{sec: conditional transport map} that under general conditions, such as under the hypothesis of Brennier's theorem \cite{villani2003topics}, the transport map $T_z$ is invertible $\rho(x|z)$-almost surely and its inverse $T_z^{-1}$ transports $\mu$ back to $\rho(x|z)$. The residual design (\ref{residual transport map}) is an effective way to ensure that $T_z$ is invertible: if the residual term is small, in the sense that $\nabla_x R \approx O$, then the inverse exists locally by the inverse function theorem, and it has the form
\begin{equation}
\label{asymptotic inverse}
S_z(x) = x - R(x,z) + O(\|\nabla_x R(x,z)\| \cdot \|R(x,z)\|).
\end{equation}
\end{enumerate}

\noindent
An additional benefit of the residual design is that the Jacobian matrix $\nabla_x T_{\tau}$ is close to the identity when the residual term is small, thus alleviating the exploding and vanishing gradient problem during training \cite{he2016deep}.\\


Regarding the inverse transport map $S(y,z)$, formula (\ref{asymptotic inverse}) suggests that we should also model $S(y,z)$ as a residual network with the same architecture as $T_{\tau}$. As argued in Section \ref{sec: conditional transport map}, after the transport map $T_{\tau}$ is obtained from the barycenter problem (\ref{data-based barycenter objective}) or (\ref{data-based factor discovery objective}), the inverse $S(y,z)$ can be found through a regression problem:
\begin{equation}
\label{inverse transport by regression}
\inf_{\theta} \mathbb{E}_{\rho(x,z)}\Big[ c\big(x, S_{\theta}(T_{\tau}(x,z),z)\big) \Big]
\approx \frac{1}{N} \sum_{i=1}^N c\big(x_i, S_{\theta}(y_i,z_i)\big).
\end{equation}

\subsubsection{Label net}
\label{sec: label net}
For the factor discovery problem (\ref{data-based factor discovery objective}), we focus on two cases, when $Z$ is either finite: $Z=\{1,\dots K\}$, or Euclidean: $Z=\R^k$. For the finite case, the conditional label distribution $p(z|x)$ becomes a probability vector, which can be modeled by the SoftMax function
\begin{equation*}
p_{\theta}(z|x) = \text{SoftMax}(p_{\theta}(x)) = \bigg[\frac{e^{p^1_{\theta}(x_i)}}{\sum_{k=1}^K e^{p^k_{\theta}(x_i)}}, ~\dots~ \frac{e^{p^K_{\theta}(x_i)}}{\sum_{k=1}^K e^{p^k_{\theta}(x_i)}} \bigg],
\end{equation*}
where $p_{\theta}(x): \R^d\to\R^K$ is some neural net. The test function $\psi^Z(z)$ reduces to a vector $[q_1,\dots q_K]$, and the transport map $T(x,z)$ splits into $K$ maps $T^k(x)$. The objective (\ref{data-based factor discovery objective}) becomes
\begin{align}
\label{discrete factor discovery objective}
\begin{split}
\sup_{\theta} \inf_{\tau} \sup_{\xi} L(\theta,\tau,\xi)
&= \frac{1}{N} \sum_{i=1}^N \sum_{k=1}^K p_{\theta}(k|x_i) \big[ c\big(x_i,T^k_{\tau}(x_i)\big) - \psi^Y_{\xi} \big(T^k_{\tau}(x_i)\big) \tilde{q}_k \big],\\
\tilde{q}_k &:= q_k - \sum_{h=1}^K  q_h \sum_{i=1}^N \frac{p_{\theta}(h|x_i)}{N}.
\end{split}
\end{align}
If we consider the conditional distributions $\rho_k := \rho(x|k)$ as clusters, then $p(k|x_i)$ is the membership probability that sample $x_i$ belongs to cluster $\rho_k$, and $\rho_1,\dots \rho_K$ become a clustering plan for $\rho(x)$ with weights $\frac{1}{N}\sum_{i=1}^N p(k|x_i)$. Hence, problem (\ref{discrete factor discovery objective}) reduces to clustering (with soft assignments).

\begin{remark}
In prior work, optimal transport barycenter has been applied to the clustering problem in \cite{tabak2018explanation,yang2019clustering}, which study the case with squared Euclidean distance cost and solve directly the primal problem
\begin{equation}
\label{clustering primal}
\min_{p(k|x_i)}Var(\text{barycenter})
\end{equation}
instead of the dual problem (\ref{discrete factor discovery objective}). If we simplify the transport maps $T_k$ by their first-moment approximations (\ref{first-moment approximation transport}), then \cite{tabak2018explanation} shows that (\ref{clustering primal}) produces the $k$-means algorithm. If we approximate $T_k$ so that it aligns the second moments of $\rho(x|z)$ and $\mu$, then \cite{yang2019clustering} shows that (\ref{clustering primal}) leads to more robust algorithms that recognize non-isotropic clusters. While \cite{tabak2018explanation,yang2019clustering} only compute the membership probabilities $p(k|x_i)$, the dual problem (\ref{discrete factor discovery objective}) solves for both $p(k|x_i)$ and $T_k$, without any simplifying assumption on $T_k$.
\end{remark}

For the Euclidean case $Z=\R^k$, we focus on deterministic labelings $p(z|x_i) = \delta_{z_i}$ and model $z_i$ by $z_{\theta}(x_i)$. Then the objective (\ref{data-based factor discovery objective}) simplifies into
\begin{align}
\label{deterministic factor discovery objective}
\begin{split}
\sup_{\theta} \inf_{\tau} \sup_{\xi} L(\theta,\tau,\xi)
&= \frac{1}{N} \sum_{i=1}^N \big[ c\big(x_i,T_{\tau}(x_i,z_{\theta}(x_i))\big) - \psi^Y_{\xi} \big(T_{\tau}(x_i,z_{\theta}(x_i))\big) \tilde{\psi}^Z_{\xi} \big(z_{\theta}(x_i)\big) \big],\\
\tilde{\psi}^Z(z) &:= \psi^Z(z) - \frac{1}{N} \sum_{i=1}^N \psi^Z(z_{\theta}(x_i)).
\end{split}
\end{align}

A useful property of the labeling $\rho(x,z)$ is that it is invariant under bijections of the latent variable space $Z$, because essentially we are only looking for a disintegration $\rho(x|z)$ regardless of the specific $z$. This is trivial for the clustering problem, since any permutation of the labels $Z=\{1,\dots K\}$ produces a different but equivalent labeling. In general, given any $\rho(x,z), T(x,z), \psi^Z(z)$ for the factor discovery problem (\ref{factor discovery objective}) and given any measurable $f: Z\to Z$, the triple
$$\big( (Id,f)\#\rho(x,z), ~T, ~\psi^Z \big)$$
produces the same value as
$$\big( \rho(x,z), ~T\circ (Id,f), ~\psi^Z \circ f \big),$$
so they can be considered equivalent solutions.

This invariance suggests that we can reduce the freedom in the architecture of $z_{\theta}$ without affecting the expressivity of the BaryNet. Let $NN(X,Z)$ denote the set of all neural nets mapping $X$ to $Z$. Originally, $z_{\theta}$ should range in $NN(\R^d,\R^k)$, which is dense in $C(K,\R^k)$ for any compact $K\subseteq\R^d$, but now we can restrict to some smaller family $\mathcal{Z} \subsetneq NN(\R^d,\R^k)$ such that $NN(\R^k,\R^k) \circ \mathcal{Z}$ is dense in $C(K,\R^k)$. For instance, it is straightforward to show that $\mathcal{Z}$ can be the set of bounded Lipschitz neural nets whose last layer is bias-free. Such restriction is helpful for training, as it reduces the size of the search space.

\begin{remark}
The finite case (\ref{discrete factor discovery objective}) and the Euclidean case (\ref{deterministic factor discovery objective}) can be combined into a cluster detection task: first, we solve (\ref{deterministic factor discovery objective}) with $Z=\R^k$ and $k\leq 3$ small, so that $Z$ can be visualized. Then, we inspect the latent variables $\{z_i\}$ to see if there are recognizable clusters, and how many. If so, we perform clustering (\ref{discrete factor discovery objective}) on either the original data $\{x_i\}$ or the processed data $\{z_i\}$.
\end{remark}

As discussed in Section \ref{sec: latent factor discovery}, degenerate solutions $\rho(x,z)$ can arise either because the complexity of $p(z|x)$ goes unbounded or because $\text{dim} Z = k \geq d = \text{dim} X$. For the latter case, we obtain the trivial labeling $z=x$, that is, $\rho(x,z) = \delta_{z}(x)\rho(z)$. One simple regularization is to impose a bottleneck architecture, setting $k<d$. It is analogous to the \textit{undercomplete autoencoder} \cite{goodfellow2016deep}, whose intermediate layers are smaller than the input/output layers so that the autoencoder cannot pass by learning the identity function. 
As it turns out, the connection to autoencoders runs deeper than this.

\subsection{Relation to autoencoders and generative modeling}
\label{sec: autoencoder}
The unsupervised BaryNet (\ref{data-based factor discovery objective}) can be conceptualized in terms of encoders and decoders. The label net $z_{\theta}$ (or more generally $p_{\theta}(z|x)$) encodes each $x_i$ into a latent code $z_i$, and to recover $x_i$, the inverse transport map $S(\cdot,z)$ decodes $z_i$ probabilistically as the conditional distribution $\rho(x|z_i) = S_{z_i}\#\mu$. Then the ``reconstruction loss" of the encoding/decoding process should be proportional to the variability of $\rho(x|z)$,
\begin{equation}
\label{reconstruction loss proportional to conditional variability}
\text{Reconstruction loss} \propto \int V\big(\rho(x|z)\big) dv(z).
\end{equation}
Meanwhile, it is natural to expect that the variability of the barycenter $V(\mu)$ is positively correlated to each $V(\rho(x|z))$, that is, greater variability in the conditional distributions results in greater variability in their representative $\mu$. By combining the two correlations, it appears that $V(\mu)$ behaves like a reconstruction loss, making the factor discovery problem analogous to an autoencoder.

We formalize this intuition in the case when $X=Y=\R^d$ with squared distance cost $c(x,y)=\|x-y\|^2$, and when all conditionals $\rho(x|z)$ are Gaussians. By Corollary \ref{cor: variability is variance}, the variability $V$ becomes the variance $Var$. Denote each $\rho(x|z)$ by $\mathcal{N}(\overline{x}(z),S(z))$, where $\overline{x}$ is the mean and $S$ is the covariance matrix. Denote the principal square root matrix by $\sqrt{S}$.

\begin{theorem}
\label{thm: Gaussian covariance}
Given any measurable space $Z$ and any $\rho(x,z) = \rho(x|z)v(z) \in P(\R^d\times Z)$ such that each $\rho(x|z)$ is a Gaussian distribution $\mathcal{N}(\overline{x}(z),S(z))$, if the marginal $\rho(x)$ has finite second moment: $\E_{\rho(x)}\big[\|x\|^2\big] \leq \infty$, then there exists a barycenter $\mu$, which is a Gaussian $\mathcal{N}(\overline{x},S)$ and satisfies
\begin{align*}
\overline{x} &= \int \overline{x}(z) dv(z) = \E_{\rho(x)}[x]\\
S &= \int \sqrt{\sqrt{S} \cdot S(z) \cdot \sqrt{S}} ~dv(z).
\end{align*}
Furthermore, if the set of $z$ such that $\rho(x|z)$ is non-degenerate (its covariance $S(z)$ is positive-definite) has positive $v$-measure, then this is the unique barycenter.
\end{theorem}

\begin{proof}
See Appendix \ref{appendix: Gaussian covariance}.
\end{proof}

Theorem \ref{thm: Gaussian covariance} implies that the variability of the barycenter $V(\mu) = Var(\mu) = Tr[S]$ is positively correlated to the variability of the conditional distributions $V(\rho(x|z)) = Tr[S(z)]$. A rigorous argument could apply the implicit function theorem on Banach spaces to show that $S$ depends differentiably on $S(z) \in L^1((\R^k,dv)\to\R^{d\times d})$, and that any perturbation to $S(z)$ that increases its eigenvalues would also increase the eigenvalues of $S$. Nevertheless, the following corollary seems sufficient to justify the positive correlation.

\begin{corollary}
\label{cor: Gaussian std}
If we further assume that each $\rho(x|z)$ is an isotropic Gaussian: $S(z) = std^2(z) \cdot Id$, then the unique barycenter is an isotropic Gaussian with a standard deviation of
\begin{equation*}
std = \int std(z) dv(z).
\end{equation*}
\end{corollary}

\begin{proof}
Insert $S(z) = std^2(z) \cdot Id$ into Theorem \ref{thm: Gaussian covariance}.
\end{proof}

It follows that $V(\mu) = std^2$ is proportional to $\int std^2(z) dv = \int V(\rho(x|z)) dv(z)$.\\

Meanwhile, the intuition (\ref{reconstruction loss proportional to conditional variability}) can be justified by the following calculation:
\begin{align*}
\int Var(\rho(x|z)) dv(z) &= \frac{1}{2} \iiint \|x-y\|^2 d\rho(x|z) d\rho(y|z) dv(z)\\
&= \frac{1}{2} \iiint \|x-y\|^2 d\rho(y|z) dp(z|x) d\rho(x)\\
&= \frac{1}{2} \E_{\rho(x)} \E_{p(z|x)} \E_{\rho(y|z)} \big[ \|x-y\|^2 \big]\\
&= \frac{1}{2} \E_{\rho(x)} \E_{\text{encoder}} \E_{\text{decoder}} \big[ \text{pointwise reconstruction error } c(x,y) \big],
\end{align*}
where we interpret the conditional label distribution $p_{\theta}(z|x_i)$ as the encoder and the conditional density $\rho(x|z_i) = S_z\#\mu$ as the probabilistic decoding of $x_i$. The last term above is exactly the reconstruction loss of a stochastic autoencoder, and has been used as objective function by models such as the Wasserstein autoencoder (WAE) \cite{tolstikhin2017wasserstein}.

We conclude that the barycenter's variability is positively correlated with the reconstruction loss of the encoder $p_{\theta}(z|x)$ and decoder $S_z\#\mu$, thus verifying the analogy between unsupervised BaryNet and autoencoders. Nevertheless, a positive correlation does not imply equivalence. For instance, for the clustering problem with $Z=\{1,\dots K\}$, the classical autoencoder's reconstruction loss reduces to the sum of within-cluster variances and the algorithm becomes $k$-means. Yet, \cite{yang2019clustering} shows empirically that minimizing the barycenter's variance leads to algorithms more robust than $k$-means.

\begin{remark}
This correlation was foreshadowed by \cite{tabak2018explanation}, which studied the primal problem (\ref{clustering primal}) of factor discovery. Suppose that all $\rho(x|z)$ have the same shape (i.e. equal up to translations), then the transport maps $T_z$ are simplified into (\ref{first-moment approximation transport}), and factor discovery reduces to finding a principal surface (hypersurface). Specifically, problem (\ref{clustering primal}) can be written as
\begin{align}
\nonumber
\min_{p(z|x)} Var(\mu) &= \min_{p(z|x)} \int Var\big(\rho(x|z)\big) dv(z)\\
\nonumber
&= \min_{p(z|x)} \min_{m(z)} \iint ||x-m(z)||^2 d\rho(x|z) dv(z)\\
\label{principal surface objective}
&= \min_{p(z|x)} \min_{m(z)} \iint ||x-m(z)||^2 dp(z|x) d\rho(x)
\end{align}
This minimization problem can be solved by alternating descent, using the following updates:
\begin{align*}
p(z|x) &\gets \delta_{z(x)}, ~z(x) = \text{argmin}_{z\in Z} ||x-m(z)||^2\\
m(z) &\gets \tilde{x}(z) = \E_{\rho(x|z)}[x]
\end{align*}
Thus, we recover the principal surface algorithm \cite{hastie2005elements}, which produces the hypersurface $m(z)$ that summarizes the data $\rho(x)$. Meanwhile, formula (\ref{principal surface objective}) is also the objective function of the undercomplete autoencoder (with a probabilistic encoder) \cite{goodfellow2016deep}. Hence, the undercomplete autoencoder can be seen as a first-moment approximation of factor discovery (when we only care about the differences in the means of $\rho(x|z)$ and restrict the transport maps to the rigid translations (\ref{first-moment approximation transport})).
This equivalence is more pronounced in the linear setting: assuming further that $z(x)$ or $m(z)$ is linear, then factor discovery (\ref{first-moment approximation transport}) reduces to principal component analysis \cite{tabak2018explanation}, while the autoencoder without nonlinearity also becomes PCA \cite{goodfellow2016deep}.
\end{remark}

Naturally, the next step is to apply the regularization techniques of autoencoders to BaryNet, as we have already explored the undercomplete autoencoder in Section \ref{sec: label net}. One popular regularization requires that the latent distribution $v(z)$ of any labeling $\rho(x,z)$ must match some prescribed distribution $P_Z$ (e.g. a unit Gaussian in $Z=\R^k$). Then, the task of the autoencoder reduces to finding a coupling $\rho(x,z)$ between $\rho(x)$ and $P_Z$, such that the encoder $\rho(z|x)$ and decoder $\rho(x|z)$ minimize the reconstruction loss. The motivation of this regularization is that $v = P_Z$ becomes easier to sample, and together with the decoder $\pi(x|z)$, it makes possible the random sampling of $\rho(x)$. Often, the requirement that $v(z) = P_Z$ is relaxed, replacing it by a penalty on some distance between $v(z)$ and $P_Z$ \cite{tolstikhin2017wasserstein}.

This regularization technique was introduced by the Adversarial autoencoder (AAE) \cite{makhzani2015adversarial}, which refers to $P_Z$ as the \textit{prior} and the latent distribution $v(z)$ as the \textit{aggregated posterior}. AAE is based on the Variational autoencoder (VAE) \cite{kingma2013auto}, which penalizes the KL-divergence between the conditional latent distribution $p(z|x)$ and $P_Z$,
$$\E_{\rho(x)} \big[ D_{KL}(p(z|x),P_Z) \big]$$
and AAE replaces this penalty by
\begin{equation*}
D_{GAN} (v(z),P_Z) := \max_{D(z)} \E_{P_z}\big[\log D(z)\big] + \E_{v(z)}\big[\log(1-D(z))\big]
\end{equation*}
This penalty is taken from GAN, and $D(z)$ is a discriminator that estimates the likelihood that a given $z$ comes from $v(z)$ or $P_z$.\\

By applying the regularization $v(z) = P_z$ to the factor discovery problem, we obtain
\begin{equation*}
\min_{\substack{\rho(x,z)\in P(X\times Z)\\
\rho_X = \rho(x), \rho_Z = P_Z}}
V(\text{barycenter of } \rho(x,z)).
\end{equation*}
Then, we derive a problem analogous to unsupervised BaryNet (\ref{data-based factor discovery objective}),
\begin{align}
\label{barycentric autoencoder}
\begin{split}
\sup_{\theta} \inf_{\tau} \sup_{\xi} L(\theta,\tau,\xi)
=& \E_{\rho(x)} \E_{p_{\theta}(z|x)} \big[ c\big(x,T_{\tau}(x,z)\big) - \psi^Y_{\xi} \big(T_{\tau}(x,z)\big) \tilde{\psi}^Z_{\xi} \big(z\big) \big]\\
&- \big[\E_{\rho(x)} \E_{p_{\theta}(z|x)} \phi_{\tau}(z) - \E_{P_Z} \phi_{\tau}(z) \big]\\
\tilde{\psi}^Z(z) :=& \psi^Z(z) - \E_{\rho(x)} \E_{p_{\theta}(\tilde{z}|x)} \psi^Z(\tilde{z}),
\end{split}
\end{align}
which we call Barycentric autoencoder (BAE). The additional term in (\ref{barycentric autoencoder}) is equivalent to
$$\inf_{\theta}\sup_{\tau} \E_{(v_{\theta}-P_Z)(z)} \big[ \phi_{\tau}(z) \big]$$
and serves as a discriminator that enforces $v=P_Z$. Alternatively, we can use the $f$-divergences or the maximum mean discrepancy \cite{gretton2012kernel,tolstikhin2017wasserstein} to penalize the difference between $v$ and $P_Z$.

Once the barycenter $\mu \approx \{y_i\}_{i=1}^N$ and inverse transport map $S(y,z)$ are computed, $\rho(x,z)$ can be estimated through
\begin{equation*}
\rho(x,z) = S\# (\mu \otimes P_Z) \approx \{S(y_i,z_j)\}_{i=1,\dots N}^{j=1,\dots M},
\end{equation*}
where $\{z_j\}_{j=1}^M$ is any sample drawn from $P_Z$. An advantage of BAE is that since the distribution $P_Z$ is known, there is an unlimited supply of $\{z_j\}$, and thus unlimited samples for $\rho(x,z)$ and $\rho(x)$.

\subsection{Semi-supervised factor discovery}
\label{sec: semi-supervised factor discovery}
Let us mention briefly that BaryNet can be adapted to the semi-supervised setting. Since the supervised (\ref{data-based barycenter objective}) and unsupervised (\ref{data-based factor discovery objective}) differ only in the freedom to choose $p_{\theta}(z|x)$, it is straightforward to modify $p(z|x)$ to be semi-supervised in the following two scenarios, which we demonstrate in the deterministic case (\ref{deterministic factor discovery objective}) with $z_{\theta}$.

In the first scenario, only a subset of the sample has labels. In this case, we have a labeled sample $\{x_i^1,z_i\}_{i=1}^N$ and an unlabeled sample $\{x_j^2\}_{j=1}^M$. Then, problem (\ref{deterministic factor discovery objective}) assigns labels or latent variables $z_j$ to the unlabeled sample through
\begin{align}
\label{semisupervised factor discovery objective}
\begin{split}
\sup_{\theta} \inf_{\tau} \sup_{\xi} L(\theta,\tau,\xi)
=& \frac{\lambda}{N} \sum_{i=1}^N \big[ c\big(x_i^1,T_{\tau}(x_i^1,z_i)\big) - \psi^Y_{\xi} \big(T_{\tau}(x_i^1,z_i)\big) \tilde{\psi}^Z_{\xi} \big(z_i\big) \big]\\
&+\frac{1-\lambda}{M} \sum_{j=1}^M \big[ c\big(x^2_j,T_{\tau}(x^2_j,z_{\theta}(x^2_j))\big) - \psi^Y_{\xi} \big(T_{\tau}(x^2_j,z_{\theta}(x^2_j))\big) \tilde{\psi}^Z_{\xi} \big(z_{\theta}(x^2_j)\big) \big]\\
\tilde{\psi}^Z(z) :=& \psi^Z(z) - \frac{\lambda}{N} \sum_{i=1}^N \psi^Z(z_i) - \frac{1-\lambda}{N} \sum_{j=1}^M \psi^Z(z_{\theta}(x^2_j))
\end{split}
\end{align}
where $\lambda \in (0,1)$ is a constant that weights the two samples.

In the second scenario, only some of the labels $z^1$ are provided, while others $z^2$ are hidden. This scenario was introduced in \cite{tabak2018explanation} as factor discovery with confounding variables: we are given a labeled sample $\{x_i,z_i^1\}$ and need to discover a hidden label $\{z_i^2\}$. For instance, $x_i$ can be climate data, $z^1_i$ can be time of the year, and the uncovered $z^2_i$ could correspond to a priori unknown climate patterns such as El Ni\~no. Then, (\ref{deterministic factor discovery objective}) assigns the label $z^2$ through
\begin{align*}
\begin{split}
\sup_{\theta} \inf_{\tau} \sup_{\xi} L(\theta,\xi,\tau) &= \frac{1}{N} \sum_{i=1}^N \Big[ c\big(x_i,T_{\tau}(x_i,(z^1_i,z^2_{\theta}(x_i)))\big) - \psi^Y_{\xi}\big(T_{\tau}(x_i,(z^1_i,z^2_{\theta}(x_i)))\big) \tilde{\psi}_{\xi}^Z \big( (z^1_i,z_{\theta}^2(x_i)\big) \Big]\\
\tilde{\psi}^Z(z) &:= \psi^Z(z) - \frac{1}{N} \sum_{i=1}^N \psi^Z\big((z_i^1,z_{\theta}^2(x_i))\big).
\end{split}
\end{align*}

\subsection{Optimization}
We train BaryNet by gradient descent. Since (\ref{data-based barycenter objective}) and (\ref{data-based factor discovery objective}) are min-max problems, we need optimization algorithms that are guaranteed to converge to the saddle points. Note that na\"{i}ve methods such as gradient descent-ascent fails to converge \cite{essid2019implicit} even for elementary problems such as $\inf_x \sup_y x\cdot y$.
Among the known saddle point algorithms, we adopt the optimistic mirror descent (OMD) \cite{mertikopoulos2019optimistic}, because it is straightforward to implement, and the quasi implicit twisted descent (QITD) \cite{essid2019implicit}, because it automatically adjusts the learning rate and can accelerate training.
For reference, the algorithms OMD and QITD are listed in Appendix \ref{appendix: saddle point algorithms}.

It is common for saddle point algorithms to assume some convexity condition, e.g. the objective function is (locally) quasiconvex-quasiconcave \cite{essid2019implicit}. Then, minimax theorems such as Sion's \cite{simons1995minimax} imply that (locally) the minimization and maximization can be interchanged. In fact, one can check that neither OMD nor QITD discriminate between
$$\inf_{x} \sup_y F(x,y)$$
and
$$\sup_{y} \inf_{x} F(x,y) = -\inf_y \sup_x \left(-F(x,y)\right),$$
since their update steps for the two problems are the same. Hence, whenever we apply these saddle point algorithms, we are allowed to exchange the inf and sup, so the factor discovery problem (\ref{data-based factor discovery objective}) can be modified into
\begin{equation}
\label{minmax factor discovery}
\sup_{\theta} \inf_{\tau} \sup_{\xi} L(\theta,\tau,\xi) = \inf_{\tau} \sup_{\theta,\xi} L(\theta,\tau,\xi)
\end{equation}

\begin{remark}
It is possible to avoid the min-max in problem (\ref{data-based barycenter objective}) by using the maximum mean discrepancy (MMD) \cite{gretton2012kernel}. The max in (\ref{data-based barycenter objective}) arises from the discriminators $\psi^Y(y)\psi^Z(z)$, which enforce the independence $\pi_{YZ} = \pi_Y \otimes \pi_Z$. We can instead penalize the MMD between $\pi_{YZ}$ and $\pi_Y\otimes\pi_Z$: let $k(y_1,y_2),h(z_1,z_2)$ be characteristic kernels on $Y,Z$, then define
%
\begin{align*}
MMD(\pi_{YZ},\pi_Y\otimes\pi_Z)
&= \E_{[\pi_{YZ}-\pi_Y\otimes\pi_Z]^{\otimes 2}(y,z,y',z')}\big[k(y,y')h(z,z')\big]\\
&\approx \frac{1}{N(N-1)}\sum_{i=1}^N \sum_{n=1}^N k(y_i,y_n) h(z_i,z_n) - \frac{2}{N^3}\sum_{i=1}^N \Big( \sum_{m=1}^N k(y_i,y_m) \cdot \sum_{n=1}^N h(z_i,z_n) \Big)\\
& \quad + \frac{1}{N^2(N^2-1)}\sum_{i,j=1}^N k(y_i,y_m) \cdot \sum_{m,n=1}^N h(z_j,z_n)
\end{align*}
where $\{y_i,z_i\}_{i=1}^N$ is a sample of $\pi_{YZ}$. Thus, (\ref{data-based barycenter objective}) simplifies to a minimization problem. A disadvantage, however, is increased computational time. Computing the objective function in (\ref{data-based barycenter objective}) takes only $O(N)$ time, whereas the above estimator for MMD takes $O(N^2)$ time. It is possible to use subsampling to reduce both the sample size of $\pi_{Y}\otimes \pi_{Z}$ and the cost of MMD down to $O(N)$, giving $O(N)$ time in total, but at the expense of increasing the variance of the estimator.
\end{remark}

\section{Test results}
\label{sec: test results}
In the following sections, we test BaryNet on real and artificial datasets. Section \ref{sec: artificial test} uses supervised BaryNet (\ref{data-based barycenter objective}) on synthetic conditional density estimation problems to verify its effectiveness. Section \ref{sec: seismic test} uses unsupervised BaryNet (\ref{data-based factor discovery objective}) on latent factor discovery problems, discovering meaningful hidden variables in climate and earthquake data. Section \ref{sec: climate test} offers a closer look into the functioning of BaryNet, showing how it learns the dependency of data $x_i$ on label $z_i$ and uncovers patterns in climate data. Finally, Section \ref{sec: color transfer} applies the transport maps $T_z,S_z$ to color transfer. All tests were conducted using \texttt{PyTorch}. See Appendix \ref{appendix: implementation} for the network architecture and training parameters of each experiment.

The BaryNets were trained using either QITD or OMD (Appendix \ref{appendix: saddle point algorithms}). QITD is a second-order method and can automatically adjust its learning rate to accelerate training, but its time complexity is $O(TD^2)$, where $T$ is training time and $D$ is the dimension of the model's parameters. Whereas OMD is a first-order method with time complexity $O(TD)$. In Sections \ref{sec: artificial test} and \ref{sec: climate test}, we use QITD algorithm in order to obtain better convergence. In Sections \ref{sec: seismic test} and \ref{sec: color transfer}, we apply OMD because the networks being trained are large so OMD could be more efficient.

\subsection{Artificial data}
\label{sec: artificial test}
In order to evaluate supervised BaryNet's performance on conditional density estimation, we devise a sample with known conditional distributions. Set $X=Y=\R^2, Z=\R$, and $c=\|x-y\|^2$. The data $\{x_i,z_i\}$ consists of 500 points drawn from the distribution $\rho(x|z)v(z)$ where $v(z)$ is uniform over $[-1,1]$ and each $\rho(x|z)$ consists of a mixture of two Gaussians,
$$\rho(x|z) = \frac{1}{2}\mathcal{N}\left( \begin{bmatrix} (z+1)/2 \\ -(z+1)/2 \end{bmatrix},
\begin{bmatrix}0.1 & \\  & 0.1 \end{bmatrix}
\right) +
\frac{1}{2}\mathcal{N}\left( \begin{bmatrix}-(z+1)/2 \\ (z+1)/2 \end{bmatrix},
\begin{bmatrix}0.1 & \\  & 0.1 \end{bmatrix} \right).$$
%
%
Below are the results of applying supervised BaryNet (\ref{data-based barycenter objective}) with the QITD algorithm.
\begin{figure}[H]
\centering
\subfloat{\includegraphics[scale=0.3]{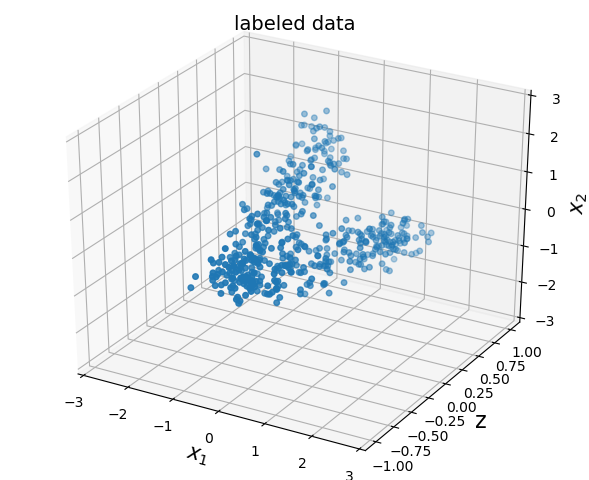}}
\subfloat{\includegraphics[scale=0.3]{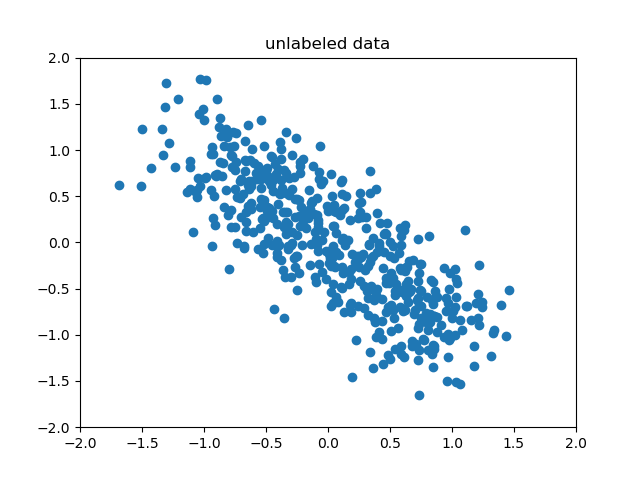}}
\subfloat{\includegraphics[scale=0.3]{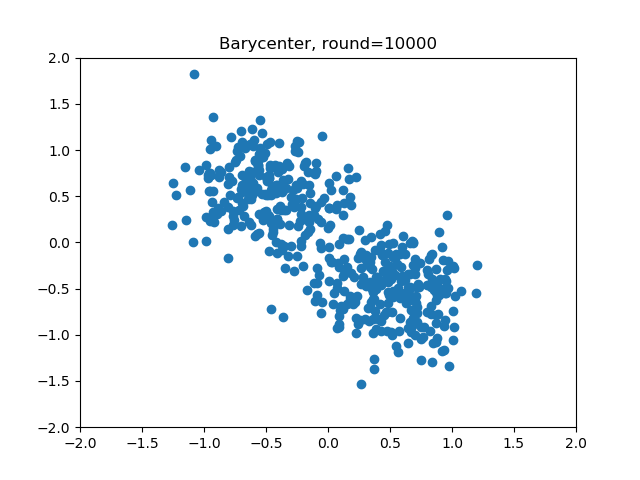}}
\caption{Left: the labeled sample $\rho(x,z) \approx \{x_i,z_i\}$. Middle: the $X$ marginal, $\rho(x)\approx\{x_i\}$. Right: the barycenter $\mu\approx\{y_i = T(x_i,z_i)\}$ produced by BaryNet.}
\end{figure}
Then, the inverse transport map $S(y,z)$ is computed from (\ref{inverse transport by regression}) using SGD. Below, the conditional distributions $\rho(x|z) \approx \{S(y_i,z)\}$ recovered by BaryNet (in orange) are compared with samples of the same size drawn from the true distribution $\rho(x|z)$ (in blue).
\begin{figure}[H]
\centering
\subfloat{\includegraphics[scale=0.3]{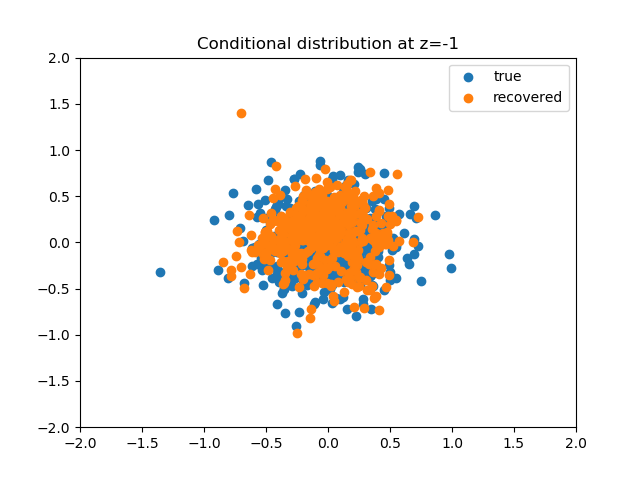}}
\subfloat{\includegraphics[scale=0.3]{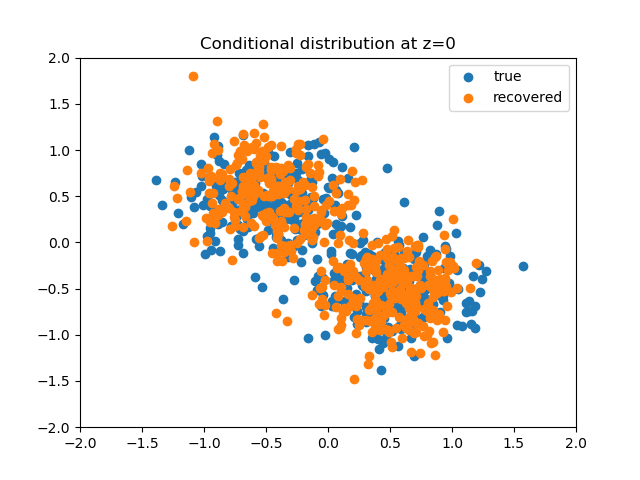}}
\subfloat{\includegraphics[scale=0.3]{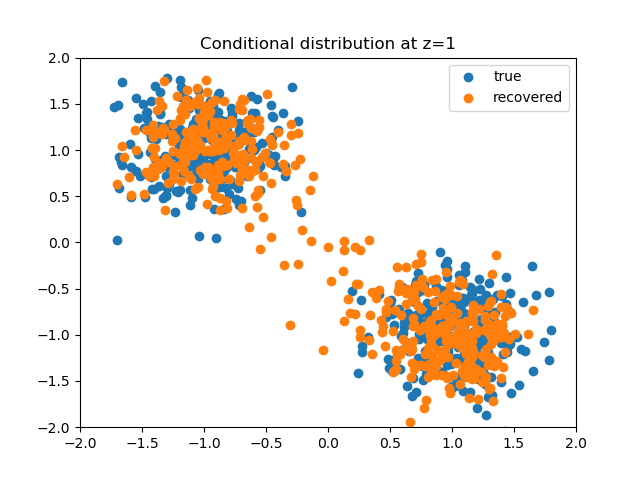}}
\caption{Left: $\rho(x|z=-1)$. Middle: $\rho(x|z=0)$. Right: $\rho(x|z=1)$. All of them show a close match.}
\end{figure}
\noindent
We see that BaryNet can reliably recover the conditional distributions $\rho(x|z)$. In particular, its performance does not deteriorate in the extreme cases with $z=\pm 1$, at the two endpoints of the support of $v(z)$.

The plot below displays the evolution of the objective $L$ from (\ref{data-based barycenter objective}) during training.
\begin{figure}[H]
\centering
\includegraphics[scale=0.66]{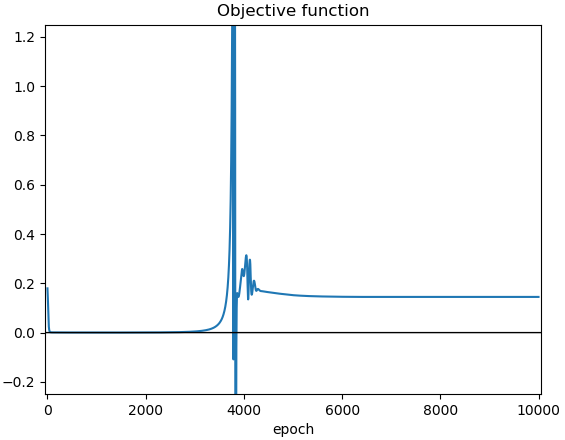}
\caption{The objective function $L$ during training.}
\end{figure}
As the test function $\psi^Y\psi^Z$ is initially set to zero, the cost term $\E [c(x,T(x,z))]$ in (\ref{data-based barycenter objective}) dominates, so the transport map remains near the identity $T(x,z) \approx x$. Then, as $\psi^Y\psi^Z$ is trained and becomes increasingly discriminative, the objective $L$ rises. The transport map $T$ responds and the training enters a brief oscillatory period, corresponding to a competitive stage of the ``game'' performed by the map and the discriminator. When BaryNet converges to the right solution, $L$ becomes flat. Faster convergence can be achieved by preconditioning $\psi$ so that it acts from the very beginning, or implementing a two-time-scale training scheme, with the test function $\psi^Y\psi^Z$ in $\inf_T \sup_{\psi^Y\psi^Z} L$ trained faster than $T$ \cite{lin2019gradient}.

\subsection{Continental climate and seismic belt}
\label{sec: seismic test}

To evaluate unsupervised BaryNet's performance on latent factor discovery, we apply it to real-world data that has a meaningful latent variable and test whether BaryNet can discover it. The first dataset is the average daily temperature recorded from 56 stations across U.S. in the ten-year period $[2009,2019)$, provided by NOAA \cite{NOAA2019daily}. The sample space is $X=Y=\R^{56}$ and each $x_i$ represents the temperature distribution in U.S. at a particular date. The cost is set to be $c=\|x-y\|^2$. An intuitive latent variable would be the seasonal effect, represented by the time of the year
\begin{equation}
\label{time of year}
\cos\Big[ \frac{2\pi}{365}(\text{date}-n) \Big],
\end{equation}
where $n$ is the coldest day in the year (around January 15 in U.S.). Thus, the latent space is $Z=\R$. We apply unsupervised BaryNet (\ref{deterministic factor discovery objective}) in its min-max formulation (\ref{minmax factor discovery}) and train it by the OMD algorithm. As argued in Section \ref{sec: label net}, we restrict the label net $z_{\theta}(x)$ to be Lipschitz, using the clamp function introduced by \cite{arjovsky2017wasserstein}, and set its last layer to be bias-free.

Below are the results of BaryNet on the temperature data $\{x_i\}$. The discovered latent variable $\{z_i\}$ exhibits periodicity in time and a strong correlation with (\ref{time of year}), with a Pearson correlation of $0.9686$, indicating that BaryNet has discovered the seasonal effect, from an input that does not contain any information on time. The non-periodic component of the discovered $z$, a multiyear modulation with a scale in the order of four years, is consistent with El Ni\~no Southern Oscilation.

\begin{figure}[H]
\centering
\includegraphics[scale=1.1]{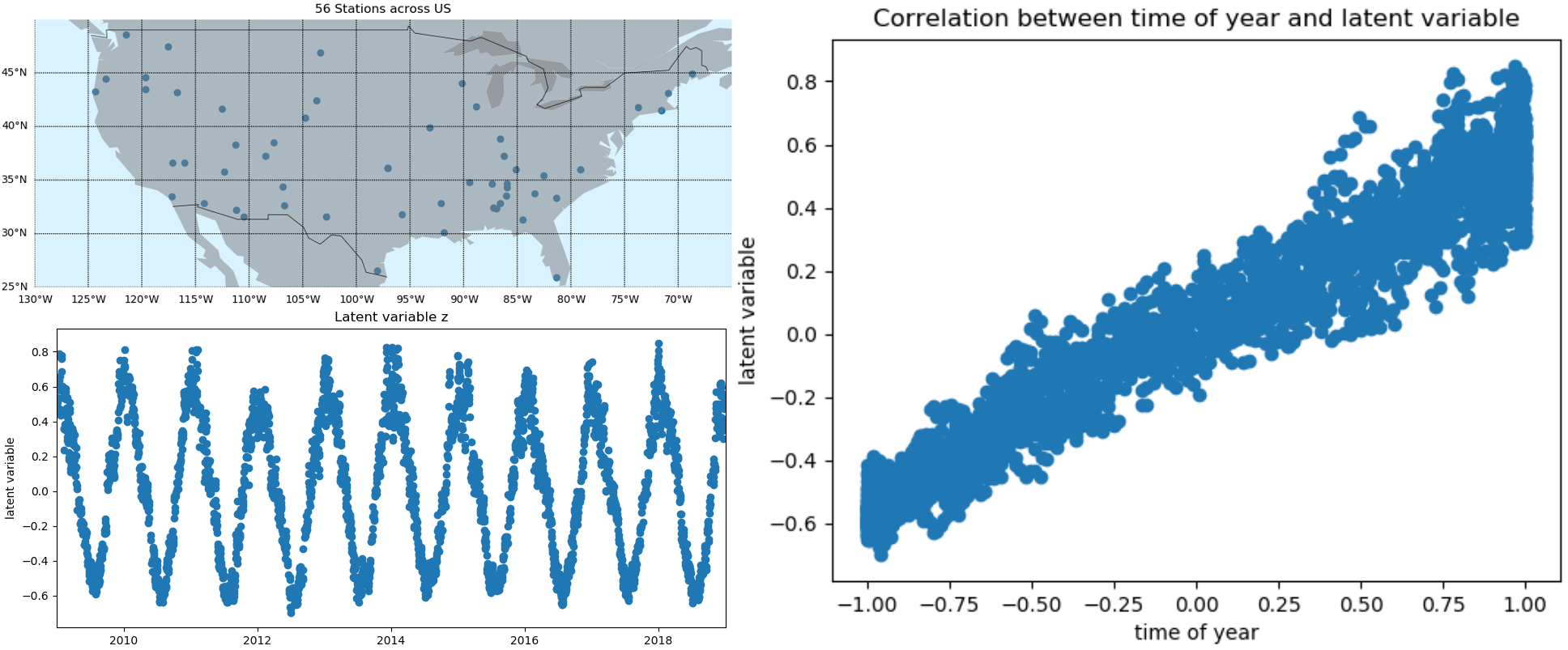}
\caption{Up left: The chosen weather stations. Down left: The latent variables $z_i$ discovered by BaryNet, plotted against time. Right: Scatter plot between the time of the year (\ref{time of year}) and $z_i$.}
\end{figure}

The second dataset consists of earthquakes' coordinates. The data is taken from USGS \cite{USGS2008earthquake}, which records historical earthquakes from 1900 to 2008. We focus on earthquakes that occurred on the Peru-Chile Trench (see Figure \ref{fig: earthquake test} below). The earthquakes' locations are represented in spherical coordinates, so the sample space has $X=Y=S^2$, and we set the cost function $c$ to be the squared great circle distance
$$c = d^2, ~d\big((x^1,x^2),(x^1_*,x^2_*)\big) = \arccos\big(\sin x^2 \sin x^2_* + \cos x^2 \cos x^2_* \cos(x^1-x^1_*)\big)$$
where $x^1,x^2$ are longitude and latitude. Judging from the earthquake plot in Figure \ref{fig: earthquake test}, since the earthquakes are distributed roughly vertically along the Peru-Chile Trench, it is intuitive that the (one-dimensional) latent variable should be proportional to the earthquake's latitude.

We apply unsupervised BaryNet (\ref{minmax factor discovery}) and train it by OMD. The discovered latent variable $\{z_i\}$ shows a strong correlation with the latitude, with a correlation of $0.9954$, indicating that BaryNet has discovered the seismic belt.
\begin{figure}[H]
    \centering
    \subfloat{\includegraphics[scale=0.52]{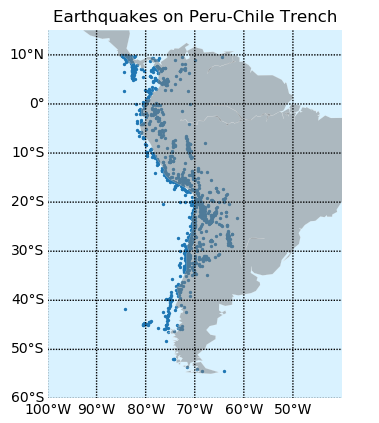}}
    \subfloat{\includegraphics[scale=0.5]{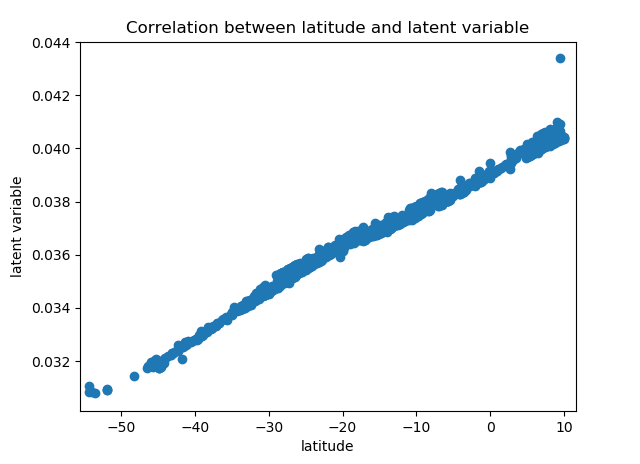}}
    \caption{Left: Earthquakes on Peru-Chile Trench. Right: Scatter plot between the earthquake's latitude and the latent variables $z_i$ discovered by BaryNet.}
    \label{fig: earthquake test}
\end{figure}

The label nets $z_{\theta}(x)$ for both the temperature and the earthquake tests are intentionally set to be feedforward. Unlike the residual nets, these are highly non-linear maps without any linear component, and it is difficult for them to learn linear mappings such as $(x^1,x^2)\mapsto x^2$, making it highly unlikely that BaryNet arrived at the desired solutions by chance.

\subsection{Hourly and seasonal temperature variation}
\label{sec: climate test}
As discussed in the introduction, supervised BaryNet ``learns" the data $\rho(x,z)$ by decomposing it into a representative $\mu$ of the conditional densities $\rho(x|z)$ plus the transformations between them. Equivalently, it represents $\rho(x,z)$ as $S\# (\mu \otimes v)$.
Thus, BaryNet can be seen as a ``probabilistic" generalization of regression, which learns the probabilistic mapping $z\mapsto \rho(x|z)$,  approximating it via
\begin{equation}
\label{probabilitic regression}
z\mapsto (S_{\theta}(\cdot,z)\circ T_{\tau})\#\rho = (S_{z})_{\theta}\#\mu_{\tau}.
\end{equation}
Its expressivity derives from the nonlinearity of $T_{\tau}$ and $S_{\theta}$, which enables BaryNet to learn the complex dependency of the data $x$ on the latent variable $z$.

We demonstrate this intuition using meteorological data where $x_i$ is the average hourly temperature at Ithaca, NY in the ten-year period $[2007,2017)$, provided by NOAA \cite{NOAA2019hourly}. The latent variables $z_i$ chosen are the time of day and time of year, represented by
\begin{equation}
\label{climate test feature functions}
\sin(2\pi \text{hour}/24), ~\cos(2\pi \text{hour}/24), ~\sin(2\pi \text{date}/365), ~\cos(2\pi \text{date}/365)
\end{equation}
Hence $X=Y=\R$, $Z=\R^4$, and we set  $c=(x-y)^2$.

BaryNet (\ref{data-based barycenter objective}) is trained on $\{x_i,z_i\}$ by the QITD algorithm. To facilitate visualization, the probabilistic regression (\ref{probabilitic regression}) is displayed through its conditional mean:
$$z \mapsto \E_{(S_z\#\rho)(x)}[x] \approx \frac{1}{N} \sum_{i=1}^N S_{\theta}(y_i,z)$$

We compare BaryNet with mean-square regression on $\{x_i,z_i\}$. The regression problem fits $\rho(x|z)$ under mean-square loss, and the optimizer is also the conditional mean $\E_{\rho(x|z)}[x]$, so it constitutes a valid comparison. Specifically, we perform linear regression, using the latent variables $z$ (\ref{climate test feature functions}) as features. It may seem unfair to compare the nonlinear BaryNet with a linear model, but the key is that BaryNet is solving the more difficult problem of learning the entire $\rho(x|z)$, whereas regression only learns $\E_{\rho(x|z)}[x]$, so we are satisfied as long as BaryNet performs at least as well as linear regression in terms of the conditional mean.
\begin{figure}[H]
\centering
\includegraphics[scale=1.5]{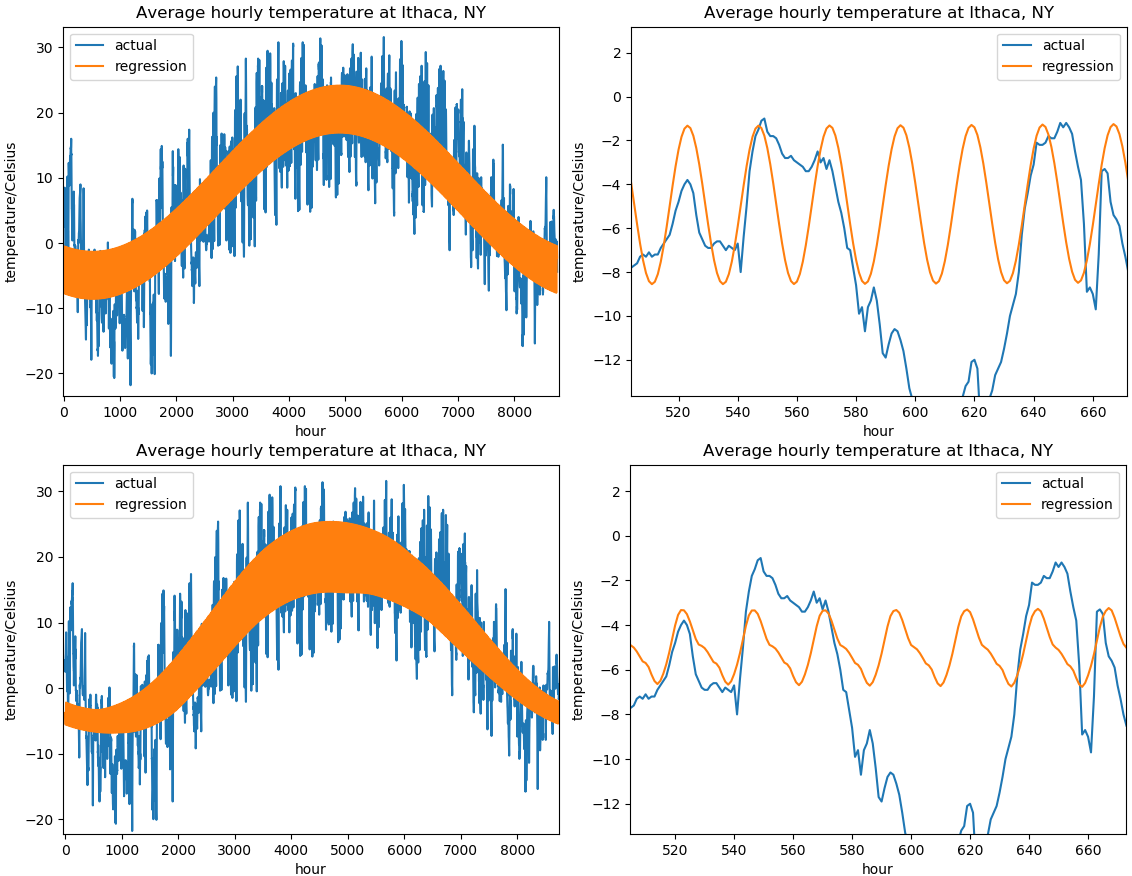}
\caption{Hourly temperature at Ithaca, NY. Left column: data for the year 2007. Right column: data for the week of Jan 21-28, 2007. Top row: Linear regression on $\{x_i,z_i\}$. Bottom row: Supervised BaryNet. The blue curve represents the temperature data, and the orange curve the regression  (for linear regression) and the conditional mean (for BaryNet). (Due to the daily oscillation, the regression curve appears as a thick band in the yearly plot). The regression curves cannot perfectly fit the data, since date and hour alone cannot fully account for the irregularity of weather systems.}
\label{fig: climate regression}
\end{figure}
BaryNet learns more features of the data than linear regression in both the yearly and weekly plots. In the yearly plot, BaryNet's curve oscillates with greater amplitude during summer than in winter, indicating that the daily temperature during summer has greater variance. In the weekly plot, BaryNet's curve is highly non-sinusoidal, and has greater upward slope during the day than downward slope at night, in agreement with the real daily cycle during winter time.

\subsection{Color transfer}
\label{sec: color transfer}
We describe briefly here a useful application of the transport maps. As a by-product of BaryNet, the transport map $T_z(x)$ and inverse transport map $S_z(y)$ can be concatenated into a transportation between any pair of conditional distributions:
\begin{equation*}
(S_{z_2} \circ T_{z_1}) \# \rho(x|z_1) = \rho(x|z_2)
\end{equation*}
Even though we can directly compute a transport map from $\rho(x|z_1)$ to $\rho(x|z_2)$, BaryNet is more efficient if one seeks all pairwise transport maps (just as we may seek all pairwise translations among several languages). If there are $K$ labels ($Z$=\{1,\dots K\}), then there will be $O(K^2)$ pairwise transport maps, whereas BaryNet only needs to compute $2K$ maps, $T_k$ and $S_k$. If $Z$ is continuous (e.g. $\R^k$), then BaryNet only needs two maps, $T(x, z)$ and $S(y, z)$, while doing individual pairwise maps is infeasible, not only because there are infinitely many of them, but also because each conditional distribution has one sample point or less.

Instead of languages, we apply this procedure to images. An image can be seen as a 3-dimensional matrix of size $3 \times H \times W$, which represent the RGB color channels, height and width (in terms of pixels). Alternatively, an image can be treated as a sample of $H \times W$ points in $\R^3$. Thus, we view the following images as color distributions, $\rho_1,\rho_2,\rho_3\in P(\R^3)$, and apply BaryNet to compute the transport maps $T_k,S_k$.
\begin{figure}[H]
\centering
\subfloat{\includegraphics[scale=0.4]{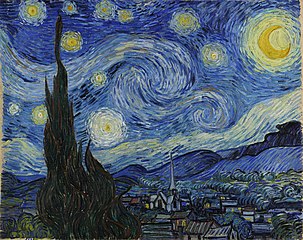}}
\subfloat{\includegraphics[scale=0.4]{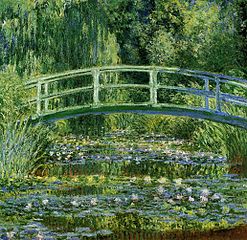}}
\subfloat{\includegraphics[scale=0.4]{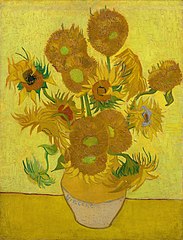}}
\subfloat{\includegraphics[scale=0.3]{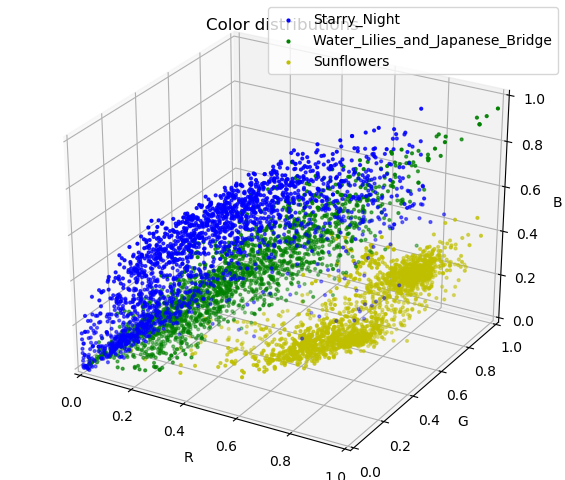}}
\caption{``Starry Night", ``Water Lilies and Japanese Bridge", ``Sunflowers" by Monet and van Gogh, and their color distributions in $\R^3$.}
\end{figure}

Then, the coloring style of image $j$ can be transferred to image $k$ by applying pixel-wise the transport map $S_j \circ T_k$ to image $k$. The results are displayed in Figure \ref{fig: color transfer} below. We used supervised BaryNet (\ref{data-based barycenter objective}) with $Z=\{1,2,3\}$ and trained it by OMD.
\begin{figure}[h]
\centering
\subfloat{\includegraphics[scale=0.12]{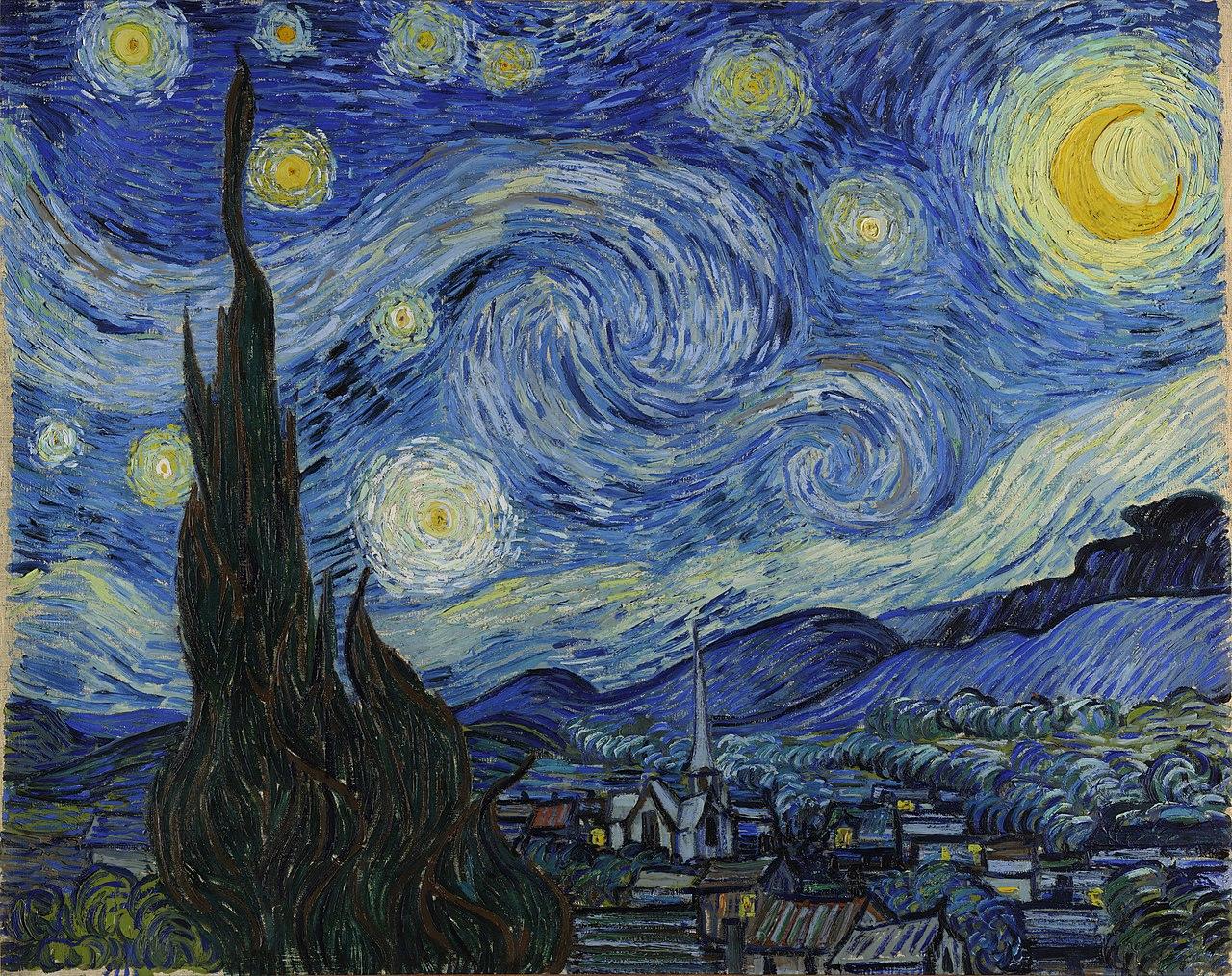}}
\subfloat{\includegraphics[scale=0.122]{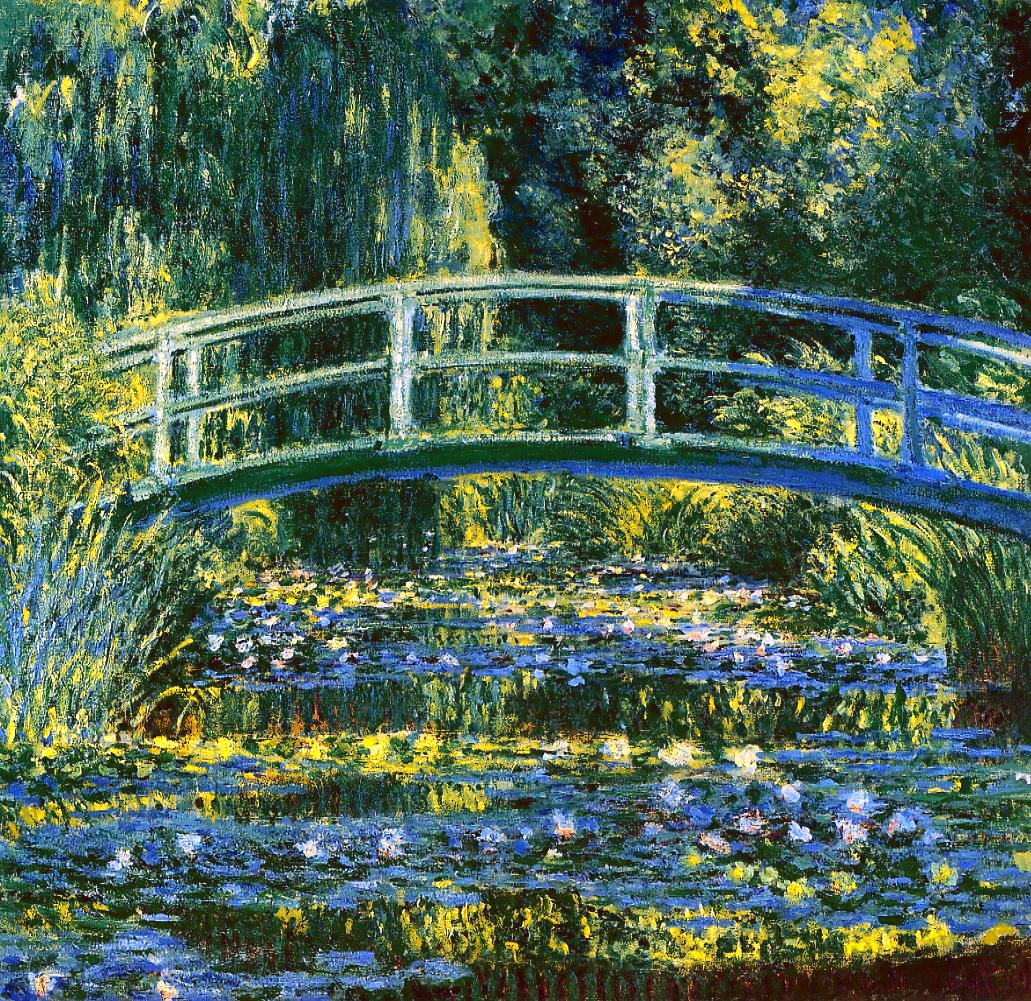}}
\subfloat{\includegraphics[scale=0.12]{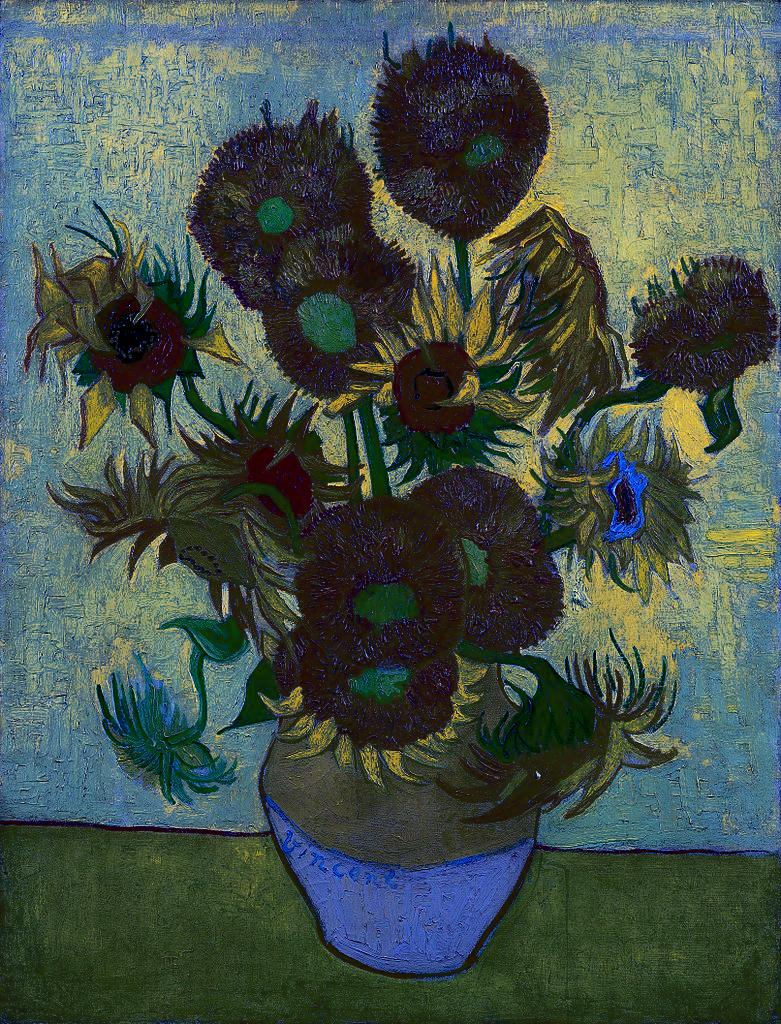}}
\vspace{-3ex}
\centering
\subfloat{\includegraphics[scale=0.12]{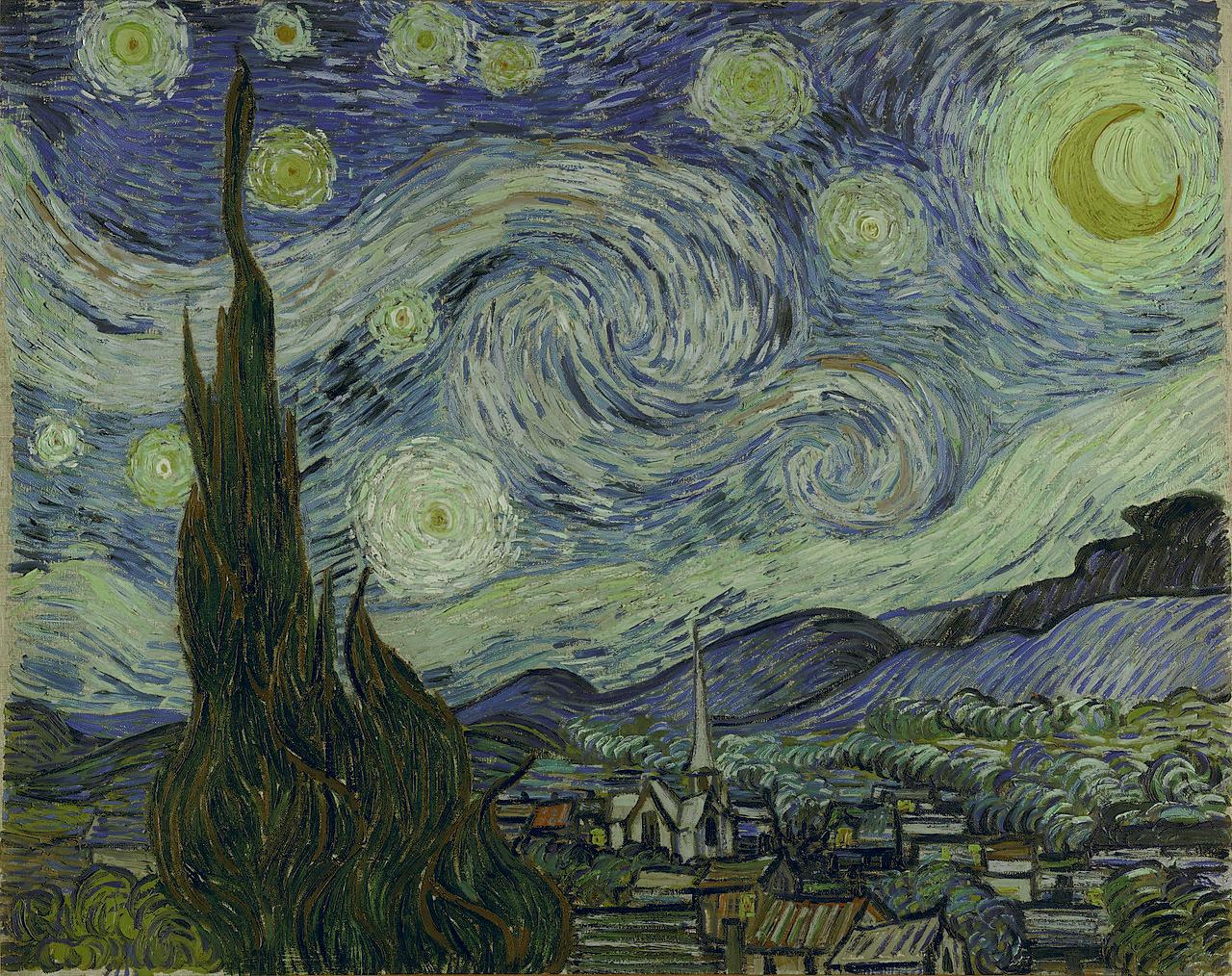}}
\subfloat{\includegraphics[scale=0.122]{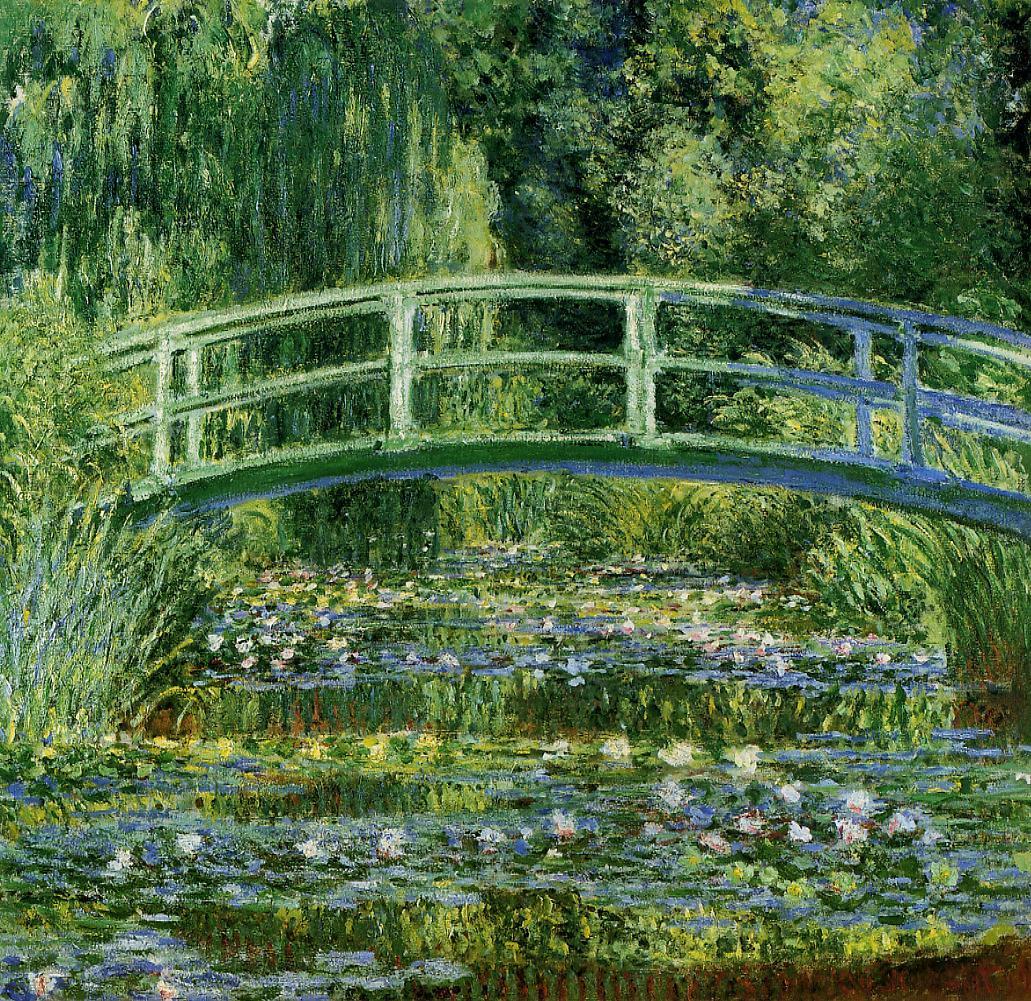}}
\subfloat{\includegraphics[scale=0.12]{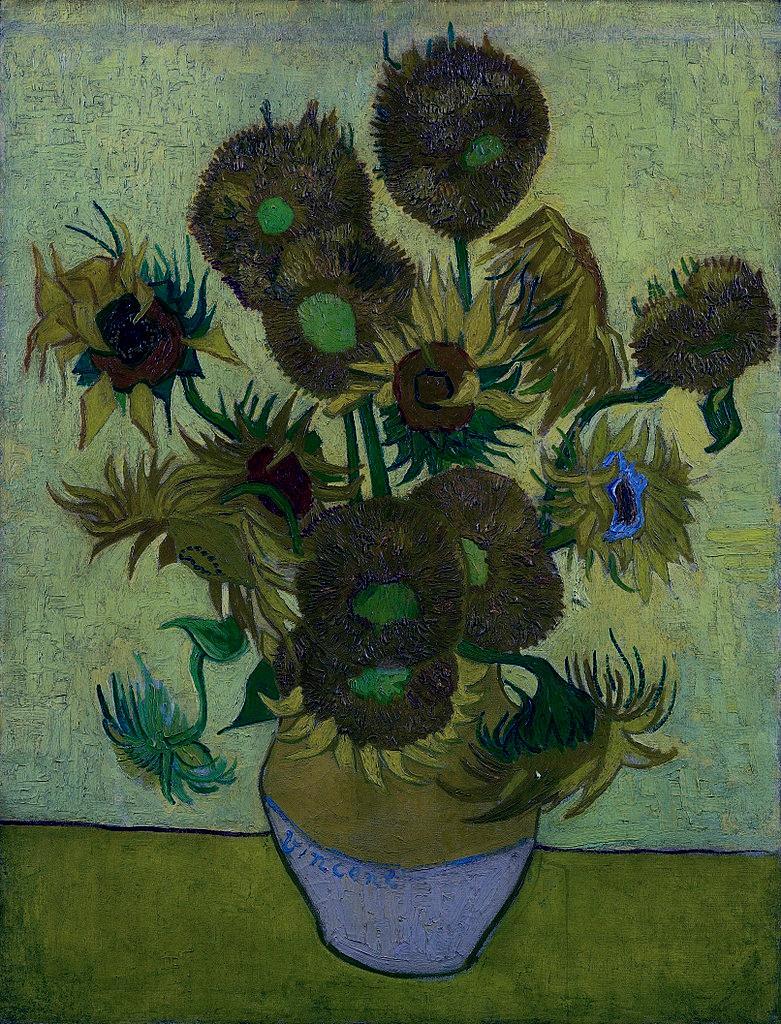}}
\vspace{-3ex}
\centering
\subfloat{\includegraphics[scale=0.12]{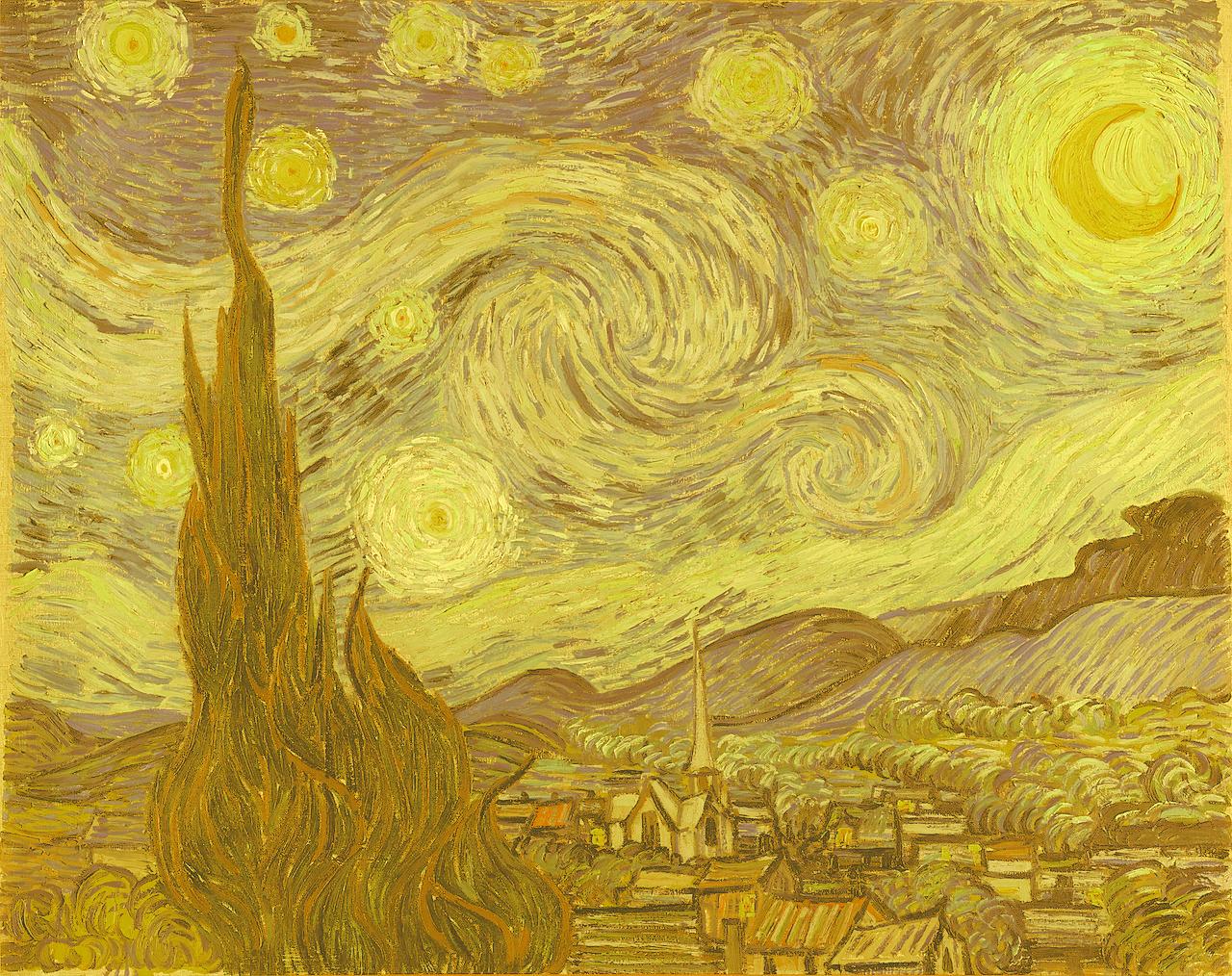}}
\subfloat{\includegraphics[scale=0.122]{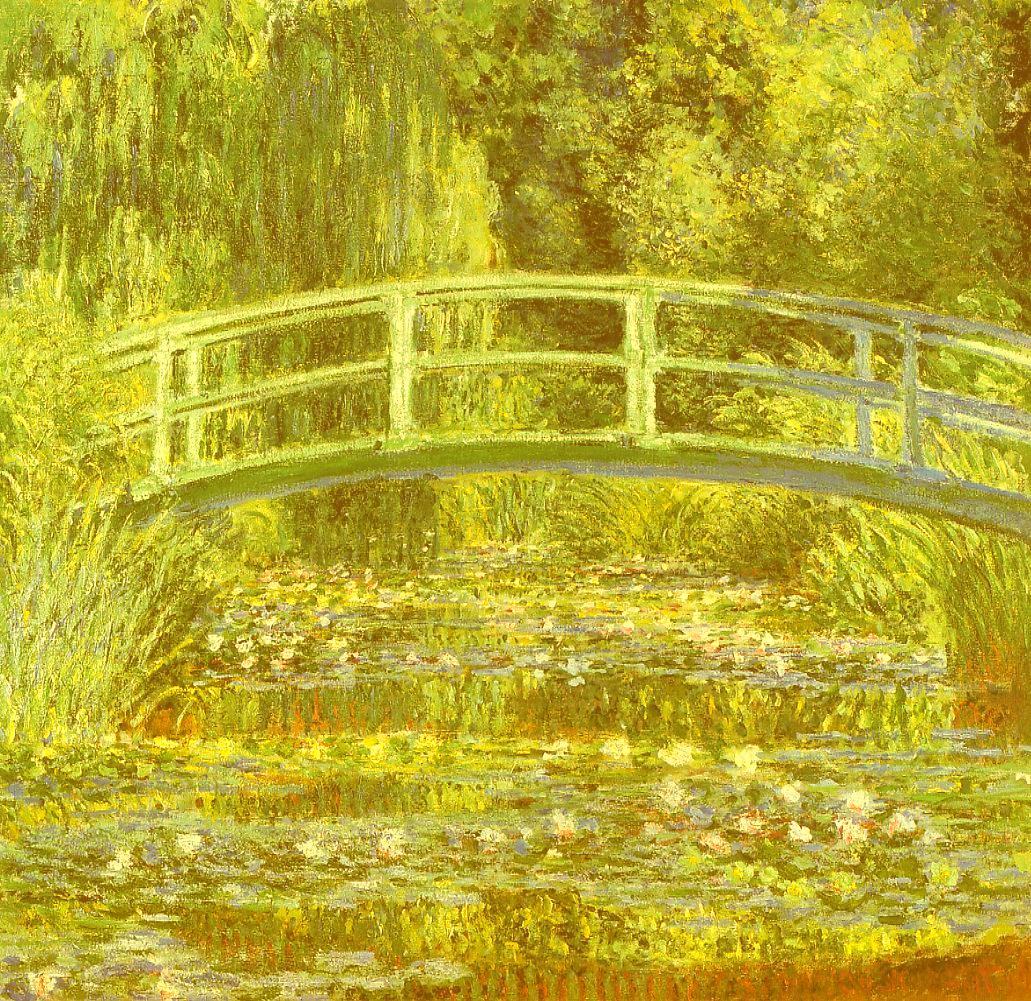}}
\subfloat{\includegraphics[scale=0.12]{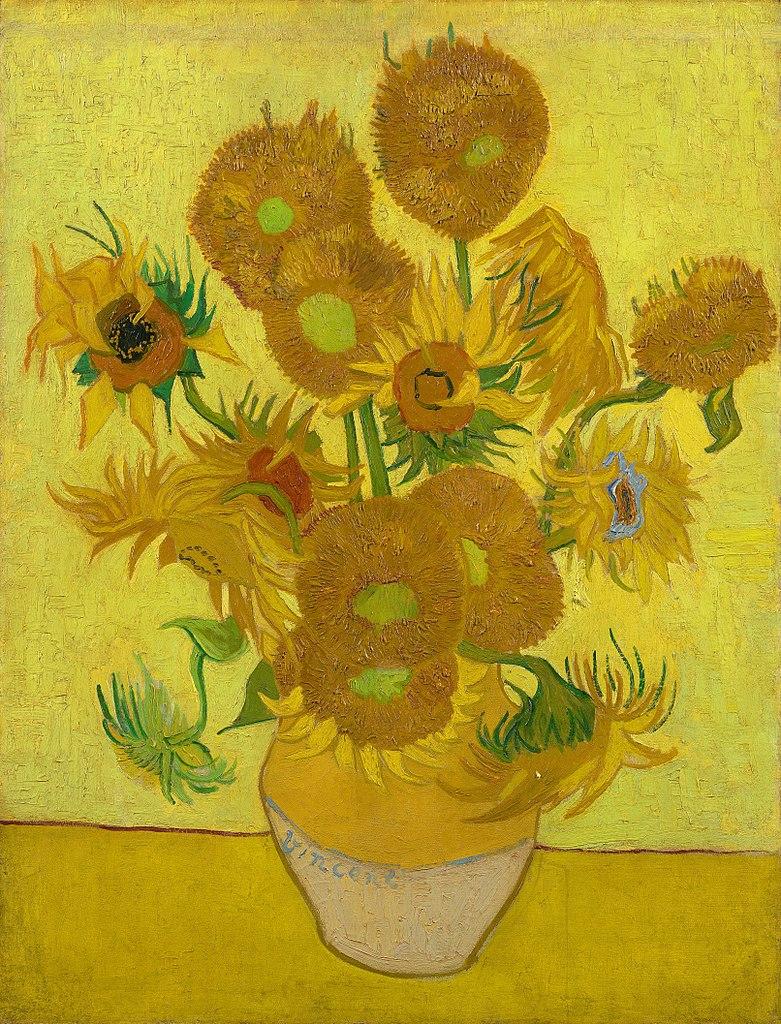}}
\caption{Color transferred images. The $k$-th column and $j$-th row is obtained by $S_j \circ T_k \# \rho_k$, so it is image $k$ with the color of image $j$.}
\label{fig: color transfer}
\end{figure}

\begin{remark}
A task very similar to color transfer is \textit{color normalization} \cite{li2015complete}, that seeks to eliminate the differences in several images' color distributions with minimum distortion. For instance, in medical imaging, we may want to remove irrelevant variations, such as different lighting conditions or staining techniques. BaryNet can easily solve this task by mapping the color distributions $\rho_k$ to their common barycenter, which by definition minimizes total distortion.
\end{remark}

\section{Conclusions}
\label{sec: conclusion}

Conditional density estimation and latent variable discovery are two intimately related problems in machine learning: one learns the dependence of data $x$ on a given variable $z$, while the other infers a latent variable $z$ from data $x$. This paper proposes to solve both problems in the framework of optimal transport barycenters. Our method is based on an intuitive principle of minimum uncertainty, that is, the goal of learning is to reduce some measure of uncertainty or variability. Specifically, for latent variable discovery, we begin with some data $\rho(x)$ and our learning ends with a joint distribution $\rho(x,z)$, which assigns the latent variables $z$ to each data point $x$ through $\rho(z|x)$. Our principle leads to maximizing the reduction in variability from $\rho(x)$ to $\rho(x,z)$.

How should we characterize the variability of a joint distribution $\rho(x,z)$? A simple approach is to take the mean variability of the conditional distributions $\rho(x|z)$. Yet, this approach does not always lead to sensible results: the $k$-means algorithm, for example, minimizes the sum of squared errors, or equivalently the weighted average of each cluster's variance, and it is known that it often fails to recognize clusters with different sizes and shapes \cite{yang2019clustering}.

Instead, we seek a distribution $\mu$ that can act as a representative of all $\rho(x|z)$, and measure the variability of $\rho(x, z)$ by that of $\mu$. This idea naturally leads to the optimal transport barycenter, which minimizes a user-specified distance between $\mu$ and each $\rho(x|z)$. A characterization of variability arises from the barycenter's optimal transport cost, indicating that variability and transport cost are two sides of the same coin. Under simplifying assumptions \cite{tabak2018explanation,yang2019clustering}, this definition of $\rho(x,z)$'s variability includes the aforementioned simple approach as special case.

It follows that latent variable discovery should seek assignments $\rho(z|x)$ that minimize the variability of the barycenter of the $\rho(x|z)$. At the same time, conditional density estimation also benefits from the barycenter representation. The difficult task of learning the possibly infinitely many $\rho(x|z)$ is reduced to learning just the barycenter $\mu$, from which we can recover each $\rho(x|z)$.

The contributions of this paper are
\begin{enumerate}
\item The introduction of the BaryNet algorithms, which use neural nets. The unsupervised BaryNet performs latent variable assignment $\rho(z|x)$, while the supervised BaryNet computes the barycenter $\mu$ and estimates each conditional distribution $\rho(x|z)$. Their effectiveness is confirmed by tests on artificial and real-world data, with Euclidean and non-Euclidean costs.

\item Enrichment of the theory of optimal transport barycenters. In particular, the existence of Kantorovich and Monge solutions for barycenters with infinitely many $\rho(x|z)$  are studied (Theorems \ref{thm: well-posedness and existence} and \ref{thm: barycenter transport map}), and geometric properties of the Wasserstein space are discovered that resemble those of Euclidean space (Theorems \ref{thm: variance decomposition} and \ref{thm: Gaussian covariance}).

\item An intimate connection between autoencoders and BaryNet is identified. In particular, with squared Euclidean distance cost and the simplifying assumption that $\rho(x|z)$ are equivalent up to translation, BaryNet includes the following algorithms as special cases: $k$-means, PCA, principle curves and surfaces, and undercomplete autoencoders.

\end{enumerate}

The theoretical framework developed in this article opens up several new directions of research:

\begin{enumerate}

\item Parallelism to autoencoders. We proposed the Barycentric autoencoder (BAE) algorithm in Section \ref{sec: autoencoder}, based on the parallelism between autoencoders and unsupervised BaryNet. It would be interesting to compare the performance of BAE with that of VAE, WAE or AAE. One can also apply the regularizations of denoising autoencoders and sparse autoencoders \cite{goodfellow2016deep} to BaryNet.

\item Cooperation for density estimation. Given labeled data $\rho(x,z)$ and any density estimation algorithm or generative modeling algorithm, such as WGAN \cite{arjovsky2017wasserstein}, one can check whether the algorithm's performance can be improved by first estimating the density of the barycenter $\mu$, and then transporting to each $\rho(x|z)$ through BaryNet, instead of learning the joint distribution $\rho(x, z)$ directly.

\item Transfer learning and domain adaptation. The semisupervised BaryNet introduced in Section \ref{sec: semi-supervised factor discovery} can be applied to solve transfer learning problems. For instance, given some unlabeled data $\{x_j'\}$ and a few labeled data $\{x_i,z_i\}$ (such that $x_i$ and $x_j'$ may be drawn from different distributions), we can perform classification on $\{x_j'\}$ based on the information of $\{x_i, z_i\}$. Then, the semisupervised BaryNet (\ref{semisupervised factor discovery objective}) produces a labeling in accordance with the principle of minimum uncertainty.

\item Metric learning: When a clustering plan or label assignment $\rho(x,z)$ is provided, we can try to infer what metric or cost function is responsible for that particular assignment of $z$. As an example, let $X$ be the space of images, and $\rho(x,z)$ be some image dataset whose labels are text descriptions. Since human vision is better tuned to discerning faces than inanimate objects such as buildings, there would be more labels related to faces than to buildings. Specifically, $X$'s subspace of facial images would have a greater density of labels $z$ than the subspace of buildings' images, or equivalently, if label $z_1$ concerns facial feature (e.g. ``smiling face") and $z_2$ concerns building's feature (e.g. ``old building"), then the Euclidean variance of $\rho(x|z_1)$ is most likely smaller than that of $\rho(x|z_2)$. Assume that $g$ is a Riemannian metric on $X$, such that $\rho(x,z)$ becomes an optimal solution to the factor discovery problem (\ref{factor discovery objective}) with squared geodesic distance cost $c=d^2(x,y)$ induced by $g$, then $(X,g)$ should not be Euclidean, and instead $g$ should assign greater distances to the subspace of faces, so that factor discovery adapts to $g$ by making $\rho(x|z_1)$ ``smaller" than $\rho(x|z_2)$.
Note that, (\ref{factor discovery objective}) evaluates any assignment $\rho(x,z)$ when given a cost $c(x,y)$; alternatively, it might be modified so as to evaluate any candidate cost or metric when given an assignment.

\end{enumerate}

\section*{Acknowledgments} The work of E. G. Tabak was partially supported by NSF grant DMS-1715753 and ONR grant N00014-15-1-2355.

\bibliographystyle{acm}
\bibliography{main}

\centerline{\bf \Large Appendices}

\appendix
\section{Assumptions}
\label{appendix: assumptions}

\begin{assumption}
\label{assumption: coercive and Heine Borel}
The cost function $c$ on $X\times Y$ is \textit{locally uniformly coercive}, that is, for some (and thus every) $y_0 \in Y$ and for every compact $K\subseteq X$,
\begin{equation*}
\lim_{d(y,y_0)\to\infty} \inf_{x\in K} c(x,y) = \infty
\end{equation*}
where the limit is taken over all sequences in $Y$. Also, the space $Y$ satisfies the \textit{Heine-Borel property} that every closed bounded subset is compact.
\end{assumption}

Most cost functions $c$ used in practice are locally uniformly coercive: for instance, given any metric space $(X,d)$, the $l^p$ distance cost $d(x_1,x_2)^p$ with $p\in(0,\infty)$ trivially satisfies the condition.
Also, a broad class of spaces $Y$ satisfy the Heine-Borel property, including Euclidean spaces, complete Riemannian manifolds, and all their closed subsets.

\begin{assumption}
\label{Monge assumption}
Assume that for all $\mu \in P(Y)$ and $v$-almost all $z$, the optimal coupling between $\rho(x|z)$ and $\mu(y)$ is unique and is a Monge solution, i.e. $\pi(x,y|z) = (Id,T_z) \# \rho(x|z)$ for some transport map $T_z$.
\end{assumption}
This assumption holds in most real-world applications. For instance, by Theorem 2.44 of \cite{villani2003topics}, if $X=Y= \R^d$ and the cost $c$ is strictly convex and superlinear, then it holds whenever almost all $\rho(x|z)$ are absolutely continuous. In practice, we are only given samples drawn from unknown probabilities, so we can typically assume that they come from absolutely continuous $\rho(x|z)$.

\section{Proof of Theorem \ref{thm: well-posedness and existence}}
\label{appendix: proof of wellposedness and existence}
Since we constructed the conditional distributions $\rho(x|z)$ using disintegration, the map $Z\to P(X)\times P(Y)$, $z\mapsto (\rho(x|z),\mu)$ is automatically measurable (in the topology of weak convergence) for any given $\mu$. Then, Corollary 5.22 of \cite{villani2008optimal} implies that there is a measurable assignment: $z\mapsto \pi(x,y|z)$ such that each $\pi(x,y|z)$ is an optimal transport plan between $\rho(x|z)$ and $\mu(y)$.

Given that $c$ is bounded below, let $c_n := \min(c,n)$ be a sequence of bounded continuous functions that increases pointwise to $c$. Then, by the monotone convergence theorem, the total transport cost from $\rho(x,z)$ to $\mu(y)$ becomes
\begin{align*}
\int_{Z} I_c(\rho(\cdot|z),\mu) d\nu(z) &= \iint c(x,z)d\pi(x,y|z)d\nu(z) = \lim_{n\to\infty} \iint c_n d\pi(x,y|z) d\nu(z)
\end{align*}
The last term is well-defined since $z\mapsto \int c_n d\pi(\cdot|z)$ is measurable, so the barycenter problem is well-defined.

Factor the measure $\pi(x,y,z)$ into $\pi(x,y|z)\nu(z)$. Then, integrating $\pi$ yields $1$, showing that $\pi$ is a probability measure, and integrating $\pi$ times test functions $\psi\in C_b(X\times Z)$ and $\phi\in C_b(Y\times Z)$ shows that $\pi$ has the marginals: $\pi_{XZ} = \rho(x,z)$ and $\pi_{YZ} = \mu(y)\otimes \nu(z)$. This proves the $\geq$ side of (\ref{joint measure formulation of total transport cost}).
Since the $\leq$ side of (\ref{joint measure formulation of total transport cost}) is evident, the first assertion of Theorem \ref{thm: well-posedness and existence} is proved.\\

To prove the second assertion, we need the following lemmas:

\begin{lemma}
\label{lemma: duality and semicontinuity}
Given any $\mu \in P(Y)$, the total transport cost (\ref{total transport cost}) satisfies the following duality formula
\begin{align}
\label{transport cost primal and dual}
\begin{split}
(\ref{total transport cost}) =&\min_{\substack{\pi \in P(X\times Y\times Z)\\ \pi_{XZ} = \rho\\ \pi_{YZ} = \mu \otimes v}}
\int_{X\times Y\times Z} c(x,y) d\pi(x,y,z)\\
=&
\sup_{\substack{\phi \in C_b(X\times Z)\\ \psi \in C_b(Y\times Z)\\ \phi+\psi \leq c}} \int_{X\times Z} \phi(x,z) d\rho(x,z) + \int_{Y\times Z} \psi(y,z) d\mu(y) d\nu(z),
\end{split}
\end{align}
so that (\ref{total transport cost}) is lower semi-continuous in $\mu$ in the topology of weak convergence of $P(Y)$.
\end{lemma}
\begin{proof}
Define a cost function from $X\times Z$ to $Y\times Z$,
$$\tilde{c}(x,z_1,y,z_2) = c(x,y) + \infty \cdot \delta_{z_1\neq z_2} = \begin{cases}
c(x,y) \text{ if } z_1 = z_2\\
\infty \ \text{ otherwise,}
\end{cases}$$
which is lower semi-continuous on $X\times Z\times Y\times Z$. Then, we can apply Kantorovich duality (Theorem 5.10 of \cite{villani2008optimal}) to $\rho(x,y)$ and $\mu\otimes v$ to obtain the duality formula
$$\min_{\substack{\tilde{\pi}\in P(X\times Z\times Y\times Z)\\
\pi_{XZ_1} = \rho\\
\pi_{YZ_2} = \mu \otimes v}}
\int \tilde{c} ~d\tilde{\pi}
= \sup_{\substack{\phi \in C_b(X\times Z)\\ \psi \in C_b(Y\times Z)\\ \phi+\psi \leq c}} \int_{X\times Z} \phi(x,z) d\rho(x,z) + \int_{Y\times Z} \psi(y,z) d\mu(y) d\nu(z).$$
This is part of the theorem that the minimum on the left side is achieved.

It remains to show that
\begin{equation}
\label{equivalence of primal problems}
\min_{\substack{\tilde{\pi}\in P(X\times Z\times Y\times Z)\\
\pi_{XZ_1} = \rho\\
\pi_{YZ_2} = \mu \otimes v}}
\int \tilde{c} ~d\tilde{\pi}
= \min_{\substack{\pi \in P(X\times Y\times Z)\\ \pi_{XZ} = \rho\\ \pi_{YZ} = \mu \otimes v}}
\int_{X\times Y\times Z} c ~d\pi.
\end{equation}
Define the map $P(x,y,z) = (x,z,y,z)$.
For the $\leq$ part of (\ref{equivalence of primal problems}): given any coupling $\pi(x,y,z)$ that solves the right side, the pushforward
$\tilde{\pi} := P\# \pi$ is applicable to the left side. For the $\geq$ part: if the left side is infinite, then we are done. Else, given an optimal coupling $\tilde{\pi}(x,z_1,y,z_2)$,
\begin{align*}
\int \tilde{c} ~d\tilde{\pi} &= \int c(x,y) + \infty \cdot \delta_{z_1\neq z_2} ~d\tilde{\pi}(x,y|z_1,z_2) d\tilde{\pi}_{Z_1Z_2}(z_1,z_2) < \infty.
\end{align*}
It follows that $\tilde{\pi}_{Z_1Z_2}\big(\{z_1\neq z_2\}\big) = 0$, so the measure $\tilde{\pi}$ must be concentrated on the diagonal $\{z_1=z_2\}$. Then, we can define the pullback $\pi = P^{-1}\#\tilde{\pi}$, which has the correct marginals $\pi_{XZ},\pi_{YZ}$ and is applicable to the right side of (\ref{equivalence of primal problems}):
$$\int c ~d\pi = \int c(x,y) ~d\tilde{\pi}(x,z,y,z) = \int \tilde{c} ~d\tilde{\pi}.$$
Hence, we have proved formula (\ref{transport cost primal and dual}).

Theorem 2.8 of \cite{billingsley1999convergence} implies that for any $\psi \in C_b(Y\times Z)$, the map
$$\mu \mapsto \mu \otimes v \mapsto \iint \psi ~dv d\mu$$
is a continuous linear functional in the topology of weak convergence of $\mu \in P(Y)$. So the right side of (\ref{transport cost primal and dual}), as a supremum over continuous functions, is lower semi-continuous in $\mu$.
\end{proof}

\begin{lemma}
\label{lemma: barycenter problem lower bound}
Given any $\mu(y) \in P(Y)$, the total transport cost (\ref{total transport cost}) has the lower bound:
\begin{equation*}
I_c(\rho(x,z),\mu) \geq I_c(\rho(x),\mu)
\end{equation*}
where we ``forget" the labeling $z$ on the right side.
\end{lemma}
\begin{proof}
This is a direct corollary of Theorem 4.8 of \cite{villani2008optimal}.
\end{proof}

Denote the total transport cost (\ref{transport cost primal and dual}) by $F(\mu)$. Let $\{ \mu^n \} \subseteq P(Y)$ be a minimizing sequence such that $F(\mu^n)$ converges to the optimal transport cost $\inf F$. If $\inf F = \infty$, then any $\mu \in P(Y)$ can serve as a barycenter. Else, assume that $\inf F < \infty$ and choose a constant $C \geq 0$ large enough so that $\sup F(\mu^n) < C$.

We show that Assumption \ref{assumption: coercive and Heine Borel} implies that $\{\mu^n\}$ is tight. Since $Y$ is assumed to satisfy the Heine-Borel property, it suffices to show that, for any $\epsilon>0$, there is some radius $R$ such that
\begin{equation}
\label{tightness from coercivity}
\sup_n \mu^n\big(B_R^C(y_0)\big) \leq \epsilon,
\end{equation}
where $y_0\in Y$ is some arbitrary point and $B_R^C(y_0)=\{y \in Y, d(y,y_0) \geq R\}$.

Let $K \subseteq X$ be a compact set whose complement has small measure $\rho(K^C) < \epsilon/2$. Since $c$ is assumed to be locally uniformly coercive and bounded below, we can choose a radius $R$ such that
$$\inf_{d(y,y_0) > R} \inf_{x \in K} c(x,y) > 2(C - \inf c)/\epsilon.$$

Assume for contradiction that there is some $\mu^n$ such that $\mu^n\big(B_R^C(y_0)\big) > \epsilon$. Then, for any coupling $\pi$ between $\rho(x)$ and $\mu^n$,
$$\pi\big(K\times B_R^C(y_0)\big) \geq \mu^n\big(B_R^C(y_0)\big) - \rho(K^C) > \epsilon/2$$
that is, any transport plan must move more than $\epsilon/2$ of mass inside $K\subseteq X$ to $B_R^C(y_0) \subseteq Y$, which gives a lower bound on the transport cost:
$$\int c\ d\pi \geq 2(C - \inf c)/\epsilon \cdot \pi\big(K\times B_R^C(y_0)\big) + \inf c \cdot \big[1-\pi\big(K\times B_R^C(y_0)\big)\big] > C$$
Therefore, the optimal transport cost is bounded by $I_c(\rho(x),\mu^n) \geq C$.

Meanwhile, Lemma \ref{lemma: barycenter problem lower bound} implies that $F(\mu^n) \geq I_c(\rho(x),\mu^n)$. Hence,
$$C > F(\mu^n) > C$$
a contradiction. It follows that condition (\ref{tightness from coercivity}) holds, so that $\{\mu^n\}$ is uniformly tight.

By Prokhorov's theorem, $\{\mu^n\}$ is precompact, so some subsequence $\{\mu^{n_k}\}$ converges weakly to some $\mu \in P(Y)$. By Lemma \ref{lemma: duality and semicontinuity}, $F$ is lower semi-continuous, so
$$\inf F = \liminf F(\mu^{n_k}) \geq F(\mu) \geq \inf F.$$
It follows that $\mu$ minimizes the total transport cost and is a barycenter of $\rho(x,z)$.

Finally, notice that the two minimization problems are equivalent:
$$\min_{\mu\in P(Y)}
\min_{\substack{\pi \in P(X\times Y\times Z)\\ \pi_{XZ} = \rho\\ \pi_{YZ} = \mu \otimes v}} \int c ~d\pi
=
\min_{\substack{\pi \in P(X\times Y\times Z)\\ \pi_{XZ} = \rho\\ \pi_{YZ} = \pi_Y \otimes \pi_Z}} \int c ~d\pi$$
Then formula (\ref{joint measure formulation of barycenter problem}) is proved.

\section{Proof of Theorem \ref{thm: barycenter transport map}}
\label{appendix: barycenter transport map}
We need the following results:

\begin{lemma}
\label{lemma: independence check}
Let $Y,Z$ be metric spaces and $\pi$ a Borel probability measure on $Y\times Z$. The two marginals $\pi_Y,\pi_Z$ are independent ($\pi = \pi_{Y} \otimes \pi_{Z}$) if and only if for all $f \in C_b(Y), g \in C_b(Z)$,
\begin{equation}
\label{integral independence}
\int_{Y \times Z} f(y) g(z) d\pi(y,z) = \int_Y f(y) d\pi_Y(y) \int_Z g(z) d\pi_{Z}(z).
\end{equation}
\end{lemma}
\begin{proof}
The ``only if" follows from straightforward integration. For the ``if" part, let $U\subseteq Y, W\subseteq Z$ be arbitrary closed subsets, and define the following $C_b$ functions
$$f_n(y) = \max\big( 1 - n \cdot d_Y(y,U),0 \big), ~g_n(z) = \max\big( 1 - n \cdot d_Z(z,U),0 \big)$$
that descend to the indicator functions of $U$ and $W$. Apply (\ref{integral independence}) to $f_n \cdot g_n$ and take the limit $n\to \infty$; the Dominated Convergence Theorem implies that
\begin{equation}
\label{product independence}
\pi(U \times W) = \pi_Y(U) \pi_Z(W).
\end{equation}
Let $P_Y, P_Z$ be the collections of all closed subsets of $Y,Z$. Let $S_Y$ be the collection of all Borel measurable subsets $U$ of $Y$ such that for any closed $W\in P_Z$, the independence formula (\ref{product independence}) holds. We seek to apply Dynkin's $\pi$-$\lambda$ theorem. $P_Y \subseteq S_Y$ is closed under intersection and thus is a $\pi$-system. Meanwhile, it is straightforward to show that $S_Y$ is closed under set difference and countable union of increasing sequence, so $S_Y$ is a $\lambda$-system. It follows from Dynkin's theorem that $S_Y$ contains the $\sigma$-algebra of $Y$.

Similarly, we define $S_Z$ to be the collection of all Borel measurable subsets $W$ of $Z$ such that for any measurable $U\in S_Y$ (not just $P_Y$), the independence formula (\ref{product independence}) holds. Repeating the above argument for $S_Z,S_Y$ shows that $S_Z$ contains the $\sigma$-algebra of $Z$. Hence, (\ref{product independence}) holds for all measurable rectangles in $Y\times Z$ and $\pi = \pi_Y \otimes \pi_Z$.
\end{proof}

\begin{corollary}
\label{cor: verify independence}
Given the same condition as in Lemma \ref{lemma: independence check}, the independence $\pi = \pi_{Y} \otimes \pi_{Z}$ holds if and only if for all $f \in C_b(Y), g \in C_b(Z)$,
\begin{equation}
\label{integral independence 2}
\int g(z) d\pi_Z(z) = 0 \to \int_{Y \times Z} f(y) g(z) d\pi(y,z) = 0.
\end{equation}
\end{corollary}
\begin{proof}
The condition (\ref{integral independence}) can be rearranged into
$$\int_{Y \times Z} f(y) \big[ g(z) - \int g\ d\pi_Z \big] d\pi(y,z) = 0.$$
\end{proof}

\begin{lemma}
\label{lemma: Dirac weakly closed}
Let $A$ be any closed subset of a metric space $Y$. The set of Dirac masses on $A$:
\begin{equation*}
\Delta_A = \{\delta_y, ~y\in A\}
\end{equation*}
is closed in the weak topology of $P(Y)$.
\end{lemma}
\begin{proof}
Suppose a sequence $\{\delta_{y_n}\}$ in $\Delta_A$ converges weakly to some $\mu\in P(Y)$. For any $n \geq 1$, let $A_n$ be the closure of the subsequence $\{y_m\}_{m \geq n}$. Then, by weak convergence, $\mu(A_n) = 1$ for all $n$ and thus $\mu(A_{\infty})=1$ where $A_{\infty} = \cap_{n} A_n$ is the set of limits of $\{y_n\}$. It follows that the set $A_{\infty}$ is non-empty. Meanwhile, $A_{\infty}$ cannot contain more than one point, otherwise it would contradict the weak convergence of $\{\delta_{y_n}\}$. Hence, $\mu$ is the Dirac mass $\delta_y$ where $y \in A$ is the only point in $A_{\infty}$.
\end{proof}

\begin{corollary}
\label{cor: Dirac identification}
Given a metric space $Y$, there exists a measureable function $F_{\Delta}: P(Y) \to Y$ such that $F_{\Delta}(\delta_y) = y$ for every Dirac mass $\delta_y$.
\end{corollary}
\begin{proof}
We can simply define (for some arbitrary fixed $y_0 \in Y$),
\begin{equation*}
F_{\Delta}(\mu) = \begin{cases}
y \text{ if } \mu \in \Delta_Y \text{ and } \mu = \delta_y\\
y_0 \text{ else} 
\end{cases}
\end{equation*}
Then, for any closed subset $A \subseteq Y$,
\begin{equation*}
F_{\Delta}^{-1}(A) = \begin{cases}
\Delta_A \text{ if } y_0 \notin A\\
\Delta_A \cup \big(P(Y)\backslash \Delta_Y \big) \text{ else}
\end{cases}
\end{equation*}
It follows from Lemma \ref{lemma: Dirac weakly closed} that $F_{\Delta}^{-1}(A)$ is a measureable subset. Hence, $F_{\Delta}$ is measureable.
\end{proof}

\medskip
As argued in the beginning of Section \ref{sec: conditional transport map}, we can define the transport maps $T(x,z)$ by formula (\ref{def: monge transport map}), which is a Monge formulation of the constraints on the Kantorovich solution from (\ref{joint measure formulation of barycenter problem}): given any candidate transport map $T: X\times Z \to Y$, the corresponding transport plan is
\begin{equation}
\label{Kantorovich solution concentrated ona graph}
\pi := (Id,T)\#\rho(x,z) \in P(X\times Z\times Y),
\end{equation}
The marginal constraint $\pi_{XY} = \rho(x,z)$ is satisfied automatically, while the independence constraint $\pi_{YZ}=\pi_Y\otimes\pi_Z$ can be checked via Corollary \ref{cor: verify independence}.

Specifically, define the indicator function:
\begin{equation*}
I(T) = \begin{cases}
0 \text{ if there exists $\mu \in P(Y)$ such that } (T,Proj_Z)\#\rho(x,z) = \mu \otimes v\\
\infty \text{ otherwise.}
\end{cases}
\end{equation*}
Then, Corollary \ref{cor: verify independence} implies that
\begin{align*}
I(T) &= \sup_{\psi_Y \in C_b(Y)}
\sup_{\substack{\psi_Z \in C_b(Z)\\ \int \psi_Z dv = 0}}
\int_{Y\times Z} \psi_Y(y) \psi_Z(z) ~d\tilde{T}\#\rho\\
&= \sup_{\psi_Y \in C_b(Y)}
\sup_{\substack{\psi_Z \in C_b(Z)\\ \int \psi_Z dv = 0}}
\int \psi_Y (T(x,z)) \psi_Z(z) ~d\rho(x,z).
\end{align*}
It follows that we have a Monge formulation of the barycenter problem (\ref{joint measure formulation of barycenter problem}):
\begin{align}
\label{Monge formulation of barycenter problem}
\begin{split}
&\inf_{\substack{\text{Borel measurable }\\ T:X\times Z \to Y}}
\int c(x,T(x,z)) ~d\rho(x,z) + I(T)\\
=& \inf_{\substack{\text{Borel measurable}\\ T:X\times Z \to Y}}~
\sup_{\substack{\psi_Y \in C_b(Y)\\ \psi_Z \in C_b(Z)\\ \int \psi_Z dv = 0}}
\int c(x,T(x,z)) - \psi_Y (T(x,z)) \psi_Z(z) ~d\rho(x,z).
\end{split}
\end{align}

Essentially, (\ref{Monge formulation of barycenter problem}) is a minimization over couplings of the form (\ref{Kantorovich solution concentrated ona graph}), which are Kantorovich solutions concentrated on graphs. Therefore, (\ref{Monge formulation of barycenter problem}) is bounded below by  (\ref{joint measure formulation of barycenter problem}), which minimizes over general Kantorovich solutions. To finish the proof, it suffices to show that (\ref{joint measure formulation of barycenter problem}) can be achieved by some transport map $T$.

Let $\pi$ be a Kantorovich solution of (\ref{joint measure formulation of barycenter problem}), which exists by Theorem \ref{thm: well-posedness and existence}, and let $\mu=\pi_Y$ be the corresponding barycenter.
By Assumption \ref{Monge assumption}, for $v$-almost all $z \in Z$, the unique optimal coupling between $\rho(x|z)$ and $\mu$ is concentrated on the graph of some transport map $T_z$. Define the map $T(x,z) := T_z(x)$. It follows that for $v$-almost all $z$,
\begin{equation*}
\pi(x,y|z) = (Id,T_z)\#\rho(x|z)
\end{equation*}
or equivalently, for $\rho$-almost all $(x,z)$,
\begin{equation*}
\pi(y|x,z) = \delta_{T(x,z)}
\end{equation*}
Let $F_{\Delta}$ be the map given by Corollary \ref{cor: Dirac identification}.
Given that $\pi(y|x,z)$ can be chosen as a measureable function from $X\times Z$ to $P(Y)$, we can conclude that
\begin{equation*}
T(x,z) = F_{\Delta}\big( \pi(y|x,z) \big)
\end{equation*}
is a measureable function from $X\times Z$ to $Y$.

Hence,
\begin{equation*}
(\ref{joint measure formulation of barycenter problem}) = \int c ~d\pi = \int c ~d\pi(y|x,z) d\rho(x,z)
=
\int c\big(x,T(x,z)\big) d\rho(x,z) \geq (\ref{Monge formulation of barycenter problem})
\end{equation*}

\section{Proof of Theorem \ref{thm: variance decomposition}}
\label{appendix: variance decomposition}

If the marginal $\rho(x)$ does not have finite second moment, then both sides of (\ref{variance decomposition continuous}) are infinite: either $Var(\mu)$ or $W_2^2(\rho(x),\mu)$ must be infinite and Lemma \ref{lemma: barycenter problem lower bound} bounds $\int W_2^2(\rho(x|z),\mu)d\nu(z)$ below by $W_2^2(\rho(x),\mu)$. Hence, in the following proof we can assume that
\begin{equation}
\label{finite second moment}
\infty > \E_{\rho(x)}\big[\|x\|^2\big] = \int_{\R^d\times Z} \|x\|^2 d\rho(x,z) = \int_Z \E_{\rho(x|z)}\big[ \|x\|^2 \big] d\nu(z).
\end{equation}

We first prove the special case where $\nu(z)$ is finitely-supported and that the conditionals $\rho(x|z)$ are absolutely continuous.

Denote the subset of $P(\R^d)$ that consists of probabilities measures with finite second moments by $P_2(\R^d)$. Denote the support of $\nu$ by $\{z_k\}_{k=1}^K \subseteq Z$. Denote the positive numbers $\nu(\{z_k\})$ by $P_k$ and the conditionals $\rho(x|z_k)$ by $\rho_k(x)$. Then, the marginal $\rho(x)$ is the weighted sum $\sum_{k=1}^K P_k\rho_k$. Condition (\ref{finite second moment}) implies that each $\rho_k \in P_2(\R^d)$.

\begin{lemma}
\label{lemma: variance decomposition discrete}
Given absolutely continuous measures $\rho_k \in P_2(\R^d)$ and weights $P_k > 0$ for $1\leq k \leq K$, there exists a unique barycenter $\mu$ and it satisfies the discrete version of (\ref{variance decomposition continuous}):
\begin{equation}
\label{variance decomposition discrete}
Var(\rho(x)) = Var(\mu) + \sum_{k=1}^K P_k W_2^2(\rho_k,\mu)
\end{equation}
\end{lemma}

\begin{proof}
By Theorems 3.1 and 5.1 of \cite{kim2017wasserstein}, since $\rho_k$ are absolutely continuous, there exists a unique barycenter $\mu(x)$ and it is also absolutely continuous. Then, Brenier's theorem (Theorem 2.12 \cite{villani2003topics}) implies that there is a unique optimal transport map $T_k$ from each $\rho_k$ to $\mu$. The transport maps have the form $T_k = \nabla \psi_k$ for some convex functions $\psi_k$, and they are invertible almost everywhere: let $\psi_k^*$ be the Legendre transform of $\psi_k$, then
$$\nabla \psi_k^* \circ \nabla \psi_k (x) = x, ~\nabla \psi_k \circ \nabla \psi_k^* (y) = y$$
for $\rho_k$-almost all $x$ and $\mu$-almost all $y$. Furthermore, $\nabla \psi_k^*$ serves as the optimal transport map from $\mu$ back to $\rho_k$.

Denote the mean of $\rho(x)$ by $\overline{x}$. Note that $\overline{x}$ is also the mean of the barycenter $\mu$: let $\pi$ be the (unique) Kantorovich solution given by Theorem \ref{thm: well-posedness and existence}, let $X_k$ be the random variables of $\rho_k=\pi_{XZ}(x|z_k)$, and let $Y$ be the random variable of $\mu=\pi_Y$. Define the mean $\overline{X}=\sum_{k=1}^K P_k X_k$.
Then the barycenter problem's objective (\ref{joint measure formulation of barycenter problem}) becomes
\begin{equation}
\label{barycenter mean argument}
\E \sum_{k=1}^K P_k ||Y-X_k||^2 = \E ||Y-\overline{X}||^2 + \sum_{k=1}^K P_k \E ||\overline{X}-X_k||^2.
\end{equation}
Since $Y$ minimizes the objective, we must have $Y=\overline{X}$, so that $\E[Y]=\overline{x}$.
Then
\begin{align*}
Var(\rho) &= \int_{\R^d} \|x-\overline{x}\|^2 d\rho(x) = \sum_{k=1}^K P_k \int \|(x-T_k(x))+(T_k(x)-\overline{x})\|^2 d\rho_k(x)\\
&= \sum_{k=1}^K P_k \int \|x-T_k(x)\|^2 + \|T_k(x)-\overline{x}\|^2 + 2 \lb x-T_k(x),T_k(x)-\overline{x} \rb d\rho_k(x).
\end{align*}
The first term is exactly the total transport cost, while the second term is the barycenter's variance:
$$\sum_{k=1}^K P_k \int \|T_k(x)-\overline{x}\|^2 d\rho_k(x) = \sum_{k=1}^K P_k \int \|y-\overline{x}\|^2 d\mu(y) = Var(\mu).$$
Regarding the third term, we use the fact that $\nabla \psi_k^*\#\mu = \rho_k$ to obtain
\begin{align*}\sum P_k \int \lb x-T_k(x),T_k(x)-\overline{x} \rb d\rho_k(x) &= \sum P_k \int \lb \nabla \psi_k^*(y)-y,y-\overline{x} \rb d\mu(y)\\
&= \int \big\lb \sum_{k=1}^K P_k \nabla \psi_k^*(y)-y,y-\overline{x} \big\rb d\mu(y).
\end{align*}
Remark 3.9 from \cite{agueh2011barycenters} shows that $\sum P_k \nabla \psi_k^*$ is exactly the identity, so the third term vanishes. It follows that formula (\ref{variance decomposition discrete}) holds.
\end{proof}

Now we tackle the general case when $\rho(x,z)$ is an arbitrary probability measure over $\R^d\times Z$.
Condition (\ref{finite second moment}) implies that $\rho(x|z) \in P_2(\R^d)$ for $v$-almost every $z$, so without loss of generality, $\rho(x|z)$ can be seen as a random variable from $Z$ to $P_2(\R^d)$. We denote its distribution by $\Omega(\eta)$, which belongs to $P(P_2(\R^d))$, the space of probability measures over $P_2(\R^d)$.

Abusing notation, we denote by $P_2(P_2(\R^d))$ the space of probabilities $\Omega'(\eta)$ on $(P_2(\R^d),W_2)$ with finite second moment, that is, for some (and thus any) $\rho_0 \in P_2(\R^d)$,
$$\int_{P_2(\R^d)} W^2_2(\rho_0,\eta)d\Omega'(\eta) < \infty.$$
Then $P_2(P_2(\R^d))$ can be equipped with the $2$-Wasserstein metric. Condition (\ref{finite second moment}) implies that our distribution $\Omega$ belongs to $P_2(P_2(\R^d))$: for any $\rho_0 \in P_2(\R^d)$,
\begin{align*}
\int_{P_2(\R^d)} W^2_2(\rho_0,\eta) ~d\Omega(\eta) &=
\int \inf_{\substack{\pi\in P(\R^d\times\R^d)\\
\pi_1=\rho_0,\pi_2=\eta}} \int \|x-x'\|^2 d\pi(x,z') ~d\Omega(\eta)\\
&\leq \iint \|x-x'\|^2 ~d\rho_0\otimes\eta(x,x') ~d\Omega(\eta) \text{ by the trivial coupling}\\
&\leq \iint 2(\|x\|^2+\|x'\|^2) d\rho_0(x)d\eta(x') d\Omega(\eta)\\
&\leq 2\E_{\rho_0(x)}[X^2] + 2\int_{P_2(\R^d)} \E_{\rho(x)}[X^2] d\Omega(\rho)\\
&\leq 2\E_{\rho_0(x)}[X^2] + 2\int_Z \E_{\rho(x|z)}[X^2] d\nu(z)\\
&< \infty \text{ by (\ref{finite second moment})}.
\end{align*}

By Theorem 6.18 of \cite{villani2008optimal}, both $(P_2(\R^d),W_2)$ and $(P_2(P_2(\R^d)),W_2)$ are Polish spaces, each of whose elements can be approximated by finitely-supported probability measures. Let $\{\Omega^n\}_{n=1}^{\infty} \subseteq P_2(P_2(\R^d))$ be a sequence of finitely-supported measures that converge to $\Omega$ in $W_2$. Then, each $\Omega^n$ can be expressed as
$$\Omega^n = \sum_{k=1}^{K^n} P_k^n \delta_{\rho_k^n},$$
where $K^n$ is the size of the support of $\Omega^n$, the positive numbers $P_k^n$ are the weights, and $\delta_{\rho_k^n}$ is the Dirac measure at $\rho_k^n \in P_2(\R^d)$. Define the marginal $\rho^n$ of $\Omega^n$ by
\begin{equation}
\label{x margin}
\rho^n = \sum_{k=1}^{K^n} P_k^n \rho_k^n.
\end{equation}
It follows that $\rho^n \in P_2(\R^d)$.

In order to apply Lemma \ref{lemma: variance decomposition discrete}, we show that these $\rho_k^n$ can be assumed to be absolutely continuous. A nice property of $(P_2(\R^d),W_2)$ is that absolutely continuous measures are dense in it: any measure in $(P_2(\R^d),W_2)$ can be approximated by finitely-supported measures, which can then be approximated by absolutely continuous measures using kernel smoothing. Thus, for each $\Omega^n$, we can construct
$$\tilde{\Omega}^n = \sum_{k=1}^{K^n} P_k^n \delta_{\tilde{\rho}_k^n}, \text{ such that }W_2^2(\rho_k^n,\tilde{\rho}_k^n)<\frac{1}{nK^n}, \text{ so that } W_2^2(\Omega^n,\tilde{\Omega}^n)<\frac{1}{n},$$
so that $\tilde{\Omega}^n$ also converge to $\Omega$ in $(P_2(P_2(\R^d)),W_2)$. It follows that $\rho^n$ are also absolutely continuous.

Now given that each $\Omega^n$ consists of absolutely continuous $\rho_k^n$, Lemma \ref{lemma: variance decomposition discrete} implies that each $\Omega^n$ has a unique barycenter $\mu^n$ and it satisfies
\begin{equation}
\label{finitely-supported decomposition}
Var(\rho^n) = Var(\mu^n) + \int_{P_2(\R^d)} W^2_2(\mu^n,\eta) d\Omega^n(\eta).
\end{equation}

The following two lemmas show that $\rho^n$ and $\mu^n$ enjoy good convergence properties.

\begin{lemma}
\label{lemma: x margin convergence}
The marginal $\rho^n$ converges to $\rho(x)$ in $(P_2(\R^d),W_2)$.
\end{lemma}
\begin{proof}
We apply condition (iv) of Theorem 7.12 from \cite{villani2003topics}, which shows that for any Polish space $X$ with metric $d$, a sequence $\eta^n$ converges to $\eta$ in $(P_2(X),W_2)$ if and only if
\begin{equation}
\label{test of W2 convergence}
\lim_{n\to\infty} \int_X \psi d\eta^n = \int_X \psi d\eta
\end{equation}
for any continuous function $\psi$ that grows at most quadratically: $|\psi(x)| \leq C(1+d(x_0,x)^2)$ for some $x_0\in X$ and $C>0$.
Therefore, it suffices to show that
$$\lim_{n\to\infty} \int_{\R^d} \psi d\rho^n = \int \psi d\rho$$
for any $\psi$ with a quadratic bound: $|\psi(x)| \leq C(1+\|x\|^2)$ for some $C>0$.

By (\ref{x margin}), it is equivalent to
\begin{equation}
\label{level-up limit}
\lim_{n\to\infty}\int_{P_2(\R^d)} F_{\psi}(\eta) d\Omega^n(\eta) = \int F_{\psi}(\eta) d\Omega(\eta),
\end{equation}
where $F_{\psi}(\eta) := \int \psi d\eta$. The function $F_{\psi}$ is continuous on $(P_2(\R^d),W_2)$ by condition (\ref{test of W2 convergence}). The quadratic bound on $\psi$ translates to a quadratic bound on $F_{\psi}$:
$$F_{\psi}(\eta) \leq C \big(W_2^2(\eta,\delta_0) +1 \big),$$
where $\delta_0$ is the Dirac measure at $0$.

Then we can apply Theorem 7.12 \cite{villani2003topics} on $(P_2(P_2(\R^d)),W_2)$, and (\ref{level-up limit}) follows from the $W_2$ convergence of $\Omega^n$ to $\Omega$.
\end{proof}

\begin{lemma}
\label{lemma: barycenter convergence}
A subsequence of $\{\mu^n\}$ converges in $(P_2(\R^d),W_2)$ to a barycenter $\mu$ of $\Omega$.
\end{lemma}

\begin{proof}
First, the total transport cost from $\Omega^n$ to its barycenter $\mu^n$ can be computed through
$$c^n := \int_{P_2(\R^d)} W_2^2(\eta,\mu^n) d\Omega^n(\eta) = W_2^2(\delta_{\mu^n},\Omega^n),$$
where the second $W_2$ belongs to $P_2(P_2(\R^d))$ and $\delta_{\mu^n}$ is the Dirac measure on $\mu^n$. Then, for any $n,m$,
\begin{align*}
c^n &\leq W_2^2(\delta_{\mu^m},\Omega^n) \text{ since $\mu^n$ minimizes total transport cost}\\
&\leq c^m + W_2^2(\Omega^n,\Omega^m) \text{ by triangle-inequality}\\
\to & |c^n-c^m| \leq W_2^2(\Omega^n,\Omega^m).
\end{align*}
Since $W_2^2(\Omega^n,\Omega)\to 0$, the difference $W_2^2(\Omega^n,\Omega^m)\to 0$ as $n,m\to\infty$, so $|c^n-c^m|\to 0$. It follows that $\{c^n\}$ is a Cauchy sequence and thus converges.

Next, we establish some uniform bound on the decay of $\{\mu^n\}$ at infinity. We begin with a weak bound: by Lemma \ref{lemma: barycenter problem lower bound} and triangle inequality,
\begin{align*}
W_2^2(\rho(x),\mu^n) &\leq W_2^2(\Omega,\delta_{\mu^n})\\
&\leq c^n + W_2^2(\Omega^n,\Omega).
\end{align*}
Denote by $\delta_0 \in P(\R^d)$ the Dirac measure at $0$ and by $B_R \subseteq \R^d$ the open ball centered at $0$ with radius $R$. By the triangle inequality, for any $R$,
\begin{align*}
W_2(\rho(x),\mu^n) &\geq |W_2(\mu^n,\delta_0)-W_2(\rho(x),\delta_0)|\\
& \geq \sqrt{\int_{\R^d-B_R} R^2 d\mu^n} - \sqrt{\int_{\R^d} \|x\|^2 d\rho(x)}\\
& \geq R \sqrt{\mu^n(\R^d-B_R)} - \sqrt{\E_{\rho(x)}[X^2]}.
\end{align*}
Combining the two inequalities, we obtain for any $n$ and $R$,
\begin{equation*}
\mu^n(\R^d-B_R) \leq \Big( \frac{\sqrt{\E_{\rho(x)}[X^2]} + \sqrt{c^n + W_2^2(\Omega^n,\Omega)}}{R} \Big)^2.
\end{equation*}
Since the second moment $\E_{\rho(x)}[X^2]$ is finite by (\ref{finite second moment}) and $c^n, W_2^2(\Omega^n,\Omega)$ are convergent sequences, there exists some constant $C$ large enough so that
\begin{equation}
\label{quadratic decay}
\sup_n \mu^n(\R^d-B_R) \leq \frac{C}{R^2}.
\end{equation}

An immediate consequence is that $\mu^n$ is uniformly tight. Then, Prokhorov's theorem implies that $\{\mu^n\}$ has a subsequence $\{\mu^{n_i}\}$ that converges weakly to some limit $\mu \in P(\R^d)$. By Kantorovich duality \cite{villani2008optimal}, the optimal transport cost $W_2^2(\cdot,\cdot)$ on $P(\R^d)$ can be expressed as a supremum over bounded continuous functions, and thus is lower semi-continuous in the topology of weak convergence of $P(\R^d)$. Thus,
$$\forall \eta \in P(\R^d), ~W_2^2(\mu,\eta) \leq \liminf_{n_i\to\infty} W_2^2(\mu^{n_i},\eta).$$
In particular, by setting $\eta=\delta_0$, we have shown that $\mu$ has finite second moment: $\mu \in P_2(\R^d)$.\\

Now we prove that this subsequence $\mu^{n_i}$ converges to $\mu$ in the stronger topology of $(P_2(\R^d),W_2)$, and we use bootstrapping to improve the tail bound (\ref{quadratic decay}). Condition (ii) of Theorem 7.12 of \cite{villani2003topics} indicates that it is necessary and sufficient to prove that
\begin{equation}
\label{tightness condition}
\lim_{R\to\infty}\limsup_{n_i\to\infty} \int_{\R^d-B_R} \|x\|^2 d\mu^{n_i}(x) = 0.
\end{equation}
We show that this condition holds for the whole sequence $\mu^n$.

Recall the arguments of Lemma \ref{lemma: variance decomposition discrete}: Since $\rho_k^{n},\mu^{n}$ are absolutely continuous for all $n$ and $1\leq k \leq K^{n}$, Brennier's theorem implies that the optimal transport maps $\nabla\psi_k^n,\nabla(\psi_k^n)^*$ between them are invertible almost everywhere and, Remark 3.9 of \cite{agueh2011barycenters} indicates that $\sum P_k \nabla (\psi_k^n)^* = Id$. Then, $\mu^{n}$-almost all $x$ can be expressed as the convex combination $x = \sum P_k^n \nabla (\psi_k^n)^*(x)$. It follows from convexity that
$$\|x\|^2 \leq \sum_{k=1}^{K^n} P_k^n \|\nabla (\psi_k^n)^*(x)\|^2.$$
Then, $\nabla (\psi_k^n)^*\#\mu^n = \rho_k^n$ implies that,
\begin{align}
\int_{\R^d-B_R} \|x\|^2 d\mu^{n}(x) &\leq \sum_{k=1}^{K^n} P_k^n \int_{\R^d-B_R} \|\nabla (\psi_k^n)^*(x)\|^2 d\mu^{n}(x) \nonumber \\
\label{separate to rho nk}
&= \sum_{k=1}^{K^n} P_k^n \int_{\nabla (\psi_k^n)^*(\R^d-B_R)} \|x\|^2 d\rho^{n}_k(x).
\end{align}
In the last line above, we are integrating the measure $\rho^{n}_k$ restricted to the domain $\nabla(\psi_k^n)^*(\R^d-B_R)$. Equivalently, we are integrating over some measure $\tilde{\rho}_k^n$ (not necessarily a probability measure) such that $0\leq \tilde{\rho}_k^n \leq \rho_k^n$ (setwise) and
$$\tilde{\rho}_k^n(\R^d) = \rho_k^n\big( \nabla(\psi_k^n)^*(\R^d-B_R) \big) = \mu^n(\R^d-B_R).$$
For convenience, for any $R, n$ and $1\leq k \leq K^n$, define the following collections of measures:
$$M_k^n(R) := \{\tilde{\rho}_k^n \in M^+(\R^d), ~\tilde{\rho}_k^n \leq \rho_k^n \text{ and } \tilde{\rho}_k^n(\R^d) \leq C/R^2\}$$
$$M^n(R) := \{\tilde{\rho}^n \in M^+(\R^d), ~\tilde{\rho}^n \leq \rho^n \text{ and } \tilde{\rho}^n(\R^d) \leq C/R^2\}$$
$$M(R) := \{\tilde{\rho} \in M^+(\R^d), ~\tilde{\rho} \leq \rho \text{ and } \tilde{\rho}(\R^d) \leq C/R^2\},$$
where $M^+(\R^d)$ is the set of nonnegative Borel measures, and the uniform upper bound $C/R^2$ comes from (\ref{quadratic decay}).

It follows that the restriction of $\rho^{n}_k$ to $\nabla(\psi_k^n)^*(\R^d-B_R)$ belongs to $M_k^n(R)$ and
\begin{equation}
\label{transfer to rho nk}
\int_{\nabla (\psi_k^n)^*(\R^d-B_R)} \|x\|^2 d\rho^{n}_k(x) \leq \sup_{\tilde{\rho}_k^n \in M_k^n(R)} \int_{\R^d} \|x\|^2 d\tilde{\rho}_k^n(x).
\end{equation}

Given any choice of $\{\tilde{\rho}_k^n\}_{k=1}^{K^n}$, it is straightforward to show that the weighted sum
$$\tilde{\rho}^n := \sum_{k=1}^{K^n} P_k^n \tilde{\rho}_k^n$$
belongs to $M^n(R)$ and
\begin{equation}
\label{merge to rho n}
\sum_{k=1}^{K^n} P_k^n
\sup_{\tilde{\rho}_k^n \in M_k^n(R)} \int \|x\|^2 d\tilde{\rho}_k^n(x)
\leq \sup_{\tilde{\rho}^n \in M^n(R)} \int \|x\|^2 d\tilde{\rho}^n.
\end{equation}

For any $n$, since $\rho^n$ is absolutely continuous, Brennier's theorem implies that there is an optimal transport map $T^n$ such that $T^n\#\rho^n = \rho$. Given any $\tilde{\rho}^n \in M^n(R)$, it is straightforward to show that $\tilde{\rho} := T_n\# \tilde{\rho} \in M(R)$ and that the optimal transport cost
$$W_2(\tilde{\rho}^n,\tilde{\rho}) \leq W_2(\rho^n,\rho).$$
Denote by $\tilde{\delta}_0$ the Dirac measure at zero with the same mass as $\tilde{\rho}^n$. By the triangle inequality,
\begin{align}
\label{move to rho}
\begin{split}
\sqrt{\int \|x\|^2 d\tilde{\rho}^n} &= W_2(\tilde{\rho}^n,\tilde{\delta}_0)
\leq W_2(\tilde{\rho},\tilde{\delta}_0) + W_2(\tilde{\rho},\tilde{\rho}^n)\\
&\leq \sqrt{\int \|x\|^2 d\tilde{\rho}} + W_2(\rho,\rho^n)
\end{split}
\end{align}

Combining the inequalities (\ref{separate to rho nk}), (\ref{transfer to rho nk}), (\ref{merge to rho n}), (\ref{move to rho}), we obtain
\begin{equation}
\label{barycenter vanishing upper bound}
\int_{\R^d-B_R} \|x\|^2 d\mu^{n}(x) 
\leq 
\Big[ \sup_{\tilde{\rho} \in M(R)}
\sqrt{\int_{\R^d} \|x\|^2 d\tilde{\rho}} + W_2(\rho,\rho^n) \Big]^2
\end{equation}
for all $R$ and $n$.

Now, we show that the $\sup_{M(R)}$ term in (\ref{barycenter vanishing upper bound}) vanishes as $R\to\infty$. Recall the definition of $M(R)$: any $\tilde{\rho} \in M(R)$ can be seen as a measure obtained from $\rho$ by removing $(1-C/R^2)$-amount of mass. Since the integrand $\|x\|^2$ is strictly increasing in the radial direction, in order to maximize $\int \|x\|^2 d\tilde{\rho}$, we should first remove the mass of $\rho$ that is closest to $0$. To formalize this intuition, we define
\begin{equation*}
r(R) := \inf \{r\geq 0 ~|~ \rho(\R^d-B_r) \leq C/R^2\}
\end{equation*}
where $B_r$ is the open ball. It follows that
\begin{equation}
\label{upper bound by removing mass}
\sup_{\tilde{\rho} \in M(R)} \int_{\R^d} \|x\|^2 d\tilde{\rho} \leq \int_{\R^d-B_{r(R)}} \|x\|^2 d\rho
\end{equation}
where the upper bound is always finite, since (\ref{finite second moment}) implies that $\rho(x)$ has finite second moment. Then, there are two possibilities: First, suppose that $\rho(x)$ is compactly-supported, say, inside some ball $B_{R_0}$. Then, we obtain the trivial bound
\begin{equation*}
\sup_{\tilde{\rho} \in M(R)} \int_{\R^d} \|x\|^2 d\tilde{\rho} \leq R_0^2 \cdot \frac{C}{R^2}
\end{equation*}
which goes to $0$ as $R \to \infty$.
Second, $\rho(x)$ has unbounded support, and thus $\rho(\R^d-B_r) > 0$ for all $r$. So the function $r(R)$ must go to infinity as $R\to \infty$. It follows that the upper bound (\ref{upper bound by removing mass}) goes to zero.

Hence, we always have
\begin{equation*}
\lim_{R\to\infty}
\sup_{\tilde{\rho} \in M(R)} \sqrt{ \int_{\R^d} \|x\|^2 d\tilde{\rho} }
= 0
\end{equation*}
Meanwhile, Lemma \ref{lemma: x margin convergence} indicates that $W_2(\rho,\rho^n)\to 0$ as $n\to\infty$. Then, condition (\ref{tightness condition}) follows from (\ref{barycenter vanishing upper bound}):
\begin{equation*}
\lim_{R\to\infty}\limsup_{n\to\infty}
\int_{\R^d-B_R} \|x\|^2 d\mu^{n}(x) 
\leq 
\Big[ \lim_{R\to\infty} \sup_{\tilde{\rho} \in M(R)}
\sqrt{\int_{\R^d} \|x\|^2 d\tilde{\rho}} + \limsup_{n\to\infty} W_2(\rho,\rho^n) \Big]^2 = 0,
\end{equation*}
and we conclude that the subsequence $\mu^{n_i}$ converges to $\mu$ in $W_2$.\\

Finally, we show that $\mu$ is a barycenter of $\Omega$. For any $\tilde{\mu} \in P_2(\R^d)$ and any $n_i$,
\begin{equation}
\label{barycenter inequality}
c^{n_i} = W_2^2(\Omega^{n_i},\delta_{\mu^{n_i}}) \leq W_2^2(\Omega^{n_i},\delta_{\tilde{\mu}}).
\end{equation}
Since $W_2(\Omega,\Omega^{n})\to 0$ and $W_2(\delta_{\mu},\delta_{\mu^{n_i}}) = W_2(\mu,\mu^{n_i}) \to 0$, the cost $W_2^2(\Omega^{n_i},\delta_{\mu^{n_i}})$ converges to $W_2^2(\Omega,\delta_{\mu})$ by triangle inequality. Then, (\ref{barycenter inequality}) implies that
$$W_2^2(\Omega,\delta_{\mu}) \leq \lim_{n_i\to\infty} W_2^2(\Omega^{n_i},\delta_{\tilde{\mu}}) = W_2^2(\Omega,\delta_{\tilde{\mu}}).$$
Since the inequality holds for all $\tilde{\mu}\in P_2(\R^d)$, the limit $\mu$ is a barycenter of $\Omega$.
\end{proof}

Now, we can take the limit in $n$ in equation (\ref{finitely-supported decomposition}). Taking a subsequence if necessary, Lemma \ref{lemma: x margin convergence} and Lemma \ref{lemma: barycenter convergence} imply that $\rho^n \to \rho$ and $\mu^n \to $ a barycenter $\mu$ in $W_2$, while the functions $Var$ and $W_2^2$ are continuous over $W_2$, so
\begin{equation*}
Var(\rho) = Var(\mu) + W_2^2(\Omega,\delta_{\mu})
\end{equation*}
which finishes the proof of formula (\ref{variance decomposition continuous}).

\section{Proof of Theorem \ref{thm: Gaussian covariance}}
\label{appendix: Gaussian covariance}

We apply the arguments of Appendix \ref{appendix: variance decomposition}. Since $\rho(x)$ is assumed to have finite second moment, $\rho(x,z)$ can be converted to a distribution $\Omega \in P_2(P_2(\R^d))$. We seek to construct a sequence of finitely-supported measures $\Omega^n$ that converge to $\Omega$ in $W_2$, such that each $\Omega^n = \sum_{k=1}^{K^n} P_k \delta_{\rho_k^n}$ and each $\rho_k^n$ is a non-degenerate Gaussian (i.e. the covariance $S(z)$ is positive-definite). Fix some $\rho_0 \in \text{supp } \Omega$. For each $n$, let $C^n \subseteq P_2(\R^d)$ be a compact subset such that
$$\int_{P_2(\R^d)-C^n} W_2^2(\rho_0,\eta) d\Omega(\eta) < \frac{1}{n}.$$
Let $\{B(\rho_k^n,1/2n)\}_{k=1}^{K^n}$ be a finite cover of $C^n$ by open balls, where $\rho_k^n \in \text{supp } \Omega$ and thus are Gaussians. If $\rho_k^n$ is degenerate, then replace it with some non-degenerate Gaussian $\tilde{\rho}_k^n \in B(\rho_k^n,1/2n)$ and use the ball $B(\tilde{\rho}_k^n,1/n)$. Define the disjoint cover $\{U_k^n\}_{k=1}^{K^n}$ by
$$U_k^n = B(\rho_k^n,1/n) - \bigcup_{h=1}^{k-1} B(\rho_h^n,1/n).$$
Define a map $F^n$ on $P_2(\R^d)$ that sends each $U_k^n$ to $\rho_k^n$ and everything else to $\rho_0$. Define
$$\Omega^n = F^n\#\Omega = \sum_{k=1}^{K^n}\Omega(U_k^n) \delta_{\rho_k^n} + \Big( 1 - \sum_{k=1}^{K^n}\Omega(U_k^n) \Big) \delta_{\rho_0}.$$
It follows that
$$W_2^2(\Omega,\Omega^n) \leq \int W_2^2(\eta,F^n(\eta)) d\Omega(\eta) < \frac{2}{n}$$
So $\Omega^n \to \Omega$ in $W_2$.

Denote $\rho_k^n$ by $\mathcal{N}(\overline{x}_k^n, S_k^n)$, and denote the $X$-marginal of $\Omega^n$ by $\rho^n = \sum P_k \rho_k^n$. Now, Theorem 6.1 of \cite{agueh2011barycenters} implies that $\Omega^n$ has a unique barycenter $\mu^n$, which is a Gaussian whose covariance $S^n$ satisfies
\begin{equation}
\label{finite Gaussian formula}
S^n = \sum_{k=1}^{K^n} P_k \sqrt{\sqrt{S^n} S_k^n \sqrt{S^n}} = \int \sqrt{\sqrt{S^n} S(\eta) \sqrt{S^n}} d\Omega^n(\eta),
\end{equation}
where $S(\rho)$ denotes the covariance of $\rho$. Also, the argument (\ref{barycenter mean argument}) implies that the mean $\overline{x}^n$ of $\mu^n$ satisfies
$$\overline{x}^n = \sum P_k \overline{x}_k^n = \E_{\rho^n(x)}[x].$$

Lemma \ref{lemma: x margin convergence} implies that $\rho^n$ converges to the marginal $\rho(x)$ in $W_2$, and Taking a subsequence if necessary, Lemma \ref{lemma: barycenter convergence} implies that $\mu^n$ converges to a barycenter $\mu$ of $\Omega$ in $W_2$. Since the set of Gaussian distributions is closed in $W_2$, this barycenter $\mu$ must be some Gaussian $\mathcal{N}(\overline{x},S)$. By Theorem 7.12 of \cite{villani2003topics}, the covariance function $S(\eta)$ is continuous over $\eta \in (P_2(\R^d),W_2)$, so $S^n$ converges to $S$. Then, we can take the limit on both sides of (\ref{finite Gaussian formula}):
\begin{align*}
S &= \lim_{n\to\infty} S^n = \lim_{n\to\infty} \int \sqrt{\sqrt{S} S(\eta) \sqrt{S}} d\Omega^n(\eta) + O\big( \sqrt{\|S-S^n\|_{op}} \big) \cdot \int \sqrt{S(\eta)} d\Omega^n(\eta)\\
& = \int \sqrt{\sqrt{S} S(\eta) \sqrt{S}} d\Omega(\eta).
\end{align*}
Similarly,
\begin{align*}
\overline{x} &= \lim_{n\to\infty} \int x d\rho^n(x) = \int x d\rho(x).
\end{align*}

Finally, suppose that the set of $z\in Z$ such that $\rho(x|z)$ is a non-degenerate Gaussian (and thus, absolutely continuous) has positive measure. Lemma 3.2.1 of \cite{pass2013optimal} implies that for each such $\rho(x|z)$, the optimal transport cost $\mu\mapsto W_2^2(\mu,\rho(x|z))$ is strictly convex. Then, the total transport cost $\mu \mapsto \int W_2^2(\mu,\rho(x|z))dv$ is strictly convex and the barycenter is unique.

\section{Saddle point algorithms}
\label{appendix: saddle point algorithms}

Let $\inf_{\tau} \sup_{\xi} L(\tau,\xi)$ be a min-max problem. For convenience, denote
\begin{equation*}
J = 
\begin{pmatrix}
I_{dim \tau} & \\
 & -I_{dim \xi}
\end{pmatrix},
~ w = \begin{bmatrix} \tau\\ \xi \end{bmatrix}
\end{equation*}

Then, the OMD algorithm \cite{mertikopoulos2019optimistic} (using Euclidean squared distance) becomes,

\medskip

\begin{algorithm}[H]
\textbf{Parameters: }{Learning rates $\eta^n$.}\\
\For{$n \gets 1,2,\dots$}{
    Compute the waiting state $\tilde{w} \gets w^n - \eta^n J \nabla L (w^n)$\\
    Actual update $w^{n+1} \gets w^n - \eta^n J \nabla L (\tilde{w})$
}
\Return $w^{\infty}$
\caption{Optimistic mirror descent}
\label{OMD algorithm}
\end{algorithm}

\medskip

\noindent
We present the QITD algorithm \cite{essid2019implicit} in its data-based setting, such that $L$ is estimated from samples, just as in (\ref{data-based barycenter objective}),

\medskip
\begin{algorithm}[H]
\textbf{Parameters: }{Iteration number $T$. Batch size $M$. Initial learning rate $\eta^0$. Decay rate $\gamma \in (0,1)$. Stopping threshold $\epsilon \ll 1$. Increase factor $\beta >0$. Maximum learning rate $\eta_{\max}$.}\\
\textbf{Data: }{Sample $X = \{x_i\}_{i=1}^N$.}\\
Initialize quasi Newton matrix $B^1 \gets J$.\\
\For{$n \gets 1$ \KwTo $T$}{
Randomly sample a minibatch $X^n$ of size $M$.\\
Compute gradient $g^n \gets \nabla L(w^n ~| X^n)$.\\
Initialize learning rate $\eta^n \gets \eta^{n-1}$.\\
Compute update $w^{n+1} \gets w^n - \eta^n B^n g^n$.\\
\While{$\eta^n > \epsilon \eta^{n-1}$ \text{\upshape and the anticipatory constraint}
\begin{equation}
\label{anticipatory constraint}
L(\tau^{n+1}, \xi^n ~| X^n) \leq L(\tau^{n+1}, \xi^{n+1} ~| X^n) \leq L(\tau^n, \xi^{n+1} ~| X^n)
\end{equation}
\text{\upshape is not satisfied}}{
Line search $\eta^n \gets \gamma \eta^n$.\\
Update $w^{n+1} \gets w^n - \eta^n B^n g^n$.
} 
\If{\text{\upshape constraint (\ref{anticipatory constraint}) has been satisfied}}{
Increase learning rate $\eta^n \gets \max( (1+\beta) \eta^n, \eta_{\max} )$
}
Compute new gradient $g^{n+1} \gets \nabla L(w^{n+1} ~| X^n)$.\\
Rank-one update of $B^n$:
\begin{align*}
s &\gets J g^{n+1} - B^n g^n\\
\alpha &\gets \frac{\|s\|^2}{\lb g^n, s \rb}\\
\alpha &\gets sign(\alpha) \min\big(|\alpha|,1\big)\\
B^{n+1} &\gets B^n + \alpha \frac{s \cdot s^T}{\|s\|^2}
\end{align*}
}
\Return $w^T$
\caption{Stochastic quasi implicit twisted descent}
\label{stochastic quasi implicit}
\end{algorithm}

The algorithm has time complexity $O(TD^2)$, where $D = dim \tau + \dim \xi$. With fixed decay rate $\gamma$ and threshold $\epsilon$, the anticipatory constraint (\ref{anticipatory constraint}) contributes a multiplicative factor to running time. The quasi Newton matrix $B^n$ can be replaced by a list of its rank-one updates: $\{s^n,\alpha^n\}_{n=1}^T$, changing the time complexity to $O(T^2D)$.

\section{Implementation details}
\label{appendix: implementation}

All BaryNet models implemented in Section 4 are based on neural networks, and the network architectures are summarized below:

\begin{table}[h!]
\small
\centering
\begin{tabular}{||c | c c c c ||} 
 \hline
 Experiment & $z(x)$ & Transport residual & $\psi^Y$ & $\psi^Z$\\
 \hline\hline
 \S \ref{sec: artificial test} & - & $3 \to 7 \to 7 \to 2$ & $2\to 6\to 6\to 1$ & $1\to 5\to 1$\\
 \S \ref{sec: seismic test} climate & $56\to 5\to 5\to 1$ & $57\to 2\to 56$ & $56\to 14\to 1$  & $1\to 5\to 1$ \\
 \S \ref{sec: seismic test} seismic & $2\to 6\to 6\to 1$ & $3\to 7\to 7\to 7\to 2$ & $2\to 5\to 5\to 5\to 1$  & $1\to 5\to 5\to 1$ \\
 \S \ref{sec: climate test} & - & $5\to 9\to 9\to 1$ & $1\to 5\to 5\to 1$ & $4\to 8\to 1$ \\
 \S \ref{sec: color transfer} & - & $3 \to 25\to 3 \to 25 \to 3$ & $3 \to 25\to 3 \to 25 \to 1$  & $\R^3$ \\
 \hline
\end{tabular}
\caption{BaryNet architectures for experiments in Section \ref{sec: test results}. ``Transport residual" refers to the residual part of the transport map (See Section \ref{sec: transport net}), and the inverse transport maps have the same architecture. Unless specified in the comments below, all activation functions are ReLU.}
\label{table: BaryNet architecture}
\end{table}

The training schemes are based on either OMD or QITD, and their parameters are summarized below:

\begin{table}[h!]
\small
\centering
\begin{tabular}{||c | c c c c c c c c ||} 
 \hline
 Experiment & Algorithm & $T$ & $M$ & $\mu^0$ & $\gamma$ & $\epsilon$ & $\beta$ & $\eta_{\max}$ \\
 \hline\hline
 \S \ref{sec: artificial test} & QITD & $10^4$ & full & $4\times 10^{-3}$ & $0.75$ & $10^{-3}$ & $0.1$ & $2\times 10^{-2}$ \\
 \S \ref{sec: seismic test} climate & OMD & $5\times 10^4$ & $10^3$ & $10^{-5}$ & - & - & - & - \\
 \S \ref{sec: seismic test} seismic & OMD & $2\times 10^4$ & full & $10^{-3}$ & - & - & - & - \\
 \S \ref{sec: climate test} & QITD & $10^4$ & $2400$ & $2\times 10^{-3}$ & $0.75$ & $10^{-3}$ & $0.1$ & $10^{-2}$ \\
 \S \ref{sec: color transfer} & OMD & $10^6$ & $10^3$ & $10^{-3}$ & - & - & - & - \\
 \hline
\end{tabular}
\caption{Training parameters. The notations are based on Appendix \ref{appendix: saddle point algorithms}: For QITD, we have iteration number $T$, batch size $M$, initial learning rate $\eta^0$, decay rate $\gamma$, stopping threshold $\epsilon$, increase factor $\beta$, and maximum learning rate $\eta_{\max}$. For OMD, the learning rate is fixed $\mu \equiv \mu^0$. If $M=$full, then gradient descent is based on the population loss over the entire sample, instead of the minibatch loss.}
\label{table: training parameters}
\end{table}

The training of the inverse transport map (\ref{inverse transport by regression}) are based on either SGD or Adam:

\begin{table}[h!]
\small
\centering
\begin{tabular}{||c | c c c c ||} 
 \hline
 Experiment & Algorithm & $T$ & $M$ & $\mu$ \\
 \hline\hline
 \S \ref{sec: artificial test} & SGD & $2\times 10^4$ & full & $5\times 10^{-2}$ \\
 \S \ref{sec: climate test} & SGD & $10^4$ & $24 \times 365$ & $5\times 10^{-3}$ \\
 \S \ref{sec: color transfer} & Adam & $10^4$ & $3000$ & $10^{-3}$ \\
 \hline
\end{tabular}
\caption{Training parameters: iteration number $T$, batch size $M$, and learning rate $\eta$. We did not use weight decay or momentum.}
\label{table: training the inverse}
\end{table}

For table \ref{table: BaryNet architecture}, if the latent variable $z$ is continuous, then the transport map has the form $T(x,z):\R^{d+k}\to\R^d$. If $z$ is discrete, then we model each $T_z(x): \R^d \to \R^d$ by networks with the same architecture. Regarding $\psi^Z$ and $z(x)$, the affine map at the last layer is always bias-free, in order to reduce unnecessary parameters. The reason is that any additive constant in $\psi^Z$ would be removed by (\ref{integral constraint for psi Z}), while the latent variables $z$ are invariant under permutation so a constant has no effect on the assignment $z(x)$ (Section \ref{sec: label net})\\

Here are the details of each experiment:
\begin{enumerate}
\item [\S \ref{sec: seismic test} climate] The daily temperature data is taken from NOAA \cite{NOAA2019daily}. We chose the 56 stations with the fewest missing values, and any $x_i$ with missing entry is discarded.

Training in high dimensions becomes more unstable, so we follow the technique introduced by the DCGAN paper \cite{radford2015unsupervised}: we applied batch normalization to all hidden layers of all networks and also applied leaky ReLU (with negative slope $0.1$) to the test functions $\psi^Y,\psi^Z$ and labeling function $z(x)$.

We have enforced Lipschitz continuity in $z(x)$, by clamping its parameters to be within $[-0.1,0.1]$ after each update step.

\item [\S \ref{sec: seismic test} seismic] The earthquake data is taken from \cite{USGS2008earthquake} and scaled by $\pi/180$ to yield spherical coordinates. We applied batch normalization to all hidden layers of $T(x,z),\psi^Y,\psi^Z,z(x)$ and also applied leaky ReLU (with negative slope $0.1$) to $\psi^Y,\psi^Z$ and $z(x)$.

\item [\S \ref{sec: climate test}] The hourly temperature data is taken from \cite{NOAA2019hourly}. We chose Ithaca, NY and the time range Jan 1 2007 to Dec 31 2016 because there are few missing values, which are filled by linear interpolation.

\item [\S \ref{sec: color transfer}] Each of the RGB color channels has range $[0,1]$. Theoretically, the transported distributions $S_j\circ T_k \# \rho_k$ are always supported in $[0,1]^3$. Nevertheless, the computed result is only an approximation to the true distribution, and sometimes a few points are mapped outside of $[0,1]^3$, so we project them back.

\end{enumerate}


\end{document}